\documentclass{article}
\usepackage[utf8]{inputenc}

\usepackage{hyperref}

\usepackage[document]{ragged2e}

\usepackage{amsmath}
\usepackage{amssymb}
\usepackage{amsthm}
\usepackage{mathrsfs}
\usepackage{mathtools}
\usepackage{comment}
\usepackage{enumerate}
\usepackage{xcolor}
\usepackage{slashed}

\usepackage[english]{babel}

\usepackage{cite}
\usepackage{url}
\usepackage{hyperref}

\usepackage{geometry}
\numberwithin{equation}{section}

\usepackage{amscd}
\usepackage{tikz-cd}
\usepackage[title]{appendix}

\makeatletter
\newtheorem*{rep@theorem}{\rep@title}
\newcommand{\newreptheorem}[2]{%
\newenvironment{rep#1}[1]{%
 \def\rep@title{#2 \ref{##1}}%
 \begin{rep@theorem}}%
 {\end{rep@theorem}}}
\makeatother

\newtheorem{theorem}{Theorem}[section]
\newreptheorem{theorem}{Theorem}

\newtheorem{proposition}[theorem]{Proposition}
\newtheorem{lemma}[theorem]{Lemma}
\newtheorem{corollary}[theorem]{Corollary}

\newtheorem{remark}[theorem]{Remark}
\newtheorem{example}[theorem]{Example}
\theoremstyle{definition}
\newtheorem{definition}[theorem]{Definition}

\newcommand{\Ham}{\text{Hm}}

\usepackage{subfiles} 

\usepackage[nottoc]{tocbibind}

\title{Supersymmetric harmonic oscillators on singular geometries}
\author{Gayana Jayasinghe}
\date{}

\begin{document}

\maketitle

\begin{abstract}
\justify 
Equivariant localization expresses global invariants in terms of local invariants, and many of them appearing in equivariant index theory, (holomorphic) Morse theory, geometric quantization and supersymmetric localization can be characterized as renormalized supertraces over cohomology groups of Hilbert complexes associated to local model geometries.
This paper extends such local invariants, introducing and studying twisted de Rham and Dolbeault complexes (including Witten deformed versions) on singular spaces equipped with generalized radial (K\"ahler Hamiltonian) Morse functions and singular metrics arising naturally in algebraic geometry and moduli problems. 

We use the $\mathcal{N}=2$ supersymmetry and nilpotency properties of these complexes to extend an ansatz of Cheeger for the eigensections of the associated Laplace/Schr\"odinger type operators, reducing the problem to the study of Sturm-Liouville operators and one dimensional Schr\"odinger operators corresponding to different choices of domains, including those with del-bar Neumann boundary conditions for Dolbeault complexes.
We define renormalized Lefschetz numbers and Morse polynomials generalizing those established in the smooth and conic settings where they have been used to compute 
many invariants of interest in physics and mathematics. We study structures on links of topological cones with singular K\"ahler metrics, which we use to describe associated analytic invariants including local cohomology groups.
The techniques and results collected here are broadly applicable in the study of global analysis on singular spaces, including proofs of localization theorems with numerous applications.
\end{abstract}

\tableofcontents

\justify

\section{Introduction}

Equivariant localization is an umbrella term used to describe a broad collection of results that are used to compute geometric and topological invariants of various spaces in the presence of symmetries. Roughly such results show that global invariants on a space are completely determined by local invariants at fixed points of symmetries, easier to define and compute at isolated fixed points.
Understanding such results is an ongoing program, not only on finite dimensional spaces (both smooth and singular), but also on various infinite dimensional spaces where results are yet to be established rigorously in many cases. 
There are some applications in physical mathematics where first a non-rigorous version of localization on an infinite dimensional space is used to reduce expressions for invariants to those on finite dimensional ones (sometimes singular) and then a second rigorous version is used on the finite dimensional spaces to express those invariants in terms of local invariants at fixed point sets (see \cite{gibbons1979classification,NekrasovABCDinsta,pestun2012localization}). Understanding these local equivariant invariants is of great importance in the development of mathematics and physics.

The Atiyah-Bott-Lefschetz fixed point theorem for elliptic (Hilbert) complexes of differential operators on smooth manifolds (\cite{AtiyahBott1,AtiyahBott2}) formalizes the fact that local equivariant invariants corresponding to global supertraces of geometric endomorphisms (corresponding to symmetries) on cohomology of those complexes can be expressed in terms of linear algebraic formulas on the tangent space at fixed points together with information from local traces of geometric endomorphisms on related bundles (which twist the coefficients of the differential operators) over isolated fixed point sets. Alternate descriptions of these local formulas in terms of equivariant eta invariants (\cite{weiping1990note,jayasinghe2023l2}), and supertraces over cohomology groups of \textbf{local Hilbert complexes} when the operators are compatible with (almost) complex structures (\cite{jayasinghe2023l2}) admit generalizations to singular spaces. 
While equivariant indices and Lefschetz numbers are usually categorified at the level of equivariant K-theory (\cite{Baumformula81}), (holomorphic) Morse polynomials participating in holomorphic Morse inequalities are categorified by (holomorphic) Witten instanton complexes (\cite{witten1982supersymmetry,witten1984holomorphic,wu2003instanton,jayasinghe2024holomorphic}) when there is $\mathcal{N}=2$ supersymmetry (abbreviated as SUSY), (c.f. \cite[\S 3]{witten1982constraints}).

The de Rham Witten instanton complexes have been extended to a large class of metrics that appear naturally on stratified pseudomanifolds, with the work in \cite{Jesus_Dunkl_2014,Jesus_Dunkl_2015,Jesus2018Wittens,Jesus2018Wittensgeneral,ludwig2013analytic,ludwig2017comparison,ludwig2017index}, establishing important analytic results, including spectral theory of associated Laplace-type operators associated to non-compact model geometries appearing at critical points of Morse functions at singularities. In \cite{Bei_2012_L2atiyahbottlefschetz}, the local Lefschetz numbers for the de Rham complex were expressed near singular fixed points of self-maps in terms of finite dimensional local cohomology groups. 

One characterization of this article is that it extends the local invariants described above from de Rham complexes to more general Hilbert complexes with $\mathcal{N}=2$ SUSY, on singular models with group actions, and with \textit{radial Morse functions} that we introduce. It was motivated by noticing that some of the analytic results established in \cite{Jesus_Dunkl_2014,Jesus_Dunkl_2015,Jesus2018Wittens,Jesus2018Wittensgeneral} in the study of the de Rham complexes could be used to study the general case. In \cite{jayasinghe2023l2}, an ansatz introduced by Cheeger in \cite{cheeger1983spectral}, in the study of the Hodge Laplacian on non-compact cones was generalized to Dolbeault complexes (with Witten deformation in \cite{jayasinghe2024holomorphic}). Moreover the local Hilbert complexes associated to truncated model geometries (tangent cones) at fixed points of self-maps and critical points of (K\"ahler Hamiltonian) Morse functions were studied, showing that the associated local Lefschetz numbers/Morse polynomials could be used to understand the global versions, and the \textit{same formulations seem to go through for more general spaces including those studied here}. The many examples and applications studied in \cite{jayasinghe2024witten,jayasinghe2023l2} indicate important connections to various fields. The work here establishes rigorous foundations for the study of SUSY Quantum Mechanics (QM), geometric quantization and equivariant localization on stratified pseudomanifolds and is a stepping stone on the way to extending better known results in those topics to singular spaces.

We develop intricate arguments using the $\mathcal{N}=2$ to understand the cohomology and the spectrum of the Laplace-type operators associated to these complexes using Cheegers ideas mentioned above, involving the study of 1-dimensional Schr\"odinger operators and Sturm-Liouville operators, which then allows the formulation of renormalized trace formulas for the non-Fredholm Hilbert complexes. We study a few examples at the end, referring to \cite{jayasinghe2023l2,jayasinghe2024holomorphic} for a broader collection of examples and applications in the conic setting, that can be extended more generally with the results here.

It is well known that the link (or cross-section) of a K\"ahler cone has what is known as a \textbf{Sasaki structure}, which appears in many studies across geometry and physics. In our study of analysis with K\"ahler structures on spaces with singular warped product metrics we introduce many notions including \textbf{Reeb type metrics, Reeb rescaling functions} which generalize those encountered on K\"ahler cones, and we relate them to holomorphic functions, local cohomology groups and other analytic objects on such spaces.

\subsection{Geometric setting and main and results}

In this article we study Hilbert complexes, specifically twisted de Rham and Dolbeault complexes with $\mathcal{N}=2$ SUSY on topological cones $\widehat{X}=C(Z)$ that are local models at singular fixed points, equipped with \textbf{singular warped product metric} (abbreviated as \textbf{swp} metrics)
\begin{equation}
\label{warped_product_metric}    
    dx^2+\sum_j f^2_j(x) g_{Z_j}
\end{equation}
where $f_j(x)$ are a smooth positive functions on $(0,\infty)$ that vanish at $x=0$, and where $g_{Z_j}$ is some smooth Riemannian metric on $Z_j$. We encode the symmetries on these spaces in terms of self-maps that give rise to geometric endomorphisms on these spaces, as well as via \textit{radial Morse functions} (see Definition \ref{Definition_radial_Morse_functions}) to which we can associate Witten deformed Hilbert complexes $\mathcal{P}_{W,B,\varepsilon}=(H, \mathcal{D}(P_\varepsilon),P_\varepsilon)$,
\begin{equation}
    0 \rightarrow \mathcal{D}_{W,B}(P_\varepsilon) \xrightarrow{P_\varepsilon} \mathcal{D}_{W,B} \xrightarrow{P_\varepsilon} \mathcal{D}_{W,B}(P_\varepsilon) \xrightarrow{P_\varepsilon} ... \xrightarrow{P_\varepsilon} \mathcal{D}_{W,B}(P_\varepsilon) \rightarrow 0
\end{equation}
where $\varepsilon=0$ is the undeformed case. 
In the undeformed case, we study twisted de Rham complexes 
$\mathcal{P}_{W,B}$
where $P=d_E$ where $E$ is a flat bundle on $X$, $W$ ($W=\max/\min$) denotes ideal boundary conditions at $x=0$ and $B$ (either $N/D$, corresponding to generalized Neumann/Dirichlet conditions) denotes boundary conditions at $x=1$ which fix an extension $\mathcal{D}_{W,B,\varepsilon}(d_E)$ of the de Rham operator $d_E$ acting on smooth forms with coefficients in $E$ and compact support in the interior of ${X}$, and together with the adjoint operator it determines a self-adjoint extension of the Hodge de Rham Dirac operator $D=d_E+d^*_E$. If there is an adapted complex structure on $X$, we define twisted Dolbeault complexes
$\mathcal{P}_{W,B}$ with $P=\overline{\partial}_E$ where $E$ is an adapted Hermitian vector bundle, $W$ denotes ideal boundary conditions at $x=0$ and $B$ denotes boundary conditions at $x=1$ which fix an extension of the Dolbeault operator, determining a self-adjoint extension of the Dolbeault-Dirac operator $D_E=\overline{\partial}_E + \overline{\partial}^*_E$. In this case, the generalized Neumann condition is the del-bar Neumann condition.

For the radial Morse functions in Definition \ref{Definition_radial_Morse_functions} (even without a K\"ahler structure, for the de Rham case) we denote the deformed complexes as $\mathcal{P}_{W,B,\varepsilon}$ where $\varepsilon$ is a semi-classical parameter (akin to the inverse temperature and the \textit{semi-classical limit is $\varepsilon \rightarrow \infty$}), and denote the deformed de Rham/ Dolbeault operators by $P_{\varepsilon}:=e^{-\varepsilon h} P e^{\varepsilon h}$. The corresponding Laplace-type operators $\Delta_{\varepsilon}=(P_{\varepsilon} + P^*_{\varepsilon})^2$ can be expanded as
\begin{equation}
\label{equation_deformed_Laplace_operator_intro_1}
    \Delta_{\varepsilon}= \Delta +\varepsilon Kh'' + \varepsilon^2 |dh|^2, \quad \Box_{\varepsilon}=\frac{1}{2} \Delta_{\varepsilon} + \varepsilon \sqrt{-1} L_V
\end{equation}
for the de Rham and Dolbeault cases respectively, where $K=\pm 1$ on the model spaces we study depending on the types of forms (which we will clarify later).  For the Dolbeault case we only do this if the swp metric on $X$ is adapted K\"ahler, and $X$ is equipped with a K\"ahler Hamiltonian vector field $V$ generating an $S^1$ action and satisfying $\iota_V \omega=dh$ where $\omega$ is the adapted K\"ahler form. Then the radial Morse function $h$ is the called the \textbf{Hamiltonian function}. We show in Subsection \ref{Subsection_adapted_geometric_structures} that if $J(\partial_x)=V/f_0(x)$ where $J$ is the corresponding adapted complex structure, then the Hamiltonian $h(x)$ is determined by the \textbf{Reeb rescaling function} $h'(x)=f_0(x)$ up to an additive constant (see Definition \ref{Definition_Reeb_type_notions}). If the K\"ahler action lifts to one on $E$, then we denote by $\sqrt{-1} L_V$ the infinitesimal action of the Hamiltonian vector field $V$ on sections of $E$ (in the case of the trivial bundle $L_V$ is the Lie derivative). Since $L_V$ commutes with $\Delta_{\varepsilon}$, it suffices to study the eigensections of the operator $\Delta_{\varepsilon}$ to understand the eigensections of $\Box_{\varepsilon}$.

We treat the de Rham and Dolbeault complexes on the same footing as far as possible, as Hilbert complexes $\mathcal{P}_{W,B,\varepsilon}=(H,\mathcal{D}_{W,b,\varepsilon}(P_\varepsilon), P_\varepsilon)$ where the differentials $P_\varepsilon$ satisfy an \textit\textbf{{independence from the metric}} property described in Remark \ref{Remark_independence_properties}, and where $D=P+P^*$ is a Dirac-type operator. These complexes are $\mathcal{N}=2$ SUSY QM systems, where the SUSY operators are $Q_1=P+P^*, Q_2= \sqrt{-1}(P-P^*)$, where the corresponding nilpotent operators are $Q_-=P$, $Q_+=P^*$ up to constants. In this setting, the corresponding Laplace-type operators $\Delta=Q_1^2=Q_2^2$ give an isomorphism between exact/co-exact eigensections (those in the null space of $P$ and $P^*$ respectively) of $\Delta$, refining an isomorphism between odd and even degree eigensections when there is only $\mathcal{N}=1$ SUSY corresponding to a Dirac type operator which is interpreted as an isomorphism between bosonic and fermionic states through the lens of supersymmetry. We discuss this in more detail in Section \ref{Section_Hilbert_complexes}.

The spectrum of these operators can be studied using a separation of variables ansatz corresponding to the $[0,1]_x$ and $Z$ factors of $X$, exploiting the radial symmetry of the model space. We consider sections $u(x)\phi$,
where the sections $\phi=\phi_{(\mu_1,...,\mu_m)}$ satisfy
\begin{equation}
    D_{Z_j}\phi_{(\mu_1,...,\mu_m)}=\mu_j \phi_{(\mu_1,...,\mu_m)}
\end{equation}
for each $j$, and are eigensections of each $\Delta_{Z_j}$ with eigenvalue $\mu^2_j$.

\begin{remark}
\label{Remark_intro_SL_organization}
We consider the singular SL equations
\begin{equation}
\label{equation_Sturm_Liouville_2}
    L_{\varepsilon}[u]=L[u]+ K\varepsilon h''(x) u+h'(x)^2 \varepsilon^2 u =\lambda^2 u, \quad  L[u](x)=-\frac{1}{F(\phi)}(F(\phi) u')'+\sum_j \frac{\mu_j^2}{f_j^2}u
\end{equation}
where $K=\pm 1$, and where $F(\phi)$ are the \textit{adjoint rescaling functions} that we will introduce in \eqref{equation_Rescaling_operator_2_adjoint_rescaling_function}. These are functions that depend on the swp metric and the \textit{multi-degree} of forms $\phi$ used in the ansatz described above. Here the $\mu^2_{j}$ are the eigenvalues of $\Delta_{Z_j}$ on the forms $\phi$, and $\mu^2=\sum_{j} \mu^2_{j}$ are the eigenvalues of $\Delta_Z$ appearing in Theorem \ref{Theorem_main_spectral_intro_version}. 
\end{remark}

If the SL problems with the domains induced by those for the Laplacians on $X$ have discrete spectrum, then we say that the singular warped product metric is of \textit{\textbf{discrete type}} (see Definition \ref{Definition_metrics_discrete_type}).
Our main result is the following.

\begin{theorem}
\label{Theorem_main_spectral_intro_version}
Given a twisted de Rham/Dolbeault complex (with Witten deformation for $\varepsilon>0$), $\mathcal{P}_{W,B,\varepsilon}=(H, \mathcal{D}(P_\varepsilon),P_\varepsilon)
$ on $X$ with a swp metric of discrete type for the case where $\varepsilon=0$, of a Morse function of discrete type for the case when $\varepsilon>0$ (and is K\"ahler Hamiltonian in the case of a twisted Dolbeault complex). We denote by $\Delta_\varepsilon$ the Witten deformed Hodge Laplacian, where $\varepsilon=0$ corresponds to the undeformed complex. Then for any such complex for any fixed $\varepsilon$
\begin{enumerate}
    \item there exists an orthonormal basis $\{\psi_{n,k}\}_{n \in \mathbb{N}, k \in \mathbb{N}}$ where each $\psi_{n,k}=u_{n,k} \phi_{n}$ is an eigensection of $\Delta_{\varepsilon}$ (for each degree $q$) where $u_{n,k}$ is a form on $(0,1)_x$ and $\phi_{n}$ is an eigensection of $\Delta_Z$ with eigenvalue $\mu^2_n$ as described in Remark \ref{Remark_intro_SL_organization},
    \item the sections $\psi_{n,k}$ are eigensections of $\Delta_{\varepsilon} + \Delta_Z$ (in the same domain as that for $\Delta_{\varepsilon}$) and this operator has discrete spectrum, with eigenvalues $\lambda_{n,k}^2 + \mu^2_{n}$, where $\lambda_{n,k}^2$ are eigenvalues of the operator $\Delta_{\varepsilon}$,
    \item for each $P_\varepsilon$ co-exact eigensection $\psi_{n,k}$ with positive eigenvalue $\lambda^2_{n,k}$ for $\Delta_{\varepsilon}$, $\frac{1}{\lambda_{n,k}} P_\varepsilon\psi_{n,k}$ is an exact eigensection with the same eigenvalues for $\Delta_{\varepsilon}^{q}$ and $\Delta_Z$.
\end{enumerate}
\end{theorem}

We delve into details about the cohomology of the complex in Subsection \ref{subsection_spec_theory_and_cohomology}.
Given a self-map $f$ of $X$ that is generated by the flow of a K\"ahler Hamiltonian isometry, the pullback $f^*$ induces an endomorphism $T$ on these Hilbert complexes, and we call them geometric endomorphisms. We define \textbf{polynomial Lefschetz heat supertrace functions}
\begin{equation}
    \mathcal{L}(\mathcal{P},T,t,s)(b)
    :=\sum_{q=0}^n b^q Tr \Big(T \circ e^{-t\Delta}e^{-s(\Delta+\sqrt{\Delta_Z})} \Big)|_{\mathcal{P}}
\end{equation}
and \textbf{polynomial Lefschetz supertrace functions}
\begin{equation}
    L(\mathcal{P},T,s)(b)
    :=\sum_{q=0}^n b^q Tr \Big(T\circ e^{-s\sqrt{\Delta_Z}} \Big)|_{\mathcal{H}^q(\mathcal{P})}
\end{equation}
analogous to \cite{jayasinghe2023l2} in the wedge setting.
In Theorem \ref{Theorem_Morse_supertrace} we prove a \textbf{renormalized McKean Singer theorem} and show that $\mathcal{L}(\mathcal{P},T,t,s)(-1)=L(\mathcal{P},T,s)(-1)$.
We denote by $\mathbb{Z}_{\geq 0}[\lambda,\lambda^{-1}]$ the \textit{ring of Laurent series in $\lambda$ with non-negative integer coefficients}. %
The polynomial Lefschetz supertrace function is also called the (local) polynomial Morse supertrace function of the Witten deformed twisted Dolbeault complexes we study, when the geometric endomorphism arises from a K\"ahler Hamiltonian isometry.
Given a $S^1_{\theta}$ group action, we will denote the induced geometric endomorphism as $T_{\theta}$, and we will continue to use this notation even when the action extends to a $\mathbb{C}^*$ action.
The following is a restatement of Corollary \ref{Corollary_Morse_supertraces}.

\begin{corollary}
\label{Corollary_Morse_supertraces_intro}
Given a swp metric of discrete type on $X$ that is adapted K\"ahler and is of Reeb type, 
where there is a K\"ahler Hamiltonian Morse function, the polynomial Morse supertrace function is a finite linear combination 
\begin{equation}
\label{equation_Laurent_series_combinations_Morse_intro}
    M(\mathcal{P}_{W,B},T_\theta,s)(b)= \sum_{q=0}^n b^q \sum_{k\in I'} C_{k,\theta} u_{k, s,\theta}
\end{equation}
where $u_{k, s,\theta} \in \mathbb{Z}_{\geq 0}[\lambda,\lambda^{-1}]$, where $\lambda=e^{-s+i\theta}$, $I'$ is a finite indexing set and where each $C_{k,\theta}$ is a real power of $\lambda$.
Additionally we have the following.
\begin{enumerate}
\item If both of the analytic functions (in the variable $s$), $M(\mathcal{P}_{W,B},T_\theta,s)(b)$ and $L(\mathcal{P}_{W,B},T_\theta,s)$ are regular at $s=0$, then
\begin{equation}
    M(\mathcal{P}_{W,B},T_\theta,0)(-1)=L(\mathcal{P}_{W,B},T_\theta,0).
\end{equation}
\item If the Reeb rescaling function $f_0(x)$ is bounding from below (respectively above), 
then $M(\mathcal{P}_{W,N},T_\theta,s)(b)$ converges for $s>0$ (respectively $s<0$), and $M(\mathcal{P}_{W,D},T_\theta,s)(b)$ converges for $s<0$ (respectively $s>0$).
\end{enumerate}
\end{corollary}

\subsubsection{Metrics which are of discrete type}
\label{subsection_prior_results}

Since our results hold for metrics of discrete type, we present some examples of such metrics.
Conic metrics are well known to be of discrete type, and includes smooth cones (cones over spheres with Euclidean volume growth), and the de Rham and Signature complexes were studied by Cheeger \cite{cheeger1983spectral}. The case of horn metrics, for $c_j>1$, has been studied in \cite{cheeger1980hodge,bruning1996signature}.

We show in Section \ref{section_Sturm_Liouville} that any swp metrics where each $f_j(x)=x^{c_j}$ for $c_j>0$ is of discrete type.
This is the form of the metric on a large class of tangent cones (in the sense of Gromov-Hausdorff convergence) at singularities.
Such metrics were already studied in \cite{Jesus2018Wittensgeneral} when $c_j<1$, where the Witten deformed de Rham case was studied, and we show how this suffices for the case of the Dolbeault complex, and thus for all Dirac-type operators when there is a compatible complex structure. 

The Weyl-Peterson metric on the Riemann moduli space is an iterated horn metric (see equation (1) of \cite{MazzeoVasy2014SpectralWeilPeterson}), and on the Teichm\"uller space admits K\"ahler actions where the model action is of the type that we study here (see \cite[\S 7.2]{AlexWright_Mirzakhani_2020}, equation (4.1) in \cite{Dean_Jesse_Quantum_Ergodic_WP_2023} and the discussion there).
Orbifold versions of the Duistermaat-Heckman localization theorem were used by Mirzakhani (see \cite{AlexWright_Mirzakhani_2020}) to prove a conjecture of Witten on those spaces. 
Localization is used to prove various important results in physics, and this becomes more complicated on general stratified spaces when the operators involved are not essentially self-adjoint, including on algebraic varities. This was studied in the case of wedge metrics in \cite{jayasinghe2023l2,jayasinghe2024holomorphic} and the work in this article allows for expanding that work to more general stratified spaces.

\subsection{Organization of paper}

In Section \ref{Section_adapted_geometry_and_Dirac} we describe the basics of adapted geometry. Following the philosophy of Melrose, we define adapted tangent bundles on which the degenerate Riemannian metrics that we study can be realized as non-degenerate bundle metrics, and introduce adapted complex structures and radial Morse functions. We then define adapted Clifford modules and Dirac-type operators for which we study the polarizations given by nilpotent operators for the Hodge Dirac operator, and for any Dirac operator when there are compatible complex structures. We also introduce boundary conditions for these nilpotent operators. We describe the de Rham and Dolbeault operators in more detail before describing Witten deformed versions of them.

In Section \ref{Section_Hilbert_complexes}, we begin by describing some abstract Hilbert complexes, before relating it to SUSY QM systems. We show that the complexes we study here correspond to $\mathcal{N}=2$ SUSY systems, and describe the semi-classical perspective of Witten deformation. We then describe the domains for the operators that we study in this article.

In Section \ref{subsection_Laplacians} we construct the Laplace-type operators corresponding to the Dirac-type operators we study. We show how we can reduce to SL problems using a carefully constructed ansatz, in both the twisted de Rham and Dolbeault complexes that we study, and study how the boundary conditions simplify as well. We also describe the SL problems corresponding to the Schr\"odinger type operators arising in Witten deformation.

In Section \ref{section_Sturm_Liouville} we study the SL results that are eventually used to prove the main results of this article in Section \ref{Section_Supersymmetric_trace_formulas}.
In Subsection \ref{subsection_Illustrative_examples} we provide a few examples of the results we study, first on a smooth space before providing a singular example, and we refer the reader to references which include computational guides. We briefly go over the connection to equivariant eta invariants in an example.

\begin{justify}
 \textbf{Acknowledgements:} 
I thank Pierre Albin for discussions about this work and his comments. I thank Jacob Shapiro for his careful reading and comments on a previous draft. I thank J\'esus \'Alvarez-L\'opez, Hadrian Quan, Xinran Yu, Manousos Maridakis, Katrina Morgan and Mengxuan Yang for discussions on related topics which motivated some of the work here.
\end{justify}

\section{Adapted geometry and Dirac operators}
\label{Section_adapted_geometry_and_Dirac}

\subsection{Adapted geometry and (K\"ahler Hamiltonian) Morse functions}
\label{Subsection_adapted_geometric_structures}

In this article we study Hilbert complexes on topological cones $\widehat{X}=C(Z)$, a compactification of $(0,\infty)_x \times Z$ at $x=0$ obtained by adding a point at $x=0$, where the \textbf{link} of the cone is $Z=Z_1 \times Z_2 \times... \times Z_m$ where each $Z_j$ is a smooth manifold, and where $\widehat{X}$ is equipped with a \textbf{singular warped product metric} (abbreviated as \textbf{swp} metrics)
\begin{equation}
    dx^2+\sum_j f^2_j(x) g_{Z_j}
\end{equation}
where $f_j(x)$ are a smooth positive functions on $(0,\infty)$ that vanish at $x=0$, and where $g_{Z_j}$ is some smooth Riemannian metric on $Z_j$.

The space $\widehat{X}$ is singular in general (when it is smooth $Z$ is a sphere), and we study the resolution $X_{\infty}=[0,\infty)_x \times Z_1 \times Z_2 \times... \times Z_m$ where the metric is degenerate at $x=0$. 
While we mostly work with truncations $X:=[0,1)_x \times Z_1 \times Z_2 \times... \times Z_m$, we define metrics and Morse functions for $X_{\infty}$, and versions for $X$ can be obtained by restriction.
Given a swp metric we define the \textbf{adapted co-tangent bundle} as follows, the dual bundle of which (with respect to the swp metric) is called the adapted tangent bundle. 
\begin{definition}[Adapted co-tangent bundle]
\label{Definition_adapted_co_tangent_bundle}
Given a warped product metric of the form in \eqref{warped_product_metric} on $[0,1)_x \times Z$, we define the adapted cotangent bundle $^{a}TX$ as the bundle 
\begin{equation}
    dx \oplus f_1(x)T^*Z_1 \oplus f_2(x)T^*Z_2  \oplus ... \oplus f_m(x)T^*Z_m
\end{equation}
where $f_j(x)T^*Z_j$ for any $j \in \{1,...,m\}$ is defined as follows. Given any local co-frame $\{ \theta(j)_1,...,\theta(j)_{l_j} \}$ on a neighbourhood of $Z_j$, we define the adapted co-frame $\{ f_j(x)\theta(j)_1,...,f_j(x)\theta(j)_{l_j} \}$, which by the Serre-Swan Theorem defines a unique bundle on $Z_j$ which we define to be $f_j(x)T^*Z_j$.
\end{definition}

While the swp metrics are degenerate Riemannian metrics (i.e. degenerate bundle metrics of the usual tangent bundle on $X$), they are non-degenerate bundle metrics of the adapted tangent bundle corresponding to the swp metric.

Given an even dimensional space $X=[0,1] \times Z$ with a singular warped product metric and a corresponding adapted cotangent bundle, we can construct the corresponding exterior algebra, the sections of which we shall call adapted forms. If there is an adapted two form $\omega$ that is closed, anti-symmetric and is a non-degenerate as a bi-linear form, we call it an \textbf{adapted symplectic form}.
We call a section $J$ of the endomorphism bundle of the adapted cotangent bundle which satisfies $J^2=-Id$ as an \textbf{adapted almost complex structure}.
If $\nabla J=0$ where $\nabla$ is a Levi-Civita connection with respect to some adapted metric $g$, then we say that the almost complex structure is \textbf{integrable} ($J$ is an \textbf{adapted complex structure}) and the metric $g$ is Hermitian.
If for a given adapted metric $g$ and an adapted complex structure $J$, $\omega(\cdot, \cdot):=g(J \cdot, \cdot)$ is an adapted symplectic form, then we say that $X$ has an \textbf{adapted K\"ahler structure} $(g,J,\omega)$.

Witten investigated deformations of various elliptic complexes corresponding to Morse functions, and vector fields with various properties. In the smooth setting the Morse lemma shows that at any isolated stable critical point of a Morse function $h$, we can pick a disc neighbourhood $X=C_x(S^{k-1})$ with a coordinate chart where the Morse function can be written as $h(0)+x^2$. 
In the geometric setting of this article we will consider the following generalization.
\begin{definition}[\textbf{Radial Morse functions}]
\label{Definition_radial_Morse_functions}
Given a space $X$ with a swp metric $g$ with radial function $x$, we say that a smooth function $h(x)$ on $(0,\infty)$ is a radial Morse function if it satisfies
\begin{equation}
    (h'(x))^2 > h''(x)
\end{equation}
for all $x>C$ for some constant $C>0$, and if $\lim_{x \rightarrow \infty}(h'(x))^2=\infty$.
\end{definition}
This definition is motivated by the form of the potential in the SL problem in \eqref{equation_Sturm_Liouville_2} and the sufficiency of the growth of potentials at $\infty$ of the terms arising there for the SL problem to be in the limit point case at $x=\infty$. 
In the setting when there is an adapted complex structure $J$, we can demand that the swp metric is Hermitian (as a bundle metric on the adapted tangent bundle) and we define the following notions.

\begin{definition}
\label{Definition_Reeb_type_notions}
Given a Hermitian swp metric $g$ on $X$ with respect to an adapted complex structure $J$, if we can find a function $f_0(x)$ such that $J(\partial_x)=(1/f_0(x))V$ where $V$ is a vector field that satisfies $\nabla_{\partial_x} V=0$ with respect to the metric connection $\nabla$, we say that the metric $g$ is one of \textbf{Reeb type}. We call the function $f_0$ as the \textbf{Reeb rescaling function}. For such metrics we define the expression
\begin{equation}
    u_{\nu}(x):=\exp \Big(\int_0^x \frac{\nu}{f_0(y)} dy \Big)
\end{equation}
when the integral converges for $x \in (0,1)$ for any $\nu \neq 0$. We then consider $u_{\nu}$ to be a function of $x$ on $X$.
If $u_{\nu}(x)$ is $L^2$ bounded (with respect to the volume measure corresponding to $g$) on $X$, \textit{only} for integers $\nu > C$ (respectively $\nu <C$) for some $C \in \mathbb{R}$, we will say that the \textbf{Reeb rescaling function is bounding from below (respectively bounding from above)}.

In addition if the swp metric is K\"ahler with respect to the adapted complex structure, and if there exists a function $h(x)$ on $X_{\infty}$ such that $dh=\iota_V \omega$ for an adapted K\"ahler structure, then we say the function $h$ is a K\"ahler Hamiltonian, and call $V$ the K\"ahler Hamiltonian vector field. If in addition $V$ is Killing, we say it generates a K\"ahler Hamiltonian isometry. 
\end{definition}

The motivation for the definitions of bounding for the Reeb rescaling function are clarified in Proposition \ref{proposition_Dolbeault_cohomology} where we explore analytic aspects of these geometric notions in more detail.
In the case where $X$ has a K\"ahler Hamiltonian swp metric where we have that $J(\partial_x)=\frac{1}{f_0(x)}V$ for a vector field $V$ which is independent of $x$ (i.e. the Lie derivative of $V$ with respect to $\partial_x$ vanishes). Then the K\"ahler form can be written as 
\begin{equation}
    \omega_X=dx \wedge f_0(x) \alpha +f_{0,D}\omega_D
\end{equation}
where $\omega_D$ is a transverse K\"ahler structure on $Z$, that is transverse to the foliation generated by $V$, and where $\alpha=V^{\flat}$. Since $d\omega_X=0$, we see that $d\alpha=-\omega_D$ and $f_{0,D}'=f_0$.
Then the Hamiltonian condition for a function $h$ is equivalent to $h'(x)=f_0(x)$, i.e. $f_{0,D}=h$ (since we demand vanishing of both functions at $x=0$).
This generalizes the structures in the conic (including smooth) setting studied in \cite{jayasinghe2024holomorphic}.

Recall that the Hodge star operator $\star_X$ for a form is determined by $\omega \wedge (\star_X \eta)=\langle \omega,\eta\rangle dvol_X$. We study some examples to build intuition.

\begin{example}
\label{Example_complex_structure}
Consider the simple case of $C(\mathbb{T}^3)$ with the metric
\begin{equation}
    dx^2+x^6d\theta_1^2+x^4(d\theta_2^2+d\theta_3^2)
\end{equation}
which has an adapted almost complex structure 
\begin{equation}
    J=dx \otimes \frac{1}{x^3} \partial_{\theta_1} - x^3 d\theta_1 \otimes \partial_x + x^2 d\theta_2 \otimes \frac{1}{x^2}\partial_{\theta_3} - x^2 d\theta_3 \otimes \frac{1}{x^2}\partial_{\theta_2}
\end{equation}
which is integrable. Observe that $J$ maps
$\partial_x$ to $\frac{1}{x^3}\partial_{\theta_1}$, $\frac{1}{x^3}\partial_{\theta_1}$ to $\partial_x$, $\frac{1}{x^2}\partial_{\theta_2}$ to $\frac{1}{x^2}\partial_{\theta_3}$ and $\frac{1}{x^2}\partial_{\theta_3}$ to $-\frac{1}{x^2}\partial_{\theta_2}$.
There is an adapted form $\omega=dx \wedge x^3d\theta_1 + x^2 d\theta_2 \wedge x^2 d\theta_3$, and $\omega \wedge \omega$ is the volume form for the metric. It is easy to see that
$\star_X \omega =\omega$, i.e. $\omega$ is a self-dual 2-form. However it is easy to check that $d\omega \neq 0$ so $\omega$ is not a K\"ahler form.
\end{example}

\begin{example}
\label{Example_Kahler_horn}
The topological cone over the circle with the metric $dx^2+x^{2c}d\theta^2$ for $c>0$, with the symplectic form given by the volume form is K\"ahler due to dimensionality reasons, similar to the smooth setting. The volume form is $dx \wedge x^{c}d\theta$, and the associated complex structure is $J=dx \otimes \frac{1}{x^c}\partial_{\theta} -x^{c}d\theta \otimes \partial_x$. 

A K\"ahler Hamiltonian Morse function (that we define formally in the next subsection) is obtained by observing that $V=\partial_{\theta}$ generates a flow that preserves the metric and the K\"ahler structure, and is a K\"ahler Hamiltonian on $X^{reg}$ with Hamiltonian function $h=\frac{1}{c+1} x^{c+1}$. 
\end{example}

\subsection{Dirac-type operators}
\label{Subsection_Dirac_operators_detail}

We now introduce adapted Dirac-type operators on $X$.
\begin{definition}
\label{adapted_Clifford_module}
Let $g$ be an swp metric on $X$. An \textbf{\textit{adapted Clifford module}} consists of
\begin{enumerate}
    \item a complex vector bundle $E \rightarrow X$
    \item a Hermitian bundle metric $g_E$
    \item a connection $\nabla^E$ on $E$, compatible with $g_E$
    \item  a bundle homomorphism from the ‘adapted Clifford algebra’ into the endomorphism bundle of $E$,
    \begin{equation*}
        {cl} : \mathbb{C}l_{a}(X)=\mathbb{C} \otimes Cl(\prescript{{a}}{}{T}^*X, \prescript{{a}}{}{g}) \rightarrow End(E)
    \end{equation*}
    compatible with the metric and connection in that, for all $\theta \in \mathcal{C}^{\infty}(X;\prescript{{a}}{}{T}^*X)$,
    \begin{itemize}
        \item $g_E({cl}(\theta)\cdot, \cdot) = -g_E(\cdot, {cl}(\theta) \cdot)$
        \item $\nabla^E_{W} {cl}(\theta)= {cl}(\theta) \nabla^E_W + {cl}(\nabla^{X}_W \theta)$ as endomorphisms of $E$, for all $W \in \prescript{{a}}{}{T}X$, and where $\nabla^X$ is the connection on forms induced by the Levi-Civita connection corresponding to the swp metric $g$.
    \end{itemize}
\end{enumerate}

This information determines an \textit{\textbf{adapted Dirac-type operator}} by
\begin{equation}
\label{equation_wedge_Dirac_operator_11}
D_{X}:\mathcal{C}^{\infty}_c(\mathring{X};E) \xrightarrow{\nabla^E} \mathcal{C}^{\infty}_c(\mathring{X};T^*X \otimes E) \xrightarrow{{cl}} \mathcal{C}^{\infty}_c(\mathring{X};E) 
\end{equation}
where we have used that $T^*X$ and $\prescript{{a}}{}{T}^*X$ are canonically isomorphic over $\widehat{X}^{reg}=\mathring{X}$, the interior of $X$.
\end{definition}

Dirac-type operators are first order operators that square to generalized Laplace-type operators (those which have principal symbol given by the (adapted) metric). Proposition 3.38 of \cite{berline2003heat} shows that there is a Clifford algebra corresponding to any given Dirac operator (that squares to a generalized Laplace operator), and it is easy to check that this extends for the operators in the adapted setting.

We demand that $E$ admits a $\mathbb{Z}_2$ grading $E=E^+ \oplus E^-$, compatible in that it is orthogonal with respect to $g_E$, parallel with respect to $\nabla^E$, and odd with respect to ${cl}$. 
We can equivalently express this in terms of an involution on $E$.
\begin{definition}
\label{Grading operator}
A $\mathbb{Z}_2$ \textbf{\textit{grading operator}} on a Dirac bundle $E$ is an involution
\begin{equation*}
    \gamma \in C^{\infty}(\mathring{X}; End(E))
\end{equation*}
satisfying
\begin{equation*}
    \gamma^2=Id, \hspace{3mm} \gamma {cl}(v)= -{cl}(v) \gamma, \hspace{3mm} \nabla^E\gamma = 0, \hspace{3mm} \gamma^*=\gamma,
\end{equation*}
which gives a splitting of the bundle $E$ into the direct sum of the spaces $E^{\pm}=\{v \in E : \gamma(v)=\pm v \}$.
\end{definition}

We can always construct such a \textbf{grading/chirality/supersymmetry} operator on even dimensional spaces as follows.
Let $\{\theta^i\}_{i=1}^d$ be an oriented $g_{a}$-orthonormal frame for the adapted cotangent bundle $^{a}T^*X$.
Then the operator 
\begin{equation}
\label{equation_even_dimensional_spaces_chirality_operator}
    \gamma= i^p \, \Pi_{i=1}^d cl(\theta^i) \in C^{\infty}(\mathring{X}; End(E))
\end{equation}
where $p=d/2$ is a grading operator. This is a well known construction for Dirac bundles associated to Clifford modules on smooth manifolds (see \cite[Lemma 3.17]{berline2003heat}) and it is easy to see it extends to the setting we study.
In local coordinates, we can express the Dirac operator constructed in equation \eqref{equation_wedge_Dirac_operator_11} as, \begin{equation}
\label{equation_local_frame_Dirac_operator}
    D_{X}= \sum_{i=0}^n {cl}(\theta^i)\nabla^E_{(\theta^i)^{\#}}.
\end{equation}

For singular warped product metrics of the form in \eqref{warped_product_metric} that we study on $X$, given an orthonormal frame $\{e_j^1,...,e_j^n\}$ of the cotangent bundle on a chart on a factor $Z_j$ of the link $Z$, we can construct operators
\begin{equation}
\label{equation_Dirac_intermediary_1}
    {D}_{Z_j}=\sum_i cl(e_j^{i}) \nabla^E_{e_j^{i \#}}, \quad \widetilde{D}_{Z_j}=\sum_i cl(f_j(x)  e_j^i) \nabla^E_{e_j^{i \#}}
\end{equation}
and with such frames on all factors, we can construct an orthonormal frame (with respect to the swp metric)
\begin{equation}
    \{ dx\}, \{ \{f_j(x) (e_j^{i}) \}_{i=1}^{\dim Z_j} \}_{j=1}^m
\end{equation}
for the adapted cotangent bundle on $X$, and we can construct the Dirac-type operator
\begin{equation}
\label{equation_Dirac_intermediary_2}
    D_{X}=D_1+D_2, \quad \text{ where   }  D_1=cl(dx) \nabla^E_{\partial_x}, \quad D_2=\sum_{j=1}^{m} \frac{1}{f_j}\widetilde{D}_{Z_j}.
\end{equation}

In the case of de Rham and Dolbeault operators, the local form of these operators depend on the chosen coordinates and frame for the bundle, and whether we represent them as operators acting on the adapted forms (sections of the adapted exterior algebra), or on the forms (sections of the usual exterior algebra on $X$, where $X$ is considered as a manifold with boundary). In the next subsections, we will be studying the local forms of these operators represented as those acting on the usual forms on $X$. In particular we will work out $cl(dx)\nabla^E_{\partial_x}$ and its square in detail, convenient for understanding the Laplace-type operator.
We use the following choices of rescalings for forms in the analysis of these operators.

\begin{definition}[Multi-degrees and rescaling factors]
\label{Definition_rescaling_function}
In the setting described above, if a differential form (possibly with twisted coefficients) $\phi_j$ has degree $k_j$ in $Z_j$, we define the rescaling factor
\begin{equation}
    R(\phi_j)=f_j(x)^{k_j}.
\end{equation}

If a form $\phi$ on $Z$ can be written as a wedge product of forms $\phi_j$ on $Z_j$ where each form $\phi_j$ has fixed degree $k_j$, we say that the $\phi$ is of \textbf{fixed multi-degree} $\underline{k}=(k_1,...,k_m)$, and define the \textbf{rescaling factor} to be 
\begin{equation}
    R(\phi)=\Pi_{j=1}^m R(\phi_j) =\Pi_{j=1}^m f_j(x)^{k_j}.
\end{equation}
\end{definition}

As an example, the rescaling functions of $d\theta_1$, $d\theta_2$ and $d\theta_1 \wedge d\theta_2$ in Example \ref{Example_complex_structure} are $x^3, x^2$ and $x^5$ respectively.

\subsubsection{Polarizations, nilpotent refinements and boundary conditions}
\label{subsection_polarizations_nilpotent}

In this article we study Dirac operators which can be written as $D_X=P+P^*$ where $P$ is a nilpotent operator ($P^2=0$) and $(P^*)^2=0$, which refine the $\mathbb{Z}_2$ grading into a $\mathbb{Z}$ grading and leads to a twisted de Rham/Dolbeault complex. In these cases we can write
\begin{equation}
    P=\sum_i A_i V_i^\flat \wedge \nabla^E_{V_i}
\end{equation}
where $V_i$ is a set of orthogonal vector fields on $X$, and where $A_i$ are coefficients in a bundle which is flat/holomorphic. This is the case for twisted de Rham and Dolbeault operators, as we will see in the constructions we present in the next subsections.

\begin{remark}[Independence property]
\label{Remark_independence_properties}
It is well known that the de Rham (Dolbeault) operators are independent of the metric (Hermitian metric) used to pick an orthonormal frame in which they can be written locally. We observe that when $A_i$ are locally constant, if the vector fields $V_i$ are rescaled by some functions $1/f_i(x)$ that satisfy $\nabla^E_{V_i}f_j(x)=0$ for all $i,j$, then $P$ can be written as $P=\sum_i A_i f_i(x) V_i^\flat \wedge \nabla^E_{\frac{1}{f_i(x)}V_i}$, and $P$ commutes with $\nabla^E_{\partial_x}$.
\end{remark}

We refer to \cite{bandara2024geometric} for investigations showing that the nilpotency is key to the metric independence property on smooth manifolds with certain non-smooth metrics.

An adapted complex structure picks out a polarization of the bundle $\prescript{{a}}{}{T}^*X \otimes \mathbb{C}$ (where we can extend the adapted metric by complex linearity), which is the distribution $W$ generated by the anti-holomorphic covector fields (see Definition 3.18 of \cite{berline2003heat}), and we can construct a unique $\mathbb{Z}_2$ graded Clifford module $S=S^+ \oplus S^-$ called the spinor module, such that 
\begin{equation}
    cl(\prescript{{a}}{}{T}^*X \otimes \mathbb{C}, \prescript{{a}}{}{g}) \otimes \mathbb{C} \cong \text{End}(S)
\end{equation}
where $S$ is simply the exterior algebra of the distribution $W$ (see Proposition 3.19 \cite{berline2003heat}, c.f. \cite{epstein2006subelliptic}). Then 
\begin{equation}
    cl(w)s = \sqrt{2} w \wedge s, \quad   cl(\overline{w})s = -\sqrt{2} \iota_{\overline{w}^\#} s
\end{equation}
if $w \in W$ ($\overline{w} \in \overline{W}=W^*$), and the chirality operator in \eqref{equation_even_dimensional_spaces_chirality_operator} reduces to $(-1)^q$ where $q$ is the degree of forms.
This shows that if we have an adapted complex structure, any Dirac-type operator can be understood in terms of twisted Dolbeault-Dirac operators which we define in more detail in Subsection \ref{subsection_Dolbeault_spin_c_Dirac}.

Given such nilpotent operators participating in a Hilbert complex, we define the \textbf{generalized Neumann/Dirichlet boundary conditions} for the Dirac operator (which we also call \textbf{absolute/relative boundary conditions}) to be
\begin{equation}
\label{boundary_conditions_for_Dirac}
    \sigma(P^*)(dx)u|_{x=1} =0, \quad \sigma(P)(dx)u|_{x=1} =0
\end{equation}
respectively.

\subsubsection{Hodge-Dirac and de Rham operators}
\label{subsection_Hodge_Dirac_operator}

Given an adapted metric, we can construct the corresponding Hodge-de Rham Dirac operator similar to the smooth setting.
We follow \cite[\S 4.3]{Zhanglectures} and introduce the Clifford operators 
\begin{equation}
    cl(u)=u \wedge - \iota_{u^\#}, \quad \widehat{cl}(u)=u \wedge + \iota_{u^\#}
\end{equation}
acting on the \textbf{adapted forms}, where $u$ is an adapted one form, and $cl(u)$ is the Clifford multiplication for the de Rham complex. Zhang's notation is for the Clifford algebra on the tangent space, while we use that on the adapted cotangent bundle by duality. This action extends to forms with coefficients in a flat bundle over $X$ similar to the smooth and wedge settings (see \cite[\S 2]{albin2013novikov}).

Given a flat connection on a bundle $E$, there is an induced connection on forms valued in $E$, which we denote by $\nabla^E$. Taking the anti-symmetrization after acting on forms by $\nabla^E$ defines an operator $d_E$ and this satisfies $d_E^2=0$ (due to the flatness of the connection) similar to the smooth setting. 

Composing $\nabla^E$ with the Clifford action we obtain the twisted Hodge de Rham Dirac operator $D$ acting on sections supported on $\mathring{X}$. The de Rham operator can be written in terms of any orthonormal frame of the adapted cotangent bundle $\{\theta^i\}$ as 
\begin{equation}
    d_E=\theta^i \wedge \nabla^E_{ \theta^{i \#}}.
\end{equation}

\begin{remark}
\label{remark_d_vs_adjoint_rescaling_de_Rham}
It is well known that the de Rham operator is independent choices of Reimannian metrics on smooth spaces, even if they are more easily expressed in orthonormal bases of the tangent bundle. Even in our setting, this continues to be the case. In particular, we can rescale the orthonormal frame of the adapted cotangent bundle by appropriate factors of $x$, and constants to obtain an orthonormal frame $\{\widetilde{\theta}^i\}$ of the cotangent bundle of $X$ with respect to the metric $\widetilde{g}=dx^2+\sum_j g_{Z_j}$, and we can write the de Rham operator as 
\begin{equation}
    d=\widetilde{\theta}^i \wedge \nabla^E_{ \widetilde{\theta}^{i \#}}
\end{equation}
since the rescaling in the tangent and co-tangent bundles cancel off.
However the adjoints $\delta$ of $d$ depend on the choice of metric on the space used in its definition.
\end{remark}

\begin{example}
\label{Example_Dirac_operator_de_Rham}
For the space in Example \ref{Example_complex_structure}, the Hodge de Rham operator takes the form 
\begin{multline}
\label{equation_Dirac_op_de_Rham_example}
    D_X= cl(dx)\nabla^E_{\partial_x}+ (x^3d\theta_1 \wedge - \iota_{\frac{1}{x^3} \partial_{\theta_2}}) \nabla^E_{\frac{1}{x^3}\partial_{\theta_1}} \\+(x^2d\theta_2 \wedge - \iota_{\frac{1}{x^2} \partial_{\theta_2}}) \nabla^E_{\frac{1}{x^2}\partial_{\theta_2}} +(x^2d\theta_3 \wedge - \iota_{\frac{1}{x^2} \partial_{\theta_3}}) \nabla^E_{\frac{1}{x^2}\partial_{\theta_3}}
\end{multline}
when $E=\mathbb{C}$.
It is well known that restricted to $X^{reg}$ we can write the operator as $D=d+d^*$ where
$d$ is the de Rham operator and $d^*$ its formal adjoint. 
The de Rham operator is independent of the metric, and this is why we can write $d$ on this space as 
\begin{multline}
\label{equation_de_Rham_example_auxiliary}
    dx \wedge  \nabla^E_{\partial_x}+ x^3d\theta_1 \wedge \nabla^E_{\frac{1}{x^3}\partial_{\theta_1}} +x^2d\theta_2 \wedge \nabla^E_{\frac{1}{x^2}\partial_{\theta_2}} +x^2d\theta_3 \wedge  \nabla^E_{\frac{1}{x^2}\partial_{\theta_3}}\\=dx \wedge  \nabla^E_{\partial_x}+ d\theta_1 \wedge \nabla^E_{\partial_{\theta_1}} +d\theta_2 \wedge \nabla^E_{\partial_{\theta_2}} +d\theta_3 \wedge  \nabla^E_{\partial_{\theta_3}}
\end{multline}
where we have used the orthonormal frames on the usual tangent/cotangent bundles to express the operator as opposed to the rescaled adapted bundles. It is the adjoint of the operator $d^*$ which depends on the metric.
The operator $cl(dx)=dx\wedge +((-1)^{km+m+1}\star dx \wedge \star)$ (on forms of degree $k$ and $m$ is the dimension of the space) where the Hodge star operator is with respect to the swp metric.
\end{example}

Given a form $\phi$ on $Z$ we can extend it to a form (section of $\Lambda^*X$) on $X$ as $\pi_Z^*\phi$ where $\pi_Z$ is the projection from $X$ to the link $Z$ (at any fixed value of the radial variable $x$). While $cl(dx)$ acts by $dx$ on $\phi$, the action on $\omega=dx \wedge \phi$ is more complicated, but can be expressed using the rescaling factor.
We assume here that $\phi$ is of a fixed multi-degree (see Definition \ref{Definition_rescaling_function} where the rescaling factors are also introduced), whence we can write 
\begin{equation}
    \star_X \omega =\star_X \Big(\frac{1}{R(\phi)} dx \wedge R(\phi) \phi \Big)=\frac{1}{R(\phi)}R(\star_Z \phi) \star_Z \phi
\end{equation}
where $\star_X$ is the Hodge star operator for the swp metric and where $\star_Z$ is the Hodge star operator for the metric on the link $Z$ appearing in the swp metric. It is easy to check that $R(\star_Z \phi)=R(dvol_Z)/R(\phi)$.

Since $\delta=d^*=(-1)^{km+m+1} \star d\star$, we see that $cl(dx)\nabla_{\partial_x}$ acts on $\omega$ by
\begin{equation}
    cl(dx)\nabla_{\partial_x} \omega = (-1)^{km+m+1} \star_X \Big( dx \wedge \nabla_{\partial_x} \frac{R(dvol_Z)}{R^2(\phi)} \star_{Z}\omega\Big) =
    -\Big(\frac{R^2(\phi)}{R(dvol_Z)}\Big)\Big(\frac{R(dvol_Z)}{R^2(\phi)}\Big)' \phi
\end{equation}
and similarly we can compute
\begin{equation}
    cl(dx) \nabla_{\partial_x} (u_1(x)\phi+ u_2(x)dx \wedge \phi)=dx \wedge u_1'(x) \phi- \Big(\frac{R^2(\phi)}{R(dvol_Z)}\Big) \Big(u_2 \frac{R(dvol_Z)}{R^2(\phi)}\Big)' \phi
\end{equation}
where the second term can be expressed as
\begin{equation}
\label{equation_definition_delta_1}
    \delta_1(u_2(x)dx \wedge \phi)=-\Bigg(u'_2 + \Big(\frac{R^2(\phi)}{R(dvol_Z)}\Big) \Big(u_2 \frac{R(dvol_Z)}{R^2(\phi)}\Big)' \Bigg)\phi=-\frac{1}{F}(u_2 F)'\phi
\end{equation}
where we define
\begin{equation}
\label{equation_Rescaling_operator_2_adjoint_rescaling_function}
    \delta_1:=D_1-d_1, \quad d_1:=dx \wedge \partial_x, \quad  F(\phi)(x):=\frac{R(dvol_Z)}{R^2(\phi)}.
\end{equation}
we call the function $F(\phi)$ for a form of fixed multi-degree $\phi$ as the \textbf{adjoint rescaling function} of $\phi$.

Let us consider two forms $u_1(x)\psi_1$ and $u_2 dx \wedge \psi_2$ where ${\partial_x} \psi_i =0$ for $i=1,2$.
Using the above computations, it is easy to see that the action of $(cl(dx)\nabla^E_{\partial_x})^2=(d_1+\delta_1)^2=D_1^2$ on these forms is given by
\begin{equation}
\label{equation_radial_derivatives_1}
    D_1^2 (u_1(x)\psi_1)= \delta_1 d_1 (u_1(x)\psi_1)=-\frac{1}{F}(u_1'F)'\psi_1, 
\end{equation}
where we denote $F(\psi_1)(x)$ by F and  
\begin{equation}
\label{equation_radial_derivatives_2}
    D_1^2 (\frac{1}{F} u_2 dx \wedge \psi_2)= d_1 \delta_1(\frac{1}{F} u_2 dx \wedge \psi_2)=-(\frac{1}{F} u'_2)' dx \wedge \psi_2
\end{equation}
where we denote $F(\psi_2)(x)$ by F where we have rescaled the form for convenience.

\begin{remark}
\label{Remark_SUSY_Hodge_star_rescalings}
Given forms $\psi_1, \psi_2$ satisfying $\star_Z \psi_1=\psi_2$, we observe that $R(dvol_Z)=R(\psi_1)R(\psi_2)$, and deduce from the definition of $F$ that $F^{-1}(\psi_1)=F(\psi_2)$, and that $F^{-1}(\psi_1)=R^2(\psi_2)/R(dvol_Z)$. This can be used to see that the action of $d_1$ and $\delta_1$ are intertwined by $\star_X$ (up to a sign), which is a consequence of the fact that $d$ and $\delta$ are intertwined by $\star_X$ (up to a sign).
We will use this in the study of Laplace-type operators in Section \ref{subsection_Laplacians}.    
\end{remark}

\begin{remark}
\label{remark_Witt_condition_1}
    This generalizes computations of the Hodge Dirac operator and the Laplace-type operator worked out in the proof of Proposition 5.33 of \cite{jayasinghe2023l2} in the conic (wedge) case. We observe that when $R^2(\phi)=R(dvol_Z)$, the computations above simplify, which in the conic case happens when $\phi$ is a middle dimensional form on the link $Z$ (if $Z$ is even dimensional). Thus middle dimensional harmonic forms $\phi$ on the link (with respect to $\Delta_Z$) extend to harmonic forms which are not in the minimal domain of the Hodge Laplacian, and the \textbf{Witt} condition (see \cite{Albin_hodge_theory_cheeger_spaces}) is not satisfied.
\end{remark}

\subsubsection{Dolbeault and Dolbeault-Dirac operators}
\label{subsection_Dolbeault_spin_c_Dirac}

We follow the conventions in \cite{wu1998equivariant}, referring to chapters 3 and 5 of \cite{duistermaat2013heat} for a more detailed construction of the spin$^\mathbb{C}$ Dirac operator in the smooth setting. In the K\"ahler case this is equal the Dolbeault-Dirac operator. We define a Clifford algebra structure on the adapted cotangent bundle which determines this operator. 

Given an adapted spin$^\mathbb{C}$ structure on $X$ and a compatible almost complex structure, the complexified adapted co-tangent bundle $^{a} T^*_{\mathbb{C}}X$ has an orthogonal splitting into the holomorphic and anti-holomorphic tangent bundles as $^aT^*X^{1,0} \oplus ^aT^*X^{0,1}$ induced by the adapted complex structure.
Given $\xi \in \Gamma(^{a} T^*_{\mathbb{C}}X)$ where $^{a} T^*_{\mathbb{C}}X=^{a} T^*X \otimes_{\mathbb{R}} \mathbb{C}$, we can write it as $\xi=\xi^{1,0}+\xi^{0,1}$ where $\xi^{1,0} \in \Gamma(^{a} T^*X^{1,0})$ and $\xi^{0,1} \in \Gamma(^{a} T^*X^{0,1})$.
Then the Clifford action of the complexified Clifford algebra is defined by
\begin{equation*}
cl(\xi^{0,1})=\sqrt{2}(\xi^{0,1}) \wedge, \quad cl(\xi^{1,0})=-\sqrt{2} i_{(\xi^{1,0})^{\#}},
\end{equation*}
where $(\xi^{1,0})^{\#}=v^{0,1} \in \Gamma(^{a} TX^{0,1})$ corresponds to $\xi^{1,0}$ via the K\"ahler swp metric. Given a Hermitian bundle $E$ on $X$ equipped with a compatible connection, this extends to a Clifford action on $End(F)$ where $F= \oplus_{q=0}^n \Lambda^{q} (^{a} T^*X^{1,0}) \otimes E$ by $cl \otimes Id$, which we denote by $cl$ with abuse of notation. 

Given a Hermitian connection on a holomorphic bundle $E$, we can extend it to a Hermitian connection on 
$F=\Lambda^{q} (^{a} T^*X^{1,0}) \otimes E$ which we denote by $\nabla^F$. 
Composing $\nabla^F$ with the Clifford action we obtain the spin$^\mathbb{C}$ Dirac operator $D$ acting on sections supported on $\mathring{X}$.

Given an orthonormal frame of the adapted cotangent bundle $\{\theta^i \}$, we have the local form of the Dirac operator given in \ref{equation_local_frame_Dirac_operator} which is the \textbf{spin$^{\mathbb{C}}$ Dirac operator} corresponding to the Clifford structure.
It is well known that in the K\"ahler case we can write the operator as $D=\sqrt{2}(\overline{\partial}+\overline{\partial}^*)$ using
\begin{equation}
\label{equation_d_bar_spin_c_correspondence_1}
\bar{\partial}_E=\frac{1}{2 \sqrt{2}} \sum_{i=1}^{2 n} cl(\theta^i-\sqrt{-1} J \theta^i) \nabla^F_{(\theta^i-\sqrt{-1} J \theta^i)^{\#}},\quad \bar{\partial}^{*}_E=\frac{1}{2 \sqrt{2}} \sum_{i=1}^{2 n} cl(\theta^i+\sqrt{-1} J \theta^i) \nabla^F_{(\theta^i-\sqrt{-1} J \theta^i)^{\#}}.
\end{equation}
where $J$ is a complex structure that tames the adapted K\"ahler form where the precise expressions arising with $cl(dx)$ can be worked out in a fashion similar to those for the Hodge Dirac operator in the previous subsection.
If there is only a complex structure, then the Dolbeault operator $P=\bar{\partial}_E$ is still defined by the expression in \eqref{equation_d_bar_spin_c_correspondence_1} above, and the operator $D=P+P^*$ is referred to as the \textbf{Dolbeault-Dirac} operator. 

\begin{example}
\label{Example_Dirac_operator_Dolbeault}
For the space in Example \ref{Example_complex_structure} equipped with the complex structure $J$ that takes $\partial_x$ to $1/x^3 \partial_{\theta_1}$ and $1/x^3 \partial_{\theta_2}$ to $1/x^2 \partial_{\theta_2}$ (this determines a complex structure on the endomorphism bundle of the adapted tangent bundle), we can write the Dolbeault operator $\overline{\partial}$ as
\begin{equation}
\label{equation_Dirac_op_abstract_example}
\frac{1}{2\sqrt{2}} \bigg( (dx - \sqrt{-1} x^3 d\theta_1) \wedge  (\partial_{x}+\sqrt{-1} \frac{1}{x^3}\partial_{\theta_1})+
(x^2 d\theta_2 - \sqrt{-1} x^2 d\theta_3) \wedge (\frac{1}{x^2}\partial_{\theta_2}+\sqrt{-1} \frac{1}{x^2}\partial_{\theta_3}) \bigg)
\end{equation}
where we have denoted by (with abuse of notation) $J$ the adapted complex structure on the adapted cotangent bundle (which maps $dx$ to $x^3 d\theta_1$, and $x^2 d\theta_2$ to $x^2 d\theta_3$).
\end{example}

In the case of a space with an adapted complex structure, the vector field $J(\partial_x)$ restricted to the link $Z$ generates a group action on $Z$, for any fixed $x>0$. If we define 
\begin{equation}
\label{equation_transversal_Dolbeault}
    P_1=\frac{1}{2 \sqrt{2}} cl(dx-\sqrt{-1} J (dx)) \nabla^F_{\partial_x+\sqrt{-1}J(\partial_x)}, \quad P_2:=P-P_1
\end{equation}
where $P=\overline{\partial}_E$ and $P_2$ is a \textbf{transverse Dolbeault operator}. In the setting where the quotient of $Z$ by the foliation generated by the vector field $J(\partial_x)$ is a smooth manifold, it inherits a complex structure given by the restriction of the adapted complex structure. The transverse Dolbeault operator is then a Dolbeault operator with twisted coefficients on the quotient. Such structures have been studied on spaces with conic metrics even in the case where the links are stratified pseudomanifolds in \cite{jayasinghe2023l2,jayasinghe2024holomorphic}.

\begin{remark}
\label{remark_d_vs_adjoint_rescaling_Dolbeault}
Similar to the observations in Remark \ref{remark_d_vs_adjoint_rescaling_de_Rham} we observe that given a complex structure, the formal operator $\bar{\partial}_E$ is independent of the metric on $g$ while the adjoint is not, since $\bar{\partial}=\pi \circ d$ where $\pi$ is the projection onto the holomorphic part of the form bundle. Thus, the transverse Dolbeault operators are also independent of the metric.
\end{remark}

\subsection{Witten deformed operators}
\label{Subsection_Witten_deformed_ops}

We briefly expand on the Witten deformed operators defined in the introduction, referring the reader to \cite{witten1982supersymmetry,Zhanglectures,jayasinghe2023l2} for more details in the de Rham setting, and \cite{witten1984holomorphic,mathai1997equivariant,wu1998equivariant,jayasinghe2024holomorphic} for details in the Dolbeault setting.

Given a radial Morse (K\"ahler Hamiltonian Morse) function $h$ on $X$, the Witten deformed operators, their adjoints and the deformed Dirac-type operators for the de Rham (Dolbeault) complexes are
\begin{equation}
    P_\varepsilon =e^{-\varepsilon h} P e^{+\varepsilon h}, \quad P^*_\varepsilon=e^{+\varepsilon h} P^* e^{-\varepsilon h}, \quad D_{\varepsilon}=P_\varepsilon+P^*_{\varepsilon}
\end{equation}
and setting $\varepsilon=0$ we recover the undeformed complexes.
In the case of the de Rham complex we have that
\begin{equation}
    D_\varepsilon=D+ \varepsilon \widehat{cl}(dh)
\end{equation}
where $\widehat{cl}$ is the (right) Clifford operator on the exterior algebra defined by $\widehat{cl}(\theta)= \theta \wedge + \iota_{\theta^\#}$. 
In the case of the Dolbeault complex, we have 
\begin{equation}
    D_\varepsilon=D+ \sqrt{-1} \varepsilon {cl}(V^{\flat})
\end{equation}
where $V^{\flat}$ is the dual form (with respect to the metric on $Z$), for the K\"ahler Hamiltonian vector field $V$.

\section{Hilbert complexes and SUSY QM}
\label{Section_Hilbert_complexes}

In this section we first discuss Hilbert complexes that we study in this article following \cite{bru1992hilbert}, and describe how they correspond to SUSY QM systems following \cite{witten1982supersymmetry} (c.f. \cite{berwickevans2024ellipticcohomologyquantumfield}). We then introduce the domains for the operators that we study on $X$.

\subsection{Hilbert complexes}

We review the definition of Hilbert complexes in \cite{bru1992hilbert}. 
\begin{definition}
A \textbf{\textit{Hilbert complex}} is a complex, $\mathcal{P}=(H_*,\mathcal{D}(P_*),P_*)$, of the form:
\begin{equation}
    0 \rightarrow \mathcal{D}(P_0) \xrightarrow{P_0} \mathcal{D}(P_1) \xrightarrow{P_1} \mathcal{D}(P_2) \xrightarrow{P_2} ... \xrightarrow{P_{n-1}} \mathcal{D}(P_n) \rightarrow 0.
\end{equation}
where each map $P_k$ is a closed 
operator which is called the differential, such that:
\begin{itemize}
    \item the domain of $P_k$, $\mathcal{D}(P_k)$, is dense in $H_k$ which is a separable Hilbert space,
    \item the range of $P_k$ satisfies $ran(P_k) \subset \mathcal{D}(P_{k+1})$,
    \item $P_{k+1} \circ P_k = 0$ for all $k$.
\end{itemize}
\end{definition}

We will often denote such complexes by $\mathcal{P}=(H,\mathcal{D}(P),P)$ without explicitly denoting the grading. 
If we forget the information of the domain, we refer to $\mathcal{P}=(H,P)$ as a \textbf{pre-Hilbert complex} and we will study how to pick suitable domains to promote pre-Hilbert complexes to Hilbert complexes.

In this article we will study Hilbert complexes where $D=P+P^*$ is a Dirac operator as introduced in the previous subsection.
Given such a Hilbert complex, we have an \textbf{associated Dirac-type operator} 
\begin{equation}
\label{Domain_Dirac_first}
   D=P+P^*, \quad \mathcal{D}(D)=\mathcal{D}(P) \cap \mathcal{D}(P^*).
\end{equation}
acting on the direct sum of Hilbert spaces $H_k$, and an \textbf{associated Laplace-type operator} 
\begin{equation}
     \Delta:=D^2=PP^*+P^*P, \quad \mathcal{D}(\Delta) = \{ v \in \mathcal{D}(D) : D v \in \mathcal{D}(D) \},
\end{equation}
where we use the fact that $P^2=0$ (from which it follows that $(P^*)^2=0$) to see that
\begin{equation}
\label{Laplacian_P_type}
\mathcal{D}(\Delta_k) = \{ v \in \mathcal{D}(P_k) \cap \mathcal{D}(P_{k-1}^*) : P_k v \in \mathcal{D}(P_k^*), P^*_{k-1} v \in \mathcal{D}(P_{k-1}) \}.
\end{equation}

For every Hilbert complex $\mathcal{P}$ there is an \textit{\textbf{adjoint Hilbert complex}} $\mathcal{P}^*$, given by
\begin{equation}
\label{adjoint_complex}
    0 \rightarrow \mathcal{D}((P_{n-1})^*) \xrightarrow{(P_{n-1})^*} \mathcal{D}((P_{n-2})^*) \xrightarrow{(P_{n-2})^*} \mathcal{D}((P_{n-3})^*) \xrightarrow{(P_{n-3})^*} ... \xrightarrow{(P_{1})^*} H_0 \rightarrow 0
\end{equation}
where the differentials are $P_k^*: Dom(P^*_k) \subset H_{k+1} \rightarrow H_k$, the Hilbert space adjoints of the differentials of $\mathcal{P}$. The Hilbert space in degree $k$ of the adjoint complex $\mathcal{P}^*$ is the Hilbert space in degree $n-k$ of the complex $\mathcal{P}$, and the operator in degree $k$ of $\mathcal{P}^*$ is the adjoint of the operator in degree $(n-1-k)$ of $\mathcal{P}$. 
We observe that the associated Laplace-type and Dirac type operators for a complex are equal to those of the adjoint complex (up to changes in degree).

In this article we work with Hilbert complexes for which the associated Laplace-type operators have discrete spectrum, and the cohomology groups, the \textit{\textbf{cohomology groups}} of a Hilbert complex are defined to be $\mathcal{H}^k(\mathcal{P}):= ker(P_k)/ran(P_{k-1})$. The corresponding cohomology groups of $\mathcal{P}^*$ are $\mathcal{H}^k(\mathcal{P}^*) := ker(P^*_{n-k-1})/ran(P^*_{n-k})$.

We can form a two step complex where the Hilbert spaces are $H^+ =\bigoplus_{q=even} H_q$, and $H^- = \bigoplus_{q=odd} H_q$, which leads to a \textit{\textbf{formal Dirac complex}} 
\begin{equation}
\label{two_term_Dirac_complex}
    0 \rightarrow \mathcal{D}(D^{+}) \xrightarrow{D^+} \mathcal{D}(D^{-}) \rightarrow 0
\end{equation}
where $D^{\pm}$ is simply the restriction of $D$ to $H^{\pm}$.

\textbf{In this article we will restrict our attention to complexes where the Hilbert spaces $H_k$ (and the domains $\mathcal{D}(\Delta)$) have orthonormal bases of eigensections of the associated Laplace-type operators}, in which setting, the cohomology of the complex is isomorphic to the null space of the Laplace type operator, and is isomorphic to the null space of the Dirac type operator (on their respective domains), and we will denote those vector spaces by $\mathcal{H}^k(\mathcal{P})$.
Having discrete spectrum implies there is a Hodge-Kodaira decomposition for the complex
\begin{center}
    $H_k=\mathcal{H}^k(\mathcal{P}) \bigoplus \overline{ran(P_{k-1})} \bigoplus \overline{ran(P_k^*)}$
\end{center}
and we refer to the sections in the second and the third summands as exact and co-exact sections.
The following is a version of Lemma 3.13 of \cite{jayasinghe2023l2} which elucidates this further. We refer to \textit{loc. cit.} for a proof, noting that it is a consequence of the nilpotency of $P, P^*$ and the fact that they commute with the Laplace-type operator.

\begin{lemma}
\label{Lemma_super_symmetry}
For a Hilbert complex $\mathcal{P}=(H,\mathcal{D}(P),P)$ where $P$ is nilpotent, the Laplace-type operators have discrete spectrum, given $\{v_{\lambda_i}\}$, an orthonormal basis of eigensections (with eigenvalues $\lambda^2_i$) for the co-exact elements of degree $k$, the set $\{\frac{1}{\lambda_i}P_k v_{\lambda_i}\}$ is an orthonormal basis of eigenvectors (of the operator $\Delta_{k+1}$) for the exact elements of degree $k+1$.
\end{lemma}

This exact/co-exact decomposition refines the decomposition given by the degree, in particular the odd and even decomposition when $P$ is a nilpotent, and such decompositions are used widely in physics (see \cite[\S 16.4]{peskin2018introduction} for a discussion of a similar decomposition with BRST operators).
This decomposition is crucial in formulating general Morse inequalities.

If instead we only have a two term Dirac complex (\ref{two_term_Dirac_complex}), there the exact and co-exact sections with respect to $D^+$ are the odd (in the image of $D^+$) and even (in the image of $D^-$) sections. This corresponds to a $\mathcal{N}=1$ SUSY QM theory in the sense we will discuss in the next subsection.

\subsubsection{Witten deformed complexes}
\label{subsubsection_Witten_Deformed_domains}

Moreover given a Morse function/ K\"ahler Hamiltonian Morse function $h$ on a space $X$ with a twisted de Rham/ Dolbeault complex $\mathcal{P}=(H, \mathcal{D}(P), P)$, we define the corresponding Witten deformed complex to be $\mathcal{P}_{\varepsilon}=(H, \mathcal{D}(P_\varepsilon), P_\varepsilon)$ where 
\begin{equation}
    \mathcal{D}(P_\varepsilon):= \{ u \in H| e^{-\varepsilon h}u \in \mathcal{D}(P) \}
\end{equation}
and it is clear that the undeformed complex is given by $\varepsilon=0$. 

\begin{remark}
\label{Remark_introductory_Morse}
The Hilbert complexes for any two values of $\varepsilon \geq 0$ are canonically isomorphic, in particular giving an isomorphism of the cohomology (and harmonic representives). In the case of the deformed Dolbeault complex, one can organize the basis of infinite dimensional harmonic sections into a (countable) direct sum of finite dimensional eigenspaces of the operator $\sqrt{-1}L_V$. The operator $\Delta_{\varepsilon}$ in \ref{equation_deformed_Laplace_operator_intro_1} is reminiscent of a Schr\"odinger operator for a potential well, where the eigenvalues (energy levels) of the non-zero energy states grow to $\infty$ in the semi-classical limit, $\varepsilon$ goes to $\infty$.

In general one should consider functions for which the gradient flow is both attracting and expanding (in different directions), but it suffices to study the case that we do here and use dualities of local complexes when proving local to global results as in \cite{jayasinghe2024holomorphic} (see also Remark \ref{Remark_signs_Morse_functions_dualities}) to study a large class of functions, including Morse functions on smooth manifolds.
\end{remark}

\subsection{SUSY QM}

In this section we give a brief introduction to supersymmetric quantum mechanics relating it to the complexes we study, mostly following the exposition in \cite{witten1982constraints} (c.f., \cite{berwickevans2024ellipticcohomologyquantumfield,deligne1999quantum}). We also describe Morse polynomials and the instanton complexes that categorify them, relating them to a semi-classical limit of a SUSY QM system following \cite{witten1982supersymmetry}.

\subsubsection{Classical and Quantum mechanical systems}

We begin by introducing several basic notions following a lecture of Witten (see Lecture 1 of part 4 of \cite{deligne1999quantum}) and refer the reader to that text for more details.

A \textit{\textbf{classical theory}} is a pair $(X, \Ham)$ where $X$ is a symplectic space and $\Ham$ is a real valued function called the Hamiltonian (defined up to a locally constant function). If $X$ is a smooth symplectic manifold, it can be reconstructed as the spectrum of the Poisson algebra $A=C^{\infty}(X)$ in the smooth setting.
A \textit{\textbf{quantum theory}} is defined to be a pair $(\mathcal{A},\widehat{\Ham})$ where $\mathcal{A}$ is a $^*$-algebra (the algebra of observables) and $\widehat{\Ham}$ is a self-adjoint element (called the Hamiltonian) of $\mathcal{A}$ defined up to a real constant term. Such pairs depend on a real positive parameter $\hbar$ (Planck's constant).

A \textit{\textbf{realization}} (or solution) of a theory is defined to be an irreducible $^*$-representation of the algebra $\mathcal{A}$ in some Hilbert space $H$ such that the spectrum of the operator $\widehat{\Ham}$ is bounded from below (representations are considered up to isomorphisms which preserve $\widehat{\Ham}$). For relativistically invariant theories i.e. theories with an action of the Poincar\'e group on $A$, the subgroup of time translations acts on elements $a \in \mathcal{A}$ by $a \mapsto e^{-it\widehat{\Ham}/\hbar}a e^{it\widehat{\Ham}/\hbar}$.

A \textit{\textbf{(quantum) ground state}} of a quantum theory $(\mathcal{A},\widehat{\Ham})$ with a realization $H$ is a vector in $H$ with the smallest possible eigenvalue.

A \textit{\textbf{quantization map}} is a linear mapping that sends elements $a\in A$ to (their quantizations) $\widehat{a} \in \mathcal{A}$ for each $\hbar$ and $\Ham$ to $\widehat{\Ham}$, with compatibility conditions for the Poisson bracket on $A$ and the Lie bracket on $\mathcal{A}$.
A quantum state $v \in H$ of norm $1$ is said to be localized near a classical state $x \in X=spec(A)$ if for any $a \in A$, its \textit{quantization} $\widehat{a}$ satisfies $\langle v, \widehat{a} v \rangle \rightarrow a$ as $\hbar \rightarrow 0$. In particular, quantum ground states \textbf{localize} near classical ground states (stationary points of $\Ham$). 

\begin{remark}[Morse inequalities and the quantum-classical correspondence]
In \cite{witten1982supersymmetry,witten1984holomorphic} Witten studied a quantization of the classical mechanical system corresponding to the gradient flow of Morse functions, where the associated Hamiltonian operator is of the form in \eqref{equation_deformed_Laplace_operator_intro_1}.
He explained that the localization of quantum ground states near the classical ground states in the semi-classical limit could be used to prove the Morse inequalities. Roughly, the deformed complex has a sub-complex of ``small eigenvalue eigensections", also called the Witten instanton complex, and is quasi-isomorphic to the undeformed complex. But the elements of this complex localize near the critical points of the Morse function (the classical ground states) in the semi-classical limit.

Many of the analytic details were filled in by Helffer and Sj\"ostrand in \cite{helffer1987wells} for the de Rham complex. In the Dolbeault case, the result was proven in \cite{mathai1997equivariant} using the analytic results and techniques developed in \cite{bismut1991complex}, and extensions were used to study various phenomena such as quantization commuting with reduction \cite{TianZhangQUANTIZATION1998,tian1998holomorphic}.
\end{remark}

\subsubsection{SUSY QM systems}

We refer the reader to \cite[\S 1]{witten1982supersymmetry} and \cite[\S 3]{berwickevans2024ellipticcohomologyquantumfield} for more details on the content in this subsection.

Given a realization of a QM theory $(H, \widehat{\Ham})$ as defined above, suppose that there is a splitting $H=H^+ \oplus H^-$ equipped with a $\mathbb{Z}_2$ grading and that $\widehat{\Ham}$ respects the grading (i.e., is an even map). Equivalently $H$ is equipped with an involution $\gamma$ (called the \textbf{grading/chirality/supersymmetry} operator, sometimes denoted by $(-1)^F$) which satisfies
\begin{equation}
\label{equation_simple_supersymmetry_Laplace}
    \gamma= \pm 1 \text{   on  } H^{\pm} ,\quad \text{and   } \widehat{\Ham} 
 \, \gamma=\gamma \, \widehat{\Ham}.
\end{equation}

Supersymmetric quantum mechanics enhances this by odd operators $Q_i \in \mathcal{A}$ for $i=1,..., \mathcal{N}$ which satisfy 
\begin{equation}
\label{equation_SUSY_basic}
    Q_i^2=\widehat{\Ham}, \quad [Q_i, Q_j]=Q_iQ_j+Q_jQ_i =0 \quad i \neq j 
\end{equation}
where the odd condition is equivalent to $\{Q_i, \gamma\} =0$ where $\{ \cdot,\cdot \}$ is the anti-commutator. The operators $Q_i$ are called the supersymmetry operators and it is said the theory has $\mathcal{N}$ supersymmetry. 
In Section 3 of \cite{witten1982constraints}, it is shown that given $\mathcal{N}=2$ SUSY,
the nilpotent operators
\begin{equation}
    Q_{\pm}=\sqrt{\frac{1}{2}} (Q_1 \pm \sqrt{-1} Q_2)
\end{equation}
are introduced, and it is shown that in the zero momentum sector (see Remark \ref{Remark_SUSY_relativistic_momentum}), which is the setting we study in this article, the SUSY algebra takes the simple form 
\begin{equation}
    Q_{\pm}^2=0, \quad Q_+Q_- + Q_- Q_+ =0.
\end{equation}

\begin{example}[de Rham complex]
In \cite[\S 2]{witten1982supersymmetry}, Witten describes the SUSY algebra on $H=L^2\Omega(X;\mathbb{R})$, generated by 
\begin{equation}
    Q_1=d+d^*, \quad Q_2= \sqrt{-1}(d-d^*), \quad H=dd^*+d^*d
\end{equation}
so that the Hamiltonian is the Hodge Laplacian acting on forms, and it is easy to verify the supersymmetry relations.

In addition to the chirality operator given by $(-1)^F$ where $F$ is the degree of the form acted on (called the \textbf{F}ermion number operator), the Hodge star operator $\star$ can be used to produce a chirality operator $(-1)^{q(q-1)/2}\star$ acting on the form bundle on even dimensional spaces $X$ to obtain a different supersymmetric theory (see page 687 \cite{witten1982supersymmetry}), the index of which corresponds to the Signature invariant.
\end{example}

\begin{example}[Witten deformed de Rham complex]
For the chirality given by the number operator, there is a \textbf{family} of realizations of SUSY QM theories determined by 
\begin{equation}
    Q_{1,\varepsilon}=d_{\varepsilon}+d^*_{\varepsilon}, \quad Q_{2,\varepsilon}= \sqrt{-1} (d_{\varepsilon}-d^*_{\varepsilon})
\end{equation}
where $\varepsilon=1/\hbar$. The Hamiltonian is the Schr\"odinger operator $\Delta_{\varepsilon}$ give in \eqref{equation_deformed_Laplace_operator_intro_1}. It is with this system (rather with the $\mathcal{N}=1$ theory obtained by forgeting $Q_2$) that Witten studied classical Morse theory. Witten introduces a different deformation for the Signature operator in \cite{witten1982supersymmetry}, in the presence of a Killing vector field.
\end{example}

\begin{example}[Witten deformed Dolbeault complex]
\label{example_Witten_deformed_QM_theory_Dolbeault}
In \cite{witten1983fermion,witten1984holomorphic} Witten studied the SUSY algebra on $H=L^2\Omega(X;E)$ where $E$ is a holomorphic vector bundle, equipped with a K\"ahler Hamiltonian function $h$ whose K\"ahler action lifts to a linear action on the bundle $E$.
This determines a family of $\mathcal{N}=2$ SUSY QM theories generated by
\begin{equation}
    Q_{1,\varepsilon}=P_{\varepsilon}+P^*_{\varepsilon}, \quad Q_{2,\varepsilon}= \sqrt{-1} (P_{\varepsilon}-P^*_{\varepsilon})
\end{equation}
where $P=\overline{\partial}_E$. 
The Hamiltonian is the Schr\"odinger type operator $\Box_{\varepsilon}$ given 
in \eqref{equation_deformed_Laplace_operator_intro_1}, which was used to formulate holomorphic Morse inequalities in \cite{witten1984holomorphic} 
(c.f. \cite{jayasinghe2024holomorphic}). 
The undeformed complex can be defined even without a K\"ahler Hamiltonian Morse function.
\end{example}

In the case of both the de Rham and Dolbeault complexes (with or without Witten deformation) with nilpotent operator $P$,
since the $\mathcal{N}=2$ SUSY algebra corresponds to $Q_1=P+P^*$ and $Q_2=\sqrt{-1}(P-P^*)$, we can check that
\begin{equation}
    Q_{+}=\sqrt{2} P^*, \quad Q_{-}=\sqrt{2} P,
\end{equation}
and although the SUSY algebra does not guarantee that the $\mathbb{Z}_2$ grading can be refined to a $\mathbb{Z}$ grading, it exists in the case of these complexes. Witten explains that the deformation given by conjugating the operators $Q_{\pm}$ is in general inequivalent to the undeformed one, although the cohomology of the complexes are isomorphic (see the discussion following equation(12) of \cite{witten1982constraints}).

Witten introduced the complex in Example \ref{example_Witten_deformed_QM_theory_Dolbeault} in \cite{witten1983fermion} while studying Fermion quantum numbers on Kaluza Klein space where the Rarita-Schwinger operators are studied on Kaluza-Klein spaces with a Riemannian metric, modded out by ghost fields (see the discussion below \textit{loc. cit.} of the article). There the BRST operator corresponding to the ghost fields is also nilpotent on the field configurations (see also \cite[\S 16.4]{peskin2018introduction}), and is actually isomorphic to the Dolbeault operator twisted by a \textit{square root} of the canonical bundle, which is why the (equivariant) indices of the Spin Dirac operator appear in the (equivariant) index formula for the Rarita-Schwinger operator. We refer to \cite[\S 7.4]{jayasinghe2024holomorphic} for a discussion on extending this to the singular setting.

\begin{remark}
The above example together with the remarks in Subsection \ref{subsection_polarizations_nilpotent} shows that a two term Dirac complex yields a $\mathcal{N}=2$ SUSY theory only if there is a compatible complex structure, although there are studies of Witten deformation on symplectic spaces with just almost complex structures (see \cite{tian1998holomorphic,TianZhangQUANTIZATION1998}).
\end{remark}

\begin{remark}
\label{Remark_SUSY_relativistic_momentum}
The supersymmetric algebra needs to be generalized in relativistic quantum field theory to include momentum operators, and we refer to \cite{witten1982supersymmetry} for more details. We note that while the momentum operators that arise in the systems described in \textit{loc. cit.} are of the form $\sqrt{-1}\varepsilon L_V$, similar to that appearing in \eqref{equation_deformed_Laplace_operator_intro_1}, the SUSY theory in \eqref{example_Witten_deformed_QM_theory_Dolbeault} is non-relativistic.
\end{remark}

A direct consequence of the SUSY algebra is that the states annihilated by the operators $Q_i$ are in the null space of $\widehat{\Ham}$, which are the zero energy states and are thus ground states (since we assume that $\widehat{\Ham}$ is a non-negative operator. If the ground state of the system has positive energy, we say that the supersymmetry is \textbf{\textit{spontaneously broken}}, and we refer to \cite{witten1982supersymmetry} for more details (c.f. \cite[\S 6]{witten1981dynamical} for models with dynamical SUSY breaking).

In general, an irreducible realization of a quantum theory can have many ground states. For example, for the theory of the Dirac operator on a compact manifold, ground states are harmonic spinors; the space of harmonic spinors can be more than one dimensional. Supersymmetry can only be present in the universe if it is \textbf{spontaneously broken}, that is if the vacuum states (lowest energy states) have positive energy. A prototypical example is the case of a spin Dirac operator on a space with positive scalar curvature, whence the positivity of the energy of vacuum states is guaranteed by Lichnerowicz's theorem. In the case where there is a K\"ahler Hamiltonian action, this is also guaranteed by the holomorphic Morse inequalities (see \cite[\S 7.3]{jayasinghe2024holomorphic}).

\begin{remark}
    There are several slightly different variations in how supersymmetric quantum mechanics and SUSY quantum mechanical systems (QM systems) are formalized in the literature, depending on various spectral properties on the spaces being studied. For instance in Definition 3.3 of \cite{berwickevans2024ellipticcohomologyquantumfield}, it is assumed that the Hamiltonian has trace class resolvent, whereas we work with systems which have compact resolvent only for operators $\Delta+\Delta_Z$ or after restricting to eigenspaces of $\sqrt{-1}L_V$ (see Theorem \ref{Theorem_main_spectral}), and hence have to renormalize various traces that we compute. 
\end{remark}

\subsection{Domains for operators on singular geometries.}
\label{subsection_domains_on_X}

Many formally self-adjoint operators on singular spaces including the ones we study in this article are not necessarily essentially self adjoint on singular spaces and spaces with boundaries.
There are two canonical domains which are the minimal domain,
\begin{equation}
    \mathcal{D}_{min}(P_X)= \{ u \in L^2(X;F) : \exists (u_n) \subseteq {C}^{\infty}_c(\mathring{X};F) \text{ s.t. }
    u_n \rightarrow u \text{ and } (P_Xu_n) \text{ is } L^2-\text{Cauchy} \},
\end{equation}
for any operator $P_X$ and
\begin{equation*}
    \mathcal{D}_{max}(P_X)= \{ u \in L^2(X;F) : (P_Xu) \in L^2(X;F) \},
\end{equation*}
where $P_Xu$ is computed distributionally.  Given any choice of self-adjoint extension $\mathcal{D}_I(D)$ of a Dirac-type operator that we study in this article we can study the spectrum of the operators $D_{Z_j}$ on the sections of a link (say at $x=1/2$) defined in equation \eqref{equation_Dirac_intermediary_1} and we can define the following conditions.

\begin{definition}[adapted Witt condition]
\label{adapted_Witt_assumption}
Consider $X$ with a swp metric and an adapted complex $\mathcal{P}_{W,B}$ of the types we consider in this article. Consider all the $\Delta_Z$ harmonic sections $\phi$ on $Z$ which determine $\Delta_Z$ harmonic sections on $X$ that are $L^2$ bounded and are contained in the domain of the Laplace-type operator $\Delta_X$ of the complex. If the adjoint rescaling function 
\begin{equation}
    F(\phi)=R^2(dvol_Z)/R(\phi^2) \neq 1
\end{equation}
for all such sections, we say that we say the complex satisfies the \textbf{adapted Witt condition}.
\end{definition}

\begin{remark}
We observe that in the case of conic metrics, $F(\phi)=1$ for forms in middle degree of a link $Z$ of even degree, and no other forms, and thus the above definition generalizes the Witt condition that is well known in the conic setting (see, e.g. \cite{Albin_signature,Albin_2017_index}).
\end{remark}

It is known that if the metric is conic (wedge) on a compact stratified pseuodomanifold, then the Witt condition guarantees that the Hodge Dirac operator is essentially self adjoint (see, e.g., \cite{Albin_signature}).
The simplest way to understand these conditions is via the Green-Stokes formula, or integration by parts near the singularity (resolved boundary at $x=0$).

Recall that for an operator $P \in \text{Diff}^1(U;E,F)$ on a manifold with boundary $(U, \partial U)$ we have the Green-Stokes formula (see, e.g., Proposition 9.1 of \cite{taylor1996partial})
\begin{equation}
\label{equation_Greens_identity}
    \langle Ps, \tau \rangle_F - \langle s, P^* \tau \rangle_E = \int_{\partial U} g_F( i\sigma_1(P)(dx) s, \tau) \operatorname{dVol}_{\partial U}
\end{equation}
where $\sigma_1(P)(dx)$ is the principal symbol of the operator $P$ evaluated at the differential of a boundary defining function $x$, where we have chosen $x$ such that the outward pointing unit normal vector to the boundary is given by $\partial_x$, and the corresponding co-vector is $dx$.

The following example shows that there are choices to be made when the Witt condition is not satisfied, illuminating how one has to be careful with integration by parts near the singularity and when formulating ($L^2$) Stokes theorems for sections in different domains.

\begin{example}
Consider the case of a cone over a two torus $X=C_x(\mathbb{T}^2)$ with wedge metric $dx^2+x^2(d\theta_1^2+d\theta_2^2)$. This is the truncated tangent cone of the suspension over the two torus $\Sigma_\phi(\mathbb{T}^2_{\theta_1,\theta_2})$ with the metric $d\phi^2+\sin^2(\phi)(d\theta_1^2+d\theta_2^2)$ (motivated by the round metric on the sphere).

Consider the forms $s=dx \wedge d\theta_1$ and $\tau=d\theta_1$.
It is easy to verify that these sections are in the null space of both $d$ and $d^*$, and are both $L^2$ bounded with respect to the measure correponding to the volume form $dx \wedge xd\theta_1 \wedge xd\theta_2$.
Then the Green-Stokes formula for $P=d^*$ in equation \eqref{equation_Greens_identity} instantiates to
\begin{equation}
    \langle d^*s, \tau \rangle_E - \langle s, d \tau \rangle_E = \int_{\mathbb{T}^2} g_E( i\sigma_1(d^*)(dx) s, \tau) dVol_{\partial X}
\end{equation}
where $E=\Lambda \prescript{a}{}{T^*X}$ is the adapted form bundle. The last integral is simply 
\begin{equation}
    \int_{\mathbb{T}^2} (\iota_{\partial_x} dx \wedge d\theta_1) \wedge (\star_{\mathbb{T}^2} d\theta_1)=\int_{\mathbb{T}^2} (d\theta_1) \wedge d\theta_2= Vol(\mathbb{T}^2) \neq 0
\end{equation}
showing that both sections cannot simultaneously be in a self-adjoint extension of $d+\delta$. An ideal boundary condition at $x=0$ will exclude one of these two sections from the domain of $d+\delta$.   
\end{example}

Thus the operator $d$ with an extension that includes both sections is not nilpotent. We refer to \cite{bandara2024geometric} for an investigation of such phenomena for smooth manifolds with non-smooth geometric data.

The following space arises naturally as the tangent cone at the singular point of the cusp curve $y^2-x^3=0$ in $\mathbb{C}^2$ studied in \cite{jayasinghe2024holomorphic,jayasinghe2023l2}. There the Witt condition is satisfied, but the Dirac type operator is not essentially self-adjoint.

\begin{example}
\label{Example_Complex_cone_4pi}
Consider the case of a cone over a circle of length $4\pi$, $X=C_x(S^1_{4\pi})$ with conic metric $dx^2+x^2(d\theta^2)$. This is K\"ahler with the complex structure that sends $dx$ to $xd\theta$, and $xd\theta$ to $-dx$. We define $z:=xe^{i\theta}$, where we emphasize that $\theta \in S^1_{4\pi}$.

Consider the sections of the Dolbeault complex on this space.
We observe that $\tau=z^{-1/2}$ is in the null space of the $\overline{\partial}$ operator, and is $L^2$ bounded with respect to the measure corresponding to the volume form $dx \wedge xd\theta$.
We also note that $s=\overline{z}^{-1/2}\overline{dz}=\overline{z}^{-1/2}e^{-i\theta}(dx-ixd\theta)$ is in the null space of the $\overline{\partial}^*$ operator and is $L^2$ bounded.
Then the Green-Stokes formula for $P=\overline{\partial}^*$ 
is
\begin{equation}
    \langle \overline{\partial}^*s, \tau \rangle_E - \langle s, \overline{\partial} \tau \rangle_E = \int_{0}^{4\pi} g_E( i\sigma_1(\overline{\partial}^*)(dx) s, \tau) d\theta
\end{equation}
where $E=\Lambda \prescript{a}{}{T^*X^{0,1}}$ is the wedge anti-holomorphic form bundle. 
The last integral is simply 
\begin{equation}
\label{equation_inner_product_0}
     i\int_{0}^{4\pi} g_E(\overline{z}^{-1/2}e^{-i\theta},\overline{z^{-1/2}}) d\theta
     =i\int_{0}^{4\pi} 1 d\theta=4\pi i \neq 0,
\end{equation}
where we used the fact that 
$i\sigma_1(\overline{\partial}^*)(dx) \overline{dz}=ie^{-i\theta}$
and that
\begin{equation}
    A=g_E(\overline{z}^{-1/2} e^{-i\theta} , \overline{z^{-1/2}})=x^{-1/2} e^{-i\theta/2} \wedge \star (x^{-1/2} e^{+i\theta/2})
\end{equation}
using the sesquilinearity of the inner product, and that $A=x^{-1} \star 1 =x^{-1} dx \wedge xd\theta=dx \wedge d\theta$, which induces the surface measure $d\theta$ on the link at $x=0$.
This shows that both sections $s,\tau$ cannot simultaneously be in a self-adjoint extension of $\overline{\partial}+\overline{\partial}^*$.
\end{example}

It is well known that the eigenvalues of a Laplace-type operator on a smooth manifold grows smaller as the volume of $Z$ grows larger. The simplest example is the circle with the flat metric $d\theta^2$ where $\theta \in S^1_L$, where the Laplacian has eigenvalues that are proportional to the square of the length $L$ of the circle, and where the eigenvalues of the corresponding Dirac operator $-i\partial_{\theta}$ are proportional to the length of the circle. The small eigenvalues of the Dirac operators on the links correspond to the choices of self-adjoint extensions of the operator on $X$ and we refer to \cite{jayasinghe2024holomorphic} for more details on the domains of Dolbeault complexes in the conic (wedge) setting.

We study complexes $\mathcal{P}_{I}=(H, \mathcal{D}_{I}(P),P)$ where $H=L^2(X;E)$ is the Hilbert space of sections on a vector bundle $E$ on $X$, and $I$ indicates one choice of domain out of several which we shall introduce now.
We denote the maximal domain for $P$ by $I=(\max,N)$ and the minimal domain for $P$ by $I=(\min,D)$. We also define the domain
\begin{multline}
    \mathcal{D}_{\min, N}(P_X)= \{ u \in \mathcal{D}_{\max,N}(P) : \exists (u_n) \subseteq \mathcal{D}_{\max,N}(P) \text{ s.t. } (u_n)|_{X_2} \subseteq {C}^{\infty}_c(\mathring{X_2};E) \\
    \text{ and }   u_n \rightarrow u \text{ and } (P_{X_2} u_n) \text{ is } L^2(X_2;E) \text{    Cauchy} \},    
\end{multline}
where $X_2$ is the space $X_2=[0,1/2]\times Z$ equipped with the restriction of the metric on $X$ to the truncation $X_2$. Roughly, this is picking the domain to be minimal near the singularity while allowing sections in the maximal domain to be present near the boundary.
We also define 
\begin{equation}
    \mathcal{D}_{\max,D}(P_X):= (\mathcal{D}_{\min,D}(P^*_X))^*
\end{equation}
where $P_X^*$ is the formal adjoint of $P_X$.
The notation of $N/D$ here corresponds to the generalized Neumann/Dirichlet boundary conditions for the associated Dirac operators of these complexes introduced in \eqref{boundary_conditions_for_Dirac}, and this can be understood using the Green-Stokes formula in \eqref{equation_Greens_identity}, near the boundary at $x=1$.

If we consider the domain for the Dirac operator of the complex $\mathcal{P}_{\max,N}$ given by \eqref{Domain_Dirac_first}, we take the intersection of $\mathcal{D}_{\max,N}(P)$ and $\mathcal{D}_{\min,D}(P^*)$.
Inclusion of a section in the maximal domain imposes regularity conditions near $x=1$, while the minimal domain is more restrictive.
The Green-Stokes condition shows that given a section $s \in \mathcal{D}_{\max,N}(P)$ and a section $\tau \in \mathcal{D}_{\min,D}(P^*)$, the boundary integral in \eqref{equation_Greens_identity} must vanish since the domains are adjoints. The restriction here comes from the condition on $\tau$ that $\sigma_1(P^*)(dx)\tau=0$ at the boundary, which can be deduced from the boundary integral in \eqref{equation_Greens_identity} and the fact that the adjoint of $\sigma(P)(dx)$ is $\sigma(P^*)(dx)$, and this is precisely the generalized Neumann condition for the operator $P$.

Similarly we see that the sections in the domains for the Dirac operator of complexes labeled with $B=D$ satisfy the generalized Dirichlet conditions.

\section{Laplace-type operators}
\label{subsection_Laplacians}

For singular warped product metrics of the form in \eqref{warped_product_metric}
on $X$, given an orthonormal frame $\{e_j^1,...,e_j^n\}$ of the cotangent bundle on a chart on a factor $Z_j$ of the link $Z$, we saw in Subsection \ref{Subsection_Dirac_operators_detail} that we can construct operators
\begin{equation}
\label{equation_Dirac_intermediary_1_second}
    {D}_{Z_j}=\sum_i cl(e_j^{i}) \nabla^E_{e_j^{i \#}}, \quad \widetilde{D}_{Z_j}=\sum_i cl(f_j(x)  e_j^i) \nabla^E_{e_j^{i \#}}
\end{equation}
and the Dirac-type operators
\begin{equation}
\label{equation_Dirac_intermediary_2_second}
    D_{X}=D_1+D_2, \quad \text{ where   }  D_1=cl(dx) \nabla^E_{\partial_x}, \quad D_2=\sum_{j=1}^{m} \frac{1}{f_j}\widetilde{D}_{Z_j}.
\end{equation}

\begin{remark}
\label{Remark_Laplacian_craziness}
The associated Laplace-type operator is 
\begin{equation}
    \Delta_X:=D^2_X=D_1^2+D_2^2+\{D_1,D_2\}
\end{equation}
and \textbf{one of the main outcomes of the work in this subsection will be to construct an orthonormal basis of eigensections for $\Delta_X$ (and for the spectrum even without boundary conditions) on which it suffices to study what are \textit{essentially} eigensections of $D_1^2+D_2^2$. We will clarify this now.
}
\end{remark}

Given adapted co-vector fields $u, v \in \Omega^{0}(X;\mathbb{C})$, the Clifford relation 
\begin{equation}
    \{cl(u), cl(v)\}=cl(u)cl(v)+cl(v)cl(u)=-2g(u, v)
\end{equation}
can be used to simplify the Laplace-type operator
\begin{equation}
    D_X^2=D_1^2+ \sum_j \frac{1}{f_j^2(x)} {D}^2_{Z_j}+ \sum_j (cl(dx)\nabla^E_{\partial_x} [\frac{1}{f_j(x)} \widetilde{D}_{Z_j}])+\sum_j (\frac{1}{f_j(x)}) [cl(dx)\widetilde{D}_{Z_j}+cl(dx)\widetilde{D}_{Z_j}] \nabla^E_{\partial_x}.
\end{equation}
Here we used the fact that $cl(dx)^2=-Id$ is a consequence of the Clifford relation which can also be used to show that $[cl(dx)\widetilde{D}_{Z_j}+cl(dx)\widetilde{D}_{Z_j}]=0$ and we obtain
\begin{equation}
\label{equation_to_say_something_52}
    \Delta_X=D_1^2 + \sum_j \frac{1}{f_j^2(x)} [cl(dx){D}_{Z_j}]^2+\sum_j cl(dx)\nabla^E_{\partial_x}(\frac{1}{f_j(x)} \widetilde{D}_{Z_j})
\end{equation}
where we also used the Clifford relation to see that $[cl(dx)\widetilde{D}_{Z_j}]^2=\widetilde{D}_{Z_j}^2$.
Given a form $\phi_{(\mu_1,...,\mu_m)}$ on $Z$ of multi-degree $\underline{k}=(k_1,...,k_m)$, with coefficients twisted by a flat bundle $E$ which satisfies 
\begin{equation}
\label{equation_ansatz_0}
    cl(dx)D_{Z_j}\phi_{(\mu_1,...,\mu_m)}=\mu_j \phi_{(\mu_1,...,\mu_m)}
\end{equation}
for each $j$, we can seek to solve the equation
\begin{equation}
\label{equation_ansatz_1}
    \Delta_X u(x)\phi_{(\mu_1,...,\mu_m)}= \lambda^2 u(x) \phi_{(\mu_1,...,\mu_m)}.
\end{equation}
Given eigensections $\{\phi_j\}_{j=1}^m$ for $\Delta_{Z_j}$ with eigenvalues $\mu_j$, we can construct $\phi_{(\mu_1,...,\mu_m)}=\wedge_j \phi_j$. In fact if all sections $\phi_j$ are exact (respectively co-exact) or harmonic, then the sections $\phi_{(\mu_1,...,\mu_m)}$ are exact (respectively co-exact), and $\phi_{(\mu_1,...,\mu_m)}$ is harmonic iff all factors $\phi_j$ are harmonic.
Given such a $d_Z$ exact form $u_1(x)\psi_1$ we can write
\begin{equation}
    \star_X u_1(x)\psi_1=F(\psi_1)u_1(x)dx \wedge (\star_Z \psi_1)=u_2(x) dx \wedge \psi_2, \quad u_2(x)=F(\psi_1)u_1(x),\quad \psi_2=\star_Z \psi_1
\end{equation}
where $\psi_2$ is a $d_Z$ co-exact form since $\star_Z$ intertwines $d_Z$ and $\delta_Z$ up to a sign (c.f. Remark \ref{Remark_SUSY_Hodge_star_rescalings}).

We showed how to express $D_1^2$ on forms in Subsection \ref{subsection_Hodge_Dirac_operator}, and this can be used to expand the operator $\Delta_X$. However there is a further simplification using the supersymmetry where the last term in \eqref{equation_to_say_something_52} avoided in the case of $\mathcal{N}=2$ SUSY, where the operator $Q_{-}=P$ 
is nilpotent (see Remark \ref{Remark_independence_properties}).
This was first shown by Cheeger in the conic setting for the Hodge Laplacian, and more generally for Dirac operators we study here for the conic setting in \cite{jayasinghe2023l2}.
As in those articles, it will be achieved by constructing an ansatz for the eigensections of the Laplace-type operators on forms (for de Rham and Dolbeault complexes). We interpret this generalization of the ansatz of Cheeger in the following manner.

\subsection{Generalized Cheeger ansatz}
\label{subsection_Cheegers_ansatz}

We observe that the summands $P, P^*$ of $D=D_1+D_2$ splits into $P=P_1+P_2$ and $P^*=P_1^*+P_2^*$, compatible with the splitting of $D$ in the sense that $D_i=P_i+P^*_i$ for $i=1,2$.
The space of sections sections consists of 6 types of sections. Those that are co-exact with respect to both $P_1, P_2$ are type $1$ while the image of such forms under $P$ are of type $2$. Those that are exact with respect to both $P_1, P_2$ are type $4$ while the image of such forms under $P^*$ are of type $3$. If there is an additional operator $\star$ (which depends on the multi-degrees of the sections) intertwines $P$ and $P^*$ (up to a sign), then $\star_X$ intertwines the sections of types $1,4$ and $2,3$. The sections of types $E,O$ are those in the null space of $P_2+P_2^*$, and sections of type $E$ are $P_1$ co-exact while those of type $O$ are $P_1$ exact.

Here we only look for an ansatz for the sections that have positive eigenvalues for the Laplacian $D^2$. The determination of sections in the null space of both $P_1+P_1^*$ and $P_2+P_2^*$ will be carried out separately in Subsection \ref{subsection_spec_theory_and_cohomology} where the SUSY again plays a key role.
We recall that $D_1=cl(dx)\nabla^E_{\partial_x}$ and its square acts differently on different types of sections, as expressed in equations \eqref{equation_radial_derivatives_1} and \eqref{equation_radial_derivatives_2} and summarize the SL operators that arise for the de Rham and Dolbeault cases in the following remark.

\begin{remark}[Singular Sturm-Liouville eigenvalue problems]
\label{SUSY_ansatz_key_1}
Thus for forms as in the discussion above where both $\psi_1, \psi_2$ satisfy equation \eqref{equation_ansatz_0} and are of a fixed multi-degree, we will show in the subsections below that we can reduce the study of $\Delta_X (u_1(x)\psi_1)$ to the study of eigensections of $\Delta_Z$ of fixed multi-degrees and of the singular SL eigenvalue problem
\begin{equation}
\label{equation_Sturm_Liouville_1_primary}
    L_1[u](x)=-\frac{1}{F(\psi_1)}(F(\psi_1) u'_1)'+\sum_j \frac{\mu_j^2}{f_j^2}u_1=\lambda^2u_1(x).
\end{equation}
In Section \ref{section_Sturm_Liouville} we will study unitarily equivalent forms of this eigenvalue problem.
\end{remark}

\begin{remark}
\label{Remark_how_generally_is_this_possible_Question}
To deduce the structure of the operators given above from the ansatz described above, it suffices to assume that there is a further direct sum decomposition $\oplus_j S_j$ of the space of sections for which the operator $\star_X$ that intertwines $P$ and its adjoint acts by sending sections $\psi$ in a given summand $S_j$ that are in the null space of $P_1$ to those of another $S_l$ where there is a fixed factor $F(x)$ such that $F^{-1}(\psi) \star_X \psi$ is in the null space of $P^*_1$. This is the case for the forms we study, and we present the details in such cases for the benefit of the reader.
\end{remark}

Since the Laplace-type operator $D_X^2$ is defined by composing the Dirac operators on their domains, the boundary condition \ref{boundary_conditions_for_Dirac} leads to the \textbf{generalized Neumann/Dirichlet boundary conditions} for the Laplace-type operators (which we also call \textbf{absolute/relative boundary conditions}) to be
\begin{equation}
\label{generalized_Neumann_boundary_conditions_for_Laplace}
    \sigma(P^*)(dx)u|_{x=1} =0, \quad \sigma(P^*)(dx)Pu|_{x=1} =0
\end{equation}
and 
\begin{equation}
\label{generalized_Dirichlet_boundary_conditions_for_Laplace}
    \sigma(P)(dx)u|_{x=1} =0, \quad \sigma(P)(dx)P^*u|_{x=1} =0
\end{equation}
respectively, where we call the first and the second boundary conditions in each as the $0$-th order and $1$-st order boundary conditions respectively.

We will see that on functions $u(x)F^{-1}(x)$ where $F(x)$ is a non-vanishing function at $x=1$, the $0$-th order conditions reduce to the same boundary conditions are those for $u(x)$ (these will correspond to Dirichlet conditions for the SL problems.
The reason we call the generalized Nemann conditions as such is because in the Dolbeault case, they correspond to the del-bar Neumann condition.

\subsection{Eigensection ansatz: Hodge Laplacian}
\label{subsection_eigensection_ansatz_Hodge}

Here we generalize an ansatz developed by Cheeger in \cite{cheeger1983spectral} to study the spectrum of the Hodge Laplacian on cones to the spaces we study here. We also discuss a version for Dolbeault Laplacians generalizing \cite{jayasinghe2023l2}. While our construction extends trivially for sections with coefficients in flat/holomorphic bundles $E$, we will do this for it for the untwisted case.

\begin{remark}
    If a given K\"ahler action lifts to one on $E$, then the Witten deformed Dolbeault operator is $P_{\varepsilon}=e^{-\varepsilon h}\overline{\partial}_E e^{+\varepsilon h}$, and we have a deformed complex $\mathcal{P}_{W,B,\varepsilon}=(L^2\Omega^{p,\cdot}(X;E), \mathcal{D}_{W,B,\varepsilon}(P_{\varepsilon}),P_{\varepsilon})$ and the corresponding Laplace-type operator $\Box_{\varepsilon}=P_{\varepsilon}P^*_{\varepsilon}+P^*_{\varepsilon}P_{\varepsilon}$ which commutes with $\Delta_Z$ and $\sqrt{-1} L_V$ where the latter is the infinitesimal action of the Hamiltonian vector field $V$ on sections (in the case of a trivial bundle this is the Lie derivative).
\end{remark}

We assume that the operators $P=d / \overline{\partial}$ and $P^*$ act as supersymmetry operators restricted to the $P$ co-exact and exact operators respectively, following Cheeger.
Indeed this is the case for the de Rham complex on $Z$ and the transverse Dolbeault complex discussed in Remark \ref{remark_d_vs_adjoint_rescaling_Dolbeault}.

We consider sections $u(x)\phi$,
where the sections $\phi=\phi_{(\mu_1,...,\mu_m)}$ satisfy
\begin{equation}
\label{equation_ansatz_0_Hodge}
    D_{Z_j}\phi_{(\mu_1,...,\mu_m)}=\mu_j \phi_{(\mu_1,...,\mu_m)}
\end{equation}
for each $j$ (eigensections of each $\Delta_{Z_j}$ with eigenvalue $\mu_j$), pulled back from $Z$ to $X$ by the projection onto a link at any fixed $x$. We further assume that $\phi$ is $d_Z$ co-exact and that $d_Z \phi=\psi$ for an exact eigensection $\psi$ of $\Delta_Z$. Both $\phi$ and $\psi$ have the same eigenvalues $\mu^2_j$ for each $\Delta_{Z_j}$.
Here $u(x)$ is a function on $(0,1)$.
We say that the forms 
\begin{equation}
\label{1_type}
    u_1(x)\phi,
\end{equation}
\begin{equation}
\label{2_type}
    u_1(x)d_Z\phi+u_1'(x)dx \wedge \phi,
\end{equation}
\begin{equation}
\label{3_type}
    \delta_1 (u_2(x)dx \wedge d_Z\phi)-u_2(x)dx \wedge \sum_j \frac{1}{f^2_j(x)}\delta_{Z_j} d_Z \phi,
\end{equation}
\begin{equation}
\label{4_type}
    u_2(x)dx \wedge d_Z \phi,
\end{equation}
\begin{equation}
\label{E_type}
    u_3(x)h,
\end{equation}
\begin{equation}
\label{O_type}
    u_3'(x)dx \wedge h
\end{equation}
are of the types $1, 2, 3, 4, E$ and $O$ respectively, where the primed notation indicates differentiation with respect to $x$. Here the sections $h^k$ are in the null space of each $\Delta_{Z_j}$.
Since
\begin{equation}
    d_X=d_Z+dx \wedge \partial_x , \quad \delta_X= \delta_1 +\sum_j \frac{1}{f^2_j(x)}\delta_{Z_j},
\end{equation}
where $\delta_1$ is the operator in \eqref{equation_definition_delta_1},
we see that $d_X$ maps forms of types $1, E$ to those of types $2,O$ respectively. 
When the sections of a given type are in addition eigensections of the operator $\Delta_X$ with eigenvalue $\lambda^2 >0$, the operator $\delta_X$ is an inverse operator for $d_X$ for the forms of types $2, O$, while also mapping eigensections of $\Delta_X$ of type 4 to those of type $3$. We also see that eigensections of types $1,4$ are intertwined by the operator $\star_X$.

\begin{remark}
\label{SUSY_ansatz_key_2}
Here we use the fact that $dx\wedge \delta_Z= -\delta_Z \wedge dx$, and that
\begin{equation}
    -u(x)dx \wedge \sum_j \frac{1}{f^2_j(x)}\delta_{Z_j} d_Z \phi^k=-u(x)dx \wedge \sum_j \frac{1}{f^2_j(x)}\mu_j^2 \phi^k=P^*(u(x)dx \wedge d_Z \phi^{k})
\end{equation}
where $P^*=\delta_X$, which makes it easier to match this ansatz with that given for the Dolbeault complex in the next subsection.
\end{remark}

It is easy to check that an eigensection of $\Delta_X$ of type $1$ (and $2$ by the supersymmetry) then has a factor $u(x)$ which satisfies \eqref{equation_Sturm_Liouville_1_primary},
simply by applying $\delta_Xd_X$ to the co-closed form, where $\mu^2_j$ are the eigenvalues of the operators $\Delta_{Z_j}$ as introduced in \eqref{equation_ansatz_0}. The case of eigensections of type $E$ (and $O$ by the supersymmetry) yields the same operator, where each $\mu_j$ vanishes since the sections $h$ are $\Delta_{Z_j}$ harmonic.
Similarly one can check that given an eigensection of $\Delta_X$ of type $4$ (and $3$ by the supersymmetry), the function $u(x)$ satisfies the same equation, again by applying $d_X\delta_X$ to the forms of type $4$.

Since $\sigma(P^*)(dx)=\iota_{\partial_x}$ for $P=d$, the generalized Neumann/absolute boundary conditions for the Laplace-type operator in the de Rham case case correspond to $u'(x)=0$ ($(uF^{-1}(\phi))'(x)=0$ for rescaled type 3 and 4 eigensections) (Neumann/Robin SL conditions) for the forms of types $1,2,E,O$, $u(x)=0$ (Dirichlet SL conditions) for the forms of types $3,4$.

Since $\sigma(P)(dx)=dx \wedge$, the generalized Dirichlet/relative boundary conditions correspond to $u(x)=0$ (Dirichlet SL conditions) for the forms of types $1,2,E,O$, $u'(x)=0$ (Robin SL conditions) for the forms of types $3,4$.

\begin{remark} [Poincar\'e duality]
\label{Remark_Poincare_duality_in_SUSY}
It is well known that the Hodge-star operator intertwines the absolute and relative boundary conditions, and in this case is precisely a \textbf{Dirichlet to Neumann operator} for the SL problems (on unrescaled sections). We observe that the Hodge star operator intertwines the forms of types $1,4$, $2,3$ and $E,O$ respectively, and thus it suffices to study Dirichlet type problems on the complex to get an orthonormal basis of eigensections for all the types of forms, using Hodge star and the operators $P,P^*$.
\end{remark}

\subsection{Eigensection ansatz: Dolbeault Laplacian}
\label{subsection_eigensection_ansatz_Dolbeault}

In the case when $X$ is K\"ahler, since the Dolbeault Laplacian is equal to the Hodge Laplacian upto a factor of $2$, the ansatz above suffices to study the eigensections (up to restricting to $(p,q)$ forms for a fixed $p$). The following construction shows how we can study a similar ansatz in the general complex setting, generalizing a construction in \cite{jayasinghe2023l2}.
We saw in Subsection \ref{subsection_Dolbeault_spin_c_Dirac} that the Dolbeault complex is equipped with transverse Dolbeault operators $P_2$ (trivial in when the space $X$ is of real dimension $2$), and its adjoint $P_2^*$.

We consider sections $u(x)\phi$, similar to the ansatz \eqref{equation_ansatz_0}, where $u$ is a function on $(0,1)$. We further assume that $\phi=P_2^* \psi$ for another eigensection $\psi$ of the transverse Laplacian $\Delta_T=P_2 P_2^* + P_2^* P_2$, that satisfies $\nabla_{\partial_x} \phi =0$.
In addition we demand that these sections are eigensections of $\Delta_Z$, that are co-exact with respect to $P_2$, which implies that $P_2 \phi=\psi$, showing that $\psi$ is another eigensection of $\Delta_Z$. Here we use the fact that $D_X=cl(dx) \nabla^E_{\partial_x} + cl(J(dx)) \nabla^E_{{J(\partial_x)}}+ D_T$ (with the rescaled transverse Dirac operator $D_T:=(P_2 +P_2^*)$) where the second and third summands commute.
Both $\phi$ and $\psi$ have the same eigenvalues $\mu^2_j$ for each $\Delta_{Z_j}$.
We say that the forms 
\begin{equation}
\label{1_type_Dol}
    u(x)\phi,
\end{equation}
\begin{equation}
\label{2_type_Dol}
    u(x)P_2\phi+ P_1 (u(x)\phi),
\end{equation}
\begin{equation}
\label{3_type_Dol}
    P^*[u(x) \beta] \wedge P_2 \phi,
\end{equation}
\begin{equation}
\label{4_type_Dol}
    u(x) \beta \wedge P_2 \phi,
\end{equation}
\begin{equation}
\label{E_type_Dol}
    u(x)h,
\end{equation}
\begin{equation}
\label{O_type_Dol}
    P_1 (u(x) h)
\end{equation}
are of the types $1, 2, 3, 4, E$ and $O$ respectively, where the primed notation indicates differentiation with respect to $x$ and where $\beta:=(dx -\sqrt{-1}J(dx))$ (note that $\sigma(P^*)(dx)=\beta \wedge$). Here $h$ are sections in the null space of the transverse Laplacian $P_2 P_2^* + P_2^* P_2$.

Then by arguing similar to the de Rham case one can see that the forms of types $1,3,E$ map to those of types $2,4,O$ under $P$, and vice versa under $P^*$.

Since $\sigma(P^*)(dx)=\iota_{V}$ for $P=\overline{\partial}_E$ where $V=\overline{\beta^{\#}}$, the generalized Neumann/absolute boundary conditions for the Laplace-type operator in the Dolbeault case correspond to $P_1(u(x)\phi^q)=0$ (Robin SL conditions) for the forms of types $1,2,E$, $u(x)=0$ (Dirichlet SL conditions) for the forms of types $3,4,O$.
Since $\sigma(P)(dx)=\beta \wedge$, the generalized Dirichlet/relative boundary conditions correspond to $u(x)=0$ (Dirichlet SL conditions) for the forms of types $1,2,E$, $P_1(F^{-1}(\phi)u(x)\phi)=0$ (Robin SL conditions) for the forms of types $3,4,O$.

\begin{remark}[Serre duality]
\label{Remark_Serre_duality_in_SUSY}
The Hodge star operator implements Serre duality in the complex setting, mapping forms in $L^2\Omega^{p,q}(X;E)$ to $L^2\Omega^{n-p,n-q}(X;E^*)$ where $E$ is a Hermitian vector bundle with dual $E^*$. However the generalized Neumann and generalized Dirichlet boundary conditions are intertwined, and the forms of types $1,2$, $3,4$ and $E,O$ are intertwined by the operator between these complexes. 
We observe that it suffices to study Dirichlet type SL problems on the complex and the Serre dual complex in order to get an orthonormal basis of eigensections for all forms on both complexes using the operators $P,P^*$ and the Hodge star operator. We refer to \cite{jayasinghe2024holomorphic} for more details on Serre duality in the conic setting (the details extend to this setting mutatis mutandis).
\end{remark}

\subsection{More general operators, and comparison with other approaches.}
\label{subsection_general_Dirac_type_other_approaches}

We consider some technical differences in this work, compared to other articles. 
We used the independence properties of the nilpotent operator $P$
to express the Laplace-type operator as in equation \eqref{equation_to_say_something_52}.
We then constructed an ansatz which reduces the problem to the study of $D_1^2+D_2^2$, where 
\begin{equation}
    D_1=cl(dx) \nabla^E_{\partial{x}}, \quad D_2=D_X-D_1,
\end{equation}
even if $\{D_1,D_2\} \neq 0$.

A different approach is used in \cite{Albin_2017_index} in the conic case (c.f.\cite{Albin_hodge_theory_cheeger_spaces} for the signature operator). There the adapted co-tangent bundle specializes to the wedge co-tangent bundle, and wedge metrics are conformal to a \textit{complete edge metric}, where the rescaled edge vector fields that locally span the edge tangent bundle are closed under the Lie bracket. 
For the operator on wedge forms, they show that the scaling behaviour (in the radial variable $x$) of the terms which occur in the commutator $\{D_1,D_2\}$ is similar to the scaling of $D_2$, simplifying the analysis. \textbf{Since $f_j(x)=x$ in the conic/wedge setting, $f'_j(x)=1$, which plays an important role in the simplification there.}

The results in \cite{Albin_2017_index} were used in \cite{jayasinghe2023l2} to show that Dirac operators compatible with spin$^\mathbb{C}$ structures in the presence of a compatible almost complex structure in the local models could be equipped with a generalization of the $\overline{\partial}$-Neumann operators used in the Dolbeault setting, utilizing work in \cite{epstein2006subelliptic,EpsteinSubellipticSpinc3_2007} where in particular it is shown that the spinors can be identified with sections in twisted forms of a fixed degree $p$ with the $(p,q)$ decomposition given by the choice of almost complex structure. Extending the work in this article to such complexes with only $\mathcal{N}=1$ SUSY would pave the way for various applications.

\subsection{Witten deformed Laplace-type/ semi-classical Schr\"odinger operators}
\label{subsection_Witten_deformed_Laplace}

Here we build on Subsection \ref{Subsection_Witten_deformed_ops}  and the references given there.
The deformed Laplace-type operators in the case of the de Rham complex $\Delta_{X,\varepsilon}=D^2_{\varepsilon}$ can be written as
\begin{equation}
    \Delta_{X,\varepsilon}=\Delta_X+\varepsilon K +\varepsilon^2 |dh|^2
\end{equation}
where $K=\{ D, \widehat{cl}(dh)\}$.
Since $h$ is a function of the radial variable $x$, and using the property that the anti-commutator $\{cl(\theta_1), \widehat{cl}(\theta_2)\}$ vanishes for any two forms $\theta_1,\theta_2$ we can simplify this term to
\begin{equation}
    K=cl(dx)\nabla_{\partial_x} \widehat{cl}(dh) + \widehat{cl}(dh) cl(dx) \nabla_{\partial_x}
\end{equation}
which can be simplified to 
\begin{equation}
    K=cl(dx) h''(x) \widehat{cl}(dx) + \{cl(dx),\widehat{cl}(dh)\} \nabla_{\partial_x}=cl(dx)\widehat{cl}(dx) h''(x)
\end{equation}
and the definitions of the Clifford operators can be used to see that $K=h''(x)M$ where $M=+1$ on normal forms (forms $v$ that satisfy $\sigma(d)(dx)v=0$, i.e. can be written as $dx \wedge v$ for a form $v$) and $M=-1$ on tangential forms (forms $v$ which satisfy $\sigma(d^*)(dx)v=0$ on $X$).
In the case of the Dolbeault complex, we denote by $V$ the K\"ahler Hamiltonian vector field, and the deformed Laplace type operator by
\begin{equation}
\label{equation_Laplacians_deformed_de_Rham_Dolbeault_relation}
    \Box_{X,\varepsilon}:=D^2_{\varepsilon}=\frac{1}{2} \Delta_{X,\varepsilon} +\varepsilon \sqrt{-1} L_V
\end{equation}
which generalizes the well known identity between the Hodge de Rham and Dolbeault Laplacians on smooth K\"ahler manifolds when $\varepsilon=0$. Here $\sqrt{-1} L_V$ commutes with $\Delta_{X,\varepsilon}$, and we can write $J(\partial_x)= \frac{1}{h'(x)} V$, where $V^{\flat}=\alpha$ as in Subsection \ref{Subsection_adapted_geometric_structures}.

Given functions $h(x)$ of the radial variable $x$, the deformed operators can be studied using separation variables similar to the undeformed case. Here it suffices to study $\Delta_{X,\varepsilon}$ deformation due to the relation \eqref{equation_Laplacians_deformed_de_Rham_Dolbeault_relation}. Using the ansatz in \eqref{equation_ansatz_0}, the forms of types $1,2,3,4,E,O$ can be used to find Sturm-Liouville operators similar to the deformed case.

Let us consider $D_{\varepsilon}=D+ \sqrt{-1} \varepsilon \widehat{cl}(dh)$, where $D$ is the Hodge-Dirac operator and $h$ is a function of $x$. 
Then we see that 
\begin{equation}
\label{equation_need_name_not_1}
    D^2_{\varepsilon}=D^2+\varepsilon h''(x)K +\varepsilon^2 (h'(x))^2,
\end{equation}
which simplifies to
\begin{equation}
\label{equation_Sturm_Liouville_2}
    L_{\varepsilon}[u]=L[u]+ K\varepsilon h''(x) u+[h'(x)]^2 \varepsilon^2 u =\lambda^2 u.
\end{equation}
where $K=\pm 1$, for the $u(x)$ factor of the ansatz \ref{equation_ansatz_0} where $L[u]=L_{\varepsilon=0}[u]$ are the Sturm-Liouville operators in Remark \eqref{SUSY_ansatz_key_1}. 

\begin{remark}
\label{Remark_signs_Morse_functions_dualities}
We observe that while $K=1$ for tangential forms and $K=-1$ for normal forms, the term where it appears in equation \eqref{equation_need_name_not_1} is $\varepsilon h''(x)K$. If we replace $h$ by $-h$, we see that this term changes by an overall minus sign while the other terms do not change in \eqref{equation_need_name_not_1}. This means that the SL operator in \eqref{equation_need_name_not_1} acting on tangential and normal forms switch when we substitute $h$ by $-h$. This is related to the fact that local (and even global) Poincar\'e duality corresponds to both the Hodge star operator (which intertwines normal and tangential forms) as well as switching $h$ with $-h$ in the Morse polynomials. This observation is key in formulating local Morse polynomials in singular settings, and we refer to \cite{jayasinghe2023l2,jayasinghe2024holomorphic} for details in the conic setting.
\end{remark}

\begin{example}
Let us assume $h=x^{c+1}/(c+1)$.
Then the Witten deformed Laplace-type operator can be shown to be
\begin{equation}
    D^2_{\varepsilon}=D^2+\varepsilon cx^{c-1}K +\varepsilon^2 x^{2c}
\end{equation}
where $K=\pm 1$ (depending on whether or not there is a $dx$ factor) and using an ansatz similary to equation \eqref{equation_ansatz_1}  the deformed Laplacian can be simplified to the study of a Sturm-Liouville equation of the form
\begin{equation}
\label{equation_Sturm_Liouville_2_example}
    L_{\varepsilon}[u]=L[u] \pm \varepsilon c x^{c-1}u+x^{2c} \varepsilon^2u =\lambda^2 u.
\end{equation}
\end{example}

The boundary conditions corresponding to the generalized Neumann/ Dirichlet conditions for the deformed complex simplify similar to the undeformed complex and yields the same types of SL boundary conditions for the operators.

\section{Singular Sturm Liouville problems}
\label{section_Sturm_Liouville}

We saw in Subsection \ref{subsection_Laplacians} that when the Laplace-type operators are restricted to forms of the $6$ types they correspond to certain singular SL operators, with Robin boundary conditions at $x=1$ and \textit{ideal} boundary conditions at $x=0$. Here we study the spectral theory for such operators, and show that a large class of metrics and Morse functions, K\"ahler Hamiltonian Morse functions are of discrete type (see also Subsection \ref{subsection_prior_results}).

We review SL theory following \cite{zettl2005sturm}.
A general Sturm-Liouville differential expression is one of the form
\begin{equation}
\label{equation_standard_SL_operator}
    \frac{1}{w(x)} \Big[ -\frac{d}{dx} p(x) \frac{d}{dx} u(x) +q(x) u(x) \Big]
\end{equation}
on an interval $(a,b) \subset \mathbb{R}$, where $p,q,r$ are considered to be Lebesgue measurable functions on $(a,b)$ where $r,q, 1/p \in L^1_{loc}((a,b))$, $r,p >0$ almost everywhere on $(a,b)$ and $q$ is real valued almost everywhere on $(a,b)$.
The equations 
\begin{equation}
    -\frac{d}{dx} p(x) \frac{d}{dx} u(x) +w(x)q(x) u(x) = \lambda^2 w(x)u(x)
\end{equation}
for $\lambda^2 \in \mathbb{C}$ are the Sturm-Liouville eigenvalue equations for such operators.
If the functions $p,q,r$ are simply smooth functions on $(a,b)$, not necessarily satisfying the integrability conditions $(r,q,1/p \in L^1_{loc}((a,b)))$ stated we will refer to them as \textbf{singular SL operators}.

If the function $1/p, q, w$ are in $L^1((d,b))$, with respect to the Lebesgue measure, for some $d>a$, then the endpoint $b$ (finite or infinite) is said to be a regular endpoint, as is the case in the problems we study in this article. We say that $u(x)$ satisfies a \textbf{regular Sturm-Liouville boundary condition} if it satisfies
\begin{equation}
\label{equation_Sturm_Liouville_boundary_condition_main}
    \gamma_1 u(1) + \gamma_2 u'(1)=0
\end{equation}
for some real constants $\gamma_1, \gamma_2$ where at least one is not zero and the boundary condition is determined by $(\gamma_1,\gamma_2)$. If $\gamma_2=0$ we call it a Dirichlet condition, and if $\gamma_1=0$ we say it is a Neumann condition, and if neither of $\gamma_1,\gamma_2$ vanish, we say it is a Robin condition.
When we consider the infinite interval $(0,\infty)$, we impose $L^2$ boundary conditions at $\infty$ (i.e., demand that the sections are in $L^2$ which imposes decay conditions near $x=\infty$).

The functions $1/p,q,w$ for the Sturm-Liouville problems that we study are not in $L^1((0,b))$ due to blow-up of the $L^1$ norm when approaching $x=0$, and in these situations the endpoint $0$ can generally be of several different types (including limit circle, limit point, oscillatory and non-oscillatory) and we refer to Definition 7.3.1 of \cite{zettl2005sturm} for more details, and \cite{stanfill2024sturmliouville} for some refinements of these types.
Moreover, there are many unitarily equivalent forms of SL operators, where for some forms, the functions $1/p,q,w$ may not satisfy the integrability conditions at $x=0$, as is the case for some equivalent forms of Bessel operators that arise in this article (see Remark \ref{remark_Bessel_not_fit}).

In Section \ref{subsection_Laplacians} we obtained the SL problem
\begin{equation}
\label{equation_Sturm_Liouville_1_primary_111}
    L_1[u](x)=-\frac{1}{F}(F u'_1)'+\sum_j \frac{\mu_j^2}{f_j^2}u_1=\lambda^2u_1(x).
\end{equation}
which is a symmetric operator on $L^2(I;Fdx)=F^{-1/2}L^2(I;dx)$ where $I=(0,1]$, and we can instead study the conjugated symmetric operator
$F^{1/2}L_1 F^{-1/2}$ on $L^2(I;dx)$ where it will be of the form $L[u]=(-\partial_x^2+V)u$ for some \textit{potential} function $V$ (since it is a symmetric operator with respect to the Lebesgue measure on an interval, no first order derivatives appear).

\begin{remark}
    This corresponds to the study of the unitarily equivalent operator $F(\psi_i)\Delta_X F^{-1}(\psi_i)$ 
    for any given $\psi_i$ with fixed multi-degree unitarily conjugate to $\Delta_X$, where the new operator acts on sections of the Hilbert space with volume measure $(dx dvol_Z)$.
\end{remark}

Thus we consider the action of $L_1[u]$ on $u=vF^{-1/2}$. Computing the derivative terms we get
\begin{equation}
    \frac{1}{F} [F(vF^{-1/2})']'=\frac{1}{F^{1/2}} v'' -\frac{1}{2}F^{-1/2}[F^{-1/2}F']' (v F^{-1/2})
\end{equation}
and we write the new operator
\begin{equation}
\label{equation_Sturm_Liouville_1_primary_conjugated}
    L[v](x)=-v''+Gv+\sum_j \frac{\mu_j^2}{f_j^2}v, \quad G(x):=\frac{1}{2}(F^{-1/2}\partial_x)^2(F),
\end{equation}
and in the notation of \eqref{equation_standard_SL_operator}, $L_1$ corresponds to $p=w=F$, while $L$ corresponds to $p=w=1$, and is a Schr\"odinger type operator.

\begin{remark}
\label{remark_Bessel_not_fit}
Given a swp metric where for each $j$, $f_j(x)=x^{c_j}$ for some real $c_j>0$, it is easy to check that each $F(\psi)=x^B$ for sections $\psi$ of a fixed multi-degree on $Z$ where $B \in \mathbb{R}$ depends on the multi-degree.
Then one can check that $G=B(B/2-1)(x^{-2})$ (not $L^2$ bounded near $x=0$ unless $B=0,2$) and \eqref{equation_Sturm_Liouville_1_primary_conjugated} corresponds to 
\begin{equation}
\label{Main_Sturm_Liouville_Schrodinger}
    L[v](x)=-v''+\frac{B(B/2-1)}{2x^2}v+\sum_j \frac{\mu_j^2}{f_j^2}v.
\end{equation}
This is unitarily equivalent to
\begin{equation}
\label{equation_Sturm_Liouville_1_specialized}
    L_1[u](x)=-\frac{1}{x^B}(x^B u')'+\sum_j \frac{\mu_j^2}{x^{c_j}}u=-u''-\frac{B}{x}u'+\sum_j \frac{\mu_j^2}{x^{c_j}}u
\end{equation}
which for the case when all $c_j=1$ is a Bessel operator (see Step 1.1.1 of the proof of Proposition 5.33 of \cite{jayasinghe2023l2}).
\end{remark}

\begin{definition}
\label{Definition_metrics_discrete_type}
If the SL problems with the domains induced by those for the Laplacians on $X$ have discrete spectrum, then we say that the singular warped product metric is of \textit{\textbf{discrete type}}.    
\end{definition}

In order to begin understanding choices of domains for such operators near $x=0$, we begin by stating Theorem X.10 of \cite{reed1975ii}.

\begin{theorem}
\label{Theorem_self_adjoint_Reed_Simon}
Let $V$ be a continuous and \textit{positive} near $x=0$. If $V(x) \geq 0.75 x^{-2}$ near $x=0$, then $-\partial_x^2+V$ is in the limit point case at zero. If for some $\delta>0$, $V(x) \leq (3/4 -\delta)x^{-2}$ near zero, then $-\partial_x^2+V$ is in the limit circle case.
\end{theorem}

This shows that for all SL problems where $f_j(x)=x^{c_j}$ for $c_j>1$ and for $\mu_j>0$, the operator is essentially self-adjoint. If there are $\Delta_Z$ harmonic sections, i.e. sections $\phi$ for which all $\mu_j=0$, then we are in the setting of sections of types $E,O$ and the positivity of the function $G$ is equivalent to $B \notin (0,2)$.
For essential self-adjointness it suffices that $B/2(B/2-1) >3/4$, equivalently when $B \in [-1,3]$.
This case can be studied similar to the wedge case studied in \cite{Pierre_Exposition}, since the terms arising are of the same form when all $\mu_j=0$.

We observe that when $B=0$ and all $\mu_j=0$, the adapted Witt condition is not satisfied. However when $B \neq 0$, the forms $u=\phi x^{-B/2}$ where $\phi$ is $\Delta_Z$ harmonic are not closed, and do not appear in the cohomology. We present the following example.

\begin{example}
\label{Example_Cone_torus_1}
Consider the cone over the 4 torus $C_x(T^4)$ with the metric $dx^2+x^2(\sum_{j=1}^4 d\theta_j^2)$. The section $\omega=x^{-1} d\theta_1$ is a $\Delta_X$ harmonic form. Note that $\delta d \omega$ is given by $\star^{-1} d \star_X (-x^{-2} dx \wedge d\theta_1)$ up to a sign, and $\star_X (-x^{-3} dx \wedge xd\theta_1)=-x^{-3} xd\theta_2 \wedge xd\theta_3 \wedge xd\theta_4$ is in the kernel of $d$. Thus $\delta d \omega=0$. Similarly it is easy to check that $\delta \omega=0$.
The form $\omega$ is $L^2$ bounded, but does not satisfy the generalized Neumann boundary condition since $\iota_{\partial_x} d\omega$ does not vanish at the boundary. Nor does it satisfy the generalized Dirichlet boundary condition since $dx \wedge \omega$ does not vanish at the boundary.
\end{example}

Theorem \ref{Theorem_self_adjoint_Reed_Simon} indicates that there is a quantization condition at $x=0$ on all sections in the domains we study, corresponding to the $L^2$ boundedness. The choices of self-adjoint boundary conditions for the operator $\Delta_X$ corresponds to stronger boundary conditions for the singular SL operator at $x=0$. Since the form of the operator when all $\mu_j$ vanish is similar to the case in \cite{Pierre_Exposition}, we refer the reader to that article for a more comprehensive exposition on determining choices of domains. We provide the following example which showcases the type of condition imposed on SL problems of types $E,O$.

\begin{example}
For the space in Example \ref{Example_Cone_torus_1}, the section $\omega_2= u(x) d\theta_1 \wedge d\theta_2$ is of type E when we set $u(x)=\sin(\lambda (x-c))$, and solves the equation $\Delta \omega_2= \lambda^2 \omega_2$, and the function $u(x)$ satisfies \eqref{equation_Sturm_Liouville_1_specialized} with $B=0$. If $\omega_2 \in \mathcal{D}_{\min}(d)$,
the inner product $\langle d\omega_2, \omega_3 \rangle_{\Lambda(\prescript{a}{}{T^*X}})$
must vanish at $x=0$ for any $\omega_3$ in the maximal domain of $\mathcal{D}(\delta)$. The sections $\omega_3=d\theta_1 \wedge d\theta_2$ and $dx \wedge \omega_3$ are both in the maximal domain of $D=d+\delta$ (in fact they are harmonic) and we see that 
\begin{equation}
    \langle d\omega_2, dx \wedge \omega_3 \rangle_{\Lambda(\prescript{a}{}{T^*X}})=u'(x) dx \wedge (dvol_T)
\end{equation}
which does not vanish at $x=0$ (even though it is integrable) unless $\cos(\lambda (-c))$, and together with a generic Robin boundary condition at $x=1$, the eigenvalues $\lambda$ are quantized.

For forms $u(x)d\theta_1 \wedge d\theta_2$ with some general function $u(x)$, the argument above shows that there are vanishing conditions for $u(x)$ and $u'(x)$ for eigensections of types $E,O$ imposed by the choices of domains that we study (in fact this extends for any choice of self-adjoint domain for $\Delta_X$).
\end{example}

We provide the following proof of compactness of resolvents for Schr\"odinger operators when $V \rightarrow \infty$ as $x \rightarrow 0$.

\begin{proposition}
\label{Propostion_Sturm_Liouville_1}
Consider the operator,
\begin{equation}
    H_\varepsilon=-\partial^2_x +V_1(x)+V_{2,\varepsilon}(x), \quad V_1(x):=G(x)+\sum_j \frac{\mu_j^2}{f_j^2(x)}, \quad V_{2,\varepsilon}(x):=\varepsilon^2 (h'(x))^2 \pm \varepsilon h''(x)
\end{equation}
acting on a self-adjoint domain $\mathcal{D} \subset L^2(I,dx)$, 
where $h(x)$ is a radial Morse functions as in Definition \ref{Definition_radial_Morse_functions} $f_j(x)=x^{c_j}$ for some $c_j>0$ for each $j$, and $G(x)=\frac{B(B/2-1)}{2x^2}$ for some finite $B$.
Then
\begin{enumerate}
    \item for any $\varepsilon \geq 0$, if $I=(0,L)$ is a finite interval with SL conditions at $x=L$ determining the domain, OR
    \item for any $\varepsilon >0$, if $I=(0,\infty)$ with $L^2$ boundary conditions at $\infty$ determining the domain,
\end{enumerate}
the operator $H_{\varepsilon}$ has discrete spectrum.
\end{proposition}

We deduce the result from the technical work in \cite{Jesus2018Wittens,Jesus2018Wittensgeneral}, built on the analysis of Dunkl harmonic oscillators in \cite{Jesus_Dunkl_2014,Jesus_Dunkl_2015}, as well as classical results in \cite{simon2008schrodinger,reed1975ii}. The results in \cite{Jesus2018Wittensgeneral} were proven for the choices of domains that we pick near $x=0$ in this article, for the de Rham complex. 
The SL operators (with domains) arising in the Dolbeault complex can be represented as SL operators \textit{for some} de Rham complex.

\begin{proof}
We first observe that for any $x>0$, the problem is of (singular in general) Sturm-Liouville type.
In the case where all $c_j \leq 1$ and when $h(x)=x^2$, discreteness of the spectrum was established for the minimal and maximal domains of the Witten deformed Hodge Laplacian in \cite{Jesus2018Wittensgeneral} (see Theorem 1.1 of \textit{loc. cit}). This covers the case when $L=\infty$ for such metrics and Morse functions. There are well known arguments which show that if the potential grows to $\infty$ as $x$ goes to $\infty$, then $\infty$ is in the limit point case and that there is discrete spectrum (see the results in \cite{simon2008schrodinger}). This shows we have the second numbered statement for the radial Morse functions in Definition \ref{Definition_radial_Morse_functions} since it guarantees that the term $V_{2,\varepsilon}$ grows as $x$ goes to $\infty$ for $\varepsilon>0$. Imposing SL boundary conditions at $x=L$, SL theory can be used to show that there is discrete spectrum for finite intervals for such metrics and Morse functions.

For the case where there is some $\mu_j>0$ for which $c_j>1$, we see that the potential $V_1$ grows faster than $x^{-2}$ near $x=0$, and by Theorem \ref{Theorem_self_adjoint_Reed_Simon} it is in the limit point case. In these cases, if all $\mu_j=0$, then the problem is equivalent to the case when all $c_j \leq 1$. Indeed for any $B$, $G(x)=\frac{B(B/2-1)}{x^2}$ can be realized as $G(x)$ in \eqref{equation_Sturm_Liouville_1_primary_conjugated} coming from a geometric problem by taking each link $Z_j$ to be a torus of sufficiently high dimension (which has harmonic sections in each degree), metrics with $f_j(x)=x^{c_j}$ and some form multi-degree for which the factor $F$ is $x^B$. Then applying the results of \cite{Jesus_Dunkl_2015,Jesus2018Wittensgeneral} to those geometric settings yields the necessary analytic results.
\end{proof}

In the case of conic (wedge) metrics, this was studied by Cheeger (see \cite[\S 3]{cheeger1983spectral}, \cite{cheeger1983hodge}), and he discusses how small eigenvalues $\mu$ for the Laplace-type operators on the links $Z$, correspond to choices of self-adjoint domains. The metrics on the links can be rescaled to make these eigenvalues larger until the corresponding Dirac-type operator is essentially self-adjoint. 
The SL operators change with such rescalings due to the eigenvalue changes, and one is in the limit point case of the Bessel operators in the conic case. Precise conditions in terms of the eigenvalues are worked out in \cite{Albin_2017_index} for general Dirac operators. The general adapted case is similar, but we do not study all choices of self-adjoint extensions in this article.

\section{Supersymmetric trace formulas}
\label{Section_Supersymmetric_trace_formulas}

We now prove the main results formulated in the introduction, first studying the spectral theory and cohomology for the complexes studied in this article, the key ingredients for which have been established in the previous sections using SUSY and which continues to play a key role in the proofs.
Next we study supertraces for the complexes, showing how Lefschetz supertraces simplify to supertraces on the cohomology groups of the complex, establishing holomorphicity properties of functions involved in renormalization and showing how in the case of group actions they correspond to Laurent series.

\subsection{Spectral theory and cohomology}
\label{subsection_spec_theory_and_cohomology}

The following is a restatement of Theorem \ref{Theorem_main_spectral_intro_version}.

\begin{theorem}
\label{Theorem_main_spectral}
Given a twisted de Rham/Dolbeault complex $\mathcal{P}_{W,B,\varepsilon}=(H, \mathcal{D}(P_\varepsilon),P_\varepsilon)
$ on $X$ with a swp metric of discrete type for the case where $\varepsilon=0$, of a Morse function of discrete type for the case when $\varepsilon>0$ (and is K\"ahler Hamiltonian in the case of a twisted Dolbeault complex). We denote by $\Delta_\varepsilon$ the Witten deformed Hodge Laplacian, where $\varepsilon=0$ corresponds to the undeformed complex. Then for any such complex,
\begin{enumerate}
    \item there exists an orthonormal basis $\{\psi_{n,k}\}_{n \in \mathbb{N}, k \in \mathbb{N}}$ where each $\psi_{n,k}=u_{n,k} \phi_{n}$ is an eigensection of $\Delta_{\varepsilon}$ (for each degree $q$) where $u_{n,k}$ is a form on $(0,1)_x$ and $\phi_{n}$ is an eigensection of $\Delta_Z$ with eigenvalue $\mu^2_n$ as described above.
    \item The sections $\psi_{n,k}$ are eigensections of $\Delta_{\varepsilon}^q + \Delta_Z^q$ in the same domain as that for $\Delta_{\varepsilon}^q$ and this operator has discrete spectrum, with eigenvalues $\lambda_{n,k}^2 + \mu^2_{n}$, where $\lambda_{(n,k)}^2$ are eigenvalues of the operator $\Delta_{\varepsilon}^q$.
    \item For each $P_\varepsilon$ co-exact eigensection $\psi_{n,k}$ with positive eigenvalue $\lambda^2_{n,k}$ for $\Delta_{\varepsilon}$, $\frac{1}{\lambda_{n,k}} P_\varepsilon\psi_{n,k}$ is an exact eigensection with the same eigenvalues for $\Delta_{\varepsilon}^{q}$ and $\Delta_Z$.
\end{enumerate}
\end{theorem}

\begin{proof} 
\noindent \textbf{Proof of item 1:} Given a twisted de Rham/ Dolbeault complex $\mathcal{P}_{W,B,\varepsilon}(X)$ (which is Witten deformed if $\varepsilon \neq 0$), we saw in Section \ref{subsection_Laplacians} that the we could use a separation of variables ansatz to reduce the problems to SL problems in Remark \ref{SUSY_ansatz_key_1} and which are singular SL problems with Dirichlet conditions at $x=1$ (see Remark \ref{Remark_Poincare_duality_in_SUSY} and \ref{Remark_Serre_duality_in_SUSY}), where the ansatz matches the structure of the eigensections in item 1 of the Theorem.
This proves item 1 of the Theorem since we assume the swp metrics and (the Morse functions/ K\"ahler Hamiltonians for $\varepsilon>0$) are of discrete type.

\noindent\textbf{Proof of items 2, 3 :} 
Item 2 is a consequence of item 1, together with the fact that $\Delta_{\varepsilon}$ commutes with $\Delta_Z$. 
Item 3 is a consequence of the construction of the basis (outlined in Subsection \ref{subsection_Laplacians}) using supersymmetry.
This proves the theorem.
\end{proof}

\begin{proposition}[Cohomology of the twisted de Rham complexes]
\label{proposition_cohomology_of_de_Rham}
Given a twisted de Rham complex $\mathcal{P}_{\max,N}(X)=(L^2\Omega^{\cdot}(X;E), \mathcal{D}_{\max,B}(d_E),d_E)$ where $X$ is equipped with a discrete swp metric, the harmonic representatives of the cohomology of the complex in degree $k$ are given by the vector space
\begin{equation}
    s\in \mathcal{H}^{k}(Z;E), ||s||_{L^2(X;\Lambda^*(\prescript{a}{}{T^*X}) \otimes E}) < \infty.
\end{equation}
where we denote by $\mathcal{H}^{k}(Z;E)$ the de Rham cohomology of the space $Z$.
\end{proposition}

\begin{proof}[Proof of Proposition \ref{proposition_cohomology_of_de_Rham}]
Since we have discrete spectrum as well as supersymmetry, we see that the kernel of the Laplace-type operator with the self-adjoint domain is isomorphic to the kernel of the corresponding Dirac operator.
Thus to find the harmonic sections we must solve the equations
\begin{equation}
\label{equation_intermediate_1}
    (P+P^*) (u\phi)=0, \quad \sigma(P^*)(dx) (u\phi)|_{x=1}=0, \quad \sigma(P^*)(dx)P (u\phi)|_{x=1}=0
\end{equation}
where $\phi$ is an eigensection (on the form bundle of $Z$) of the operator $D_Z$ with eigenvalue $\mu_j$ for each operator $D_{Z_j}$. However we also know by the Hodge-Kodaira decomposition that these harmonic forms are in the null space of $P$ (also in the null space of $P^*$), which is stronger than the last boundary condition.

Moreover given any element of the null space of $\Delta_X$, we can write it as a sum of normal and tangential forms (with and without $dx$ factors respectively) $s=u_n(x) dx \wedge \phi_n+u_t(x) \phi_t$, where similar to the ansatz in \eqref{equation_ansatz_0} this is also an eigensection of $\Delta_Z$. We will assume that $\phi_n$ has a fixed multi-degree (general harmonic forms are linear combinations of such ones).
Then one can check that the condition that $s$ is closed is equivalent to 
\begin{equation}
\label{equation_1_inter}
    u_t'(x)\phi_t =u_n(x) d_Z \phi_n, \quad u_t(x) d_Z \phi_t=0
\end{equation}
and the co-closed condition is equivalent to 
\begin{equation}
\label{equation_2_inter}
    u_t(x) \sum_j \frac{1}{f_j(x)^2}\delta_{Z_j} \phi_t = \frac{1}{F}(Fu_n)'(x) \phi_n, \quad u_n(x) \sum_j \frac{1}{f_j(x)^2}\delta_{Z_j} \phi_n=0.
\end{equation}
where we denote $F(\phi_n)$ by $F$ for brevity.
Unless one or both of $u_n, u_t$ vanishes identically, we can deduce that $\phi_t$ is $d_Z$ closed and $\phi_n$ is $\delta_Z$ closed, and moreover that $\delta_Z \phi_t=K_1\phi_n$, and $d_Z\phi_n=K_2\phi_t$ (since both $\phi_n$ and $\phi_t$ are independent of $x$) where $K_1,K_2$ are constants. This shows that $\phi_n$ and $\phi_t$ are super partners for the de Rham complex on $Z$, and we can arrange for $K_1=K_2$ to be the eigenvalues of $\sqrt{\Delta_Z}$ on $Z$ (perhaps up to factors appearing in $u_n(x), u_t(x)$) for the eigensections $s$.

The first boundary condition in \eqref{equation_intermediate_1} is equivalent to $u_n(1)=0$, and this together with \eqref{equation_1_inter} shows that $u_t'(1)=0$. Moreover from \eqref{equation_1_inter} we deduce that $u_t'(x)=K_2u_n(x)$.
This shows that $s$ has the structure of a type 2 form given in \eqref{2_type}, and can be expressed as $d(u_n(x)\phi_t)$, which contradicts the fact that it is in the cohomology group, isomorphic to $ker(d)/ im(d)$.
This reduces the possibilities to the cases when one of $u_n$ or $u_t$ vanishes while the other does not (that is, the forms are either normal or tangential).

Now we consider the case when $u_t=0$, and note that the closed and co-closed conditions imply that $\phi_n$ is harmonic, and $u_n'(x)=0$, which together with the first boundary condition (implying $u_n(1)=0$) implies that $u_n=0$.

In the final case where $u_n(x)=0$, the closed and co-closed conditions imply that $\phi_t$ is harmonic, and $u_t'(x)=0$, which corresponds precisely to the sections in the null space of $\Delta_Z$ up to constant factors. Taking into consideration the $L^2$ boundedness yields the result.
\end{proof}

\begin{remark}
In the case of $\varepsilon>0$, the results above can be extended to Witten deformed complexes $\mathcal{P}_{W,\varepsilon}=(L^2(X_{\infty};E), \mathcal{D}_{W}(P_\varepsilon),P_\varepsilon)$ with $L^2$ boundary conditions at $\infty$. These complexes are quasi-isomorphic to the complex $\mathcal{P}_{W,N,\varepsilon}=(L^2(X_{\infty};E), \mathcal{D}_{W,N}(P_\varepsilon),P_\varepsilon)$.

If we replace the radial Morse functions $h$ we consider with $-h$, then the corresponding deformed complexes on the non-compact spaces are quasi-isomorphic to $\mathcal{P}_{W,D,\varepsilon}=(L^2(X_{\infty};E), \mathcal{D}_{W,D}(P_\varepsilon),P_\varepsilon)$ with the new radial Morse function (c.f. Remark \ref{Remark_introductory_Morse}).
\end{remark}

\begin{remark}
\label{Remark_de_Rham_Dolbeault_correspondence}
It is well known that the Dolbeault cohomology in all degrees $p+q=k$ is isomorphic to the de Rham cohomology of degree $k$ on smooth complex manifolds. This continues to be the case for the local and global de Rham cohomology when $d$ is equipped with the minimal domain (and the adjoint with the maximal domain) for spaces with wedge metrics, and we refer the reader to \cite[\S 7]{jayasinghe2024holomorphic} for details. There it was shown that the \textit{rigid} Dolbeault cohomology is isomorphic to the de Rham complex as well, and rigidity was used to extend formulas used by Gibbons-Hawking in the study of gravitational instantons. Those results seem to extend to more singular settings, as verifiable in many explicit examples, motivating the choices of domains near the critical points/fixed points that we study in this article.
We also refer the reader to Example \ref{Example_spinning_sphere_intro} where the correspondence of the $\chi_{-1}$ invariant and the equivariant Euler characteristic for the de Rham complex is shown to hold for the minimal domains, but not when the Dolbeault complex is equipped with the maximal domain.

Serendipitously some important analytic results that we require in this work have been proven in \cite{Jesus_Dunkl_2014,Jesus_Dunkl_2015,Jesus2018Wittens,Jesus2018Wittensgeneral} for the de Rham complex with those domains. 
\end{remark}

This result is a generalization of analogous results in the literature including some where the cohomology groups are not computed for domains with boundary conditions (see \cite{jayasinghe2023l2,cruz2020examples,cheeger1980hodge,Jesus2018Wittensgeneral}) and some for spaces with non-compact ends (\cite{hunsicker2015weighted,HodgetheorygravitinstaRafe2004}).
While the cohomology groups of twisted de Rham complexes are finite dimensional vector spaces by the above proposition (the cohomology of smooth links being finite dimensional), they are infinite in the case of Dolbeault complexes we study, and we need to renormalize/regularize trace formulae. The following result gives important details about the Dolbeault cohomology, and generalizes results in smooth and conic settings.

\begin{proposition}[Cohomology of twisted Dolbeault complexes]
\label{proposition_Dolbeault_cohomology}
Consider a twisted Dolbeault complex $\mathcal{P}_{W,B}(X)$ where $X$ is equipped with an swp metric of finite type that is adapted Hermitian and of Reeb type with Reeb rescaling function $f_0(x)$. Then the cohomology of the Dolbeault complex is a countable set of sections which are of the form $\{u_l(x) \phi_l \}_{l \in \mathbb{Z}}$ where each $\phi_l$ is an eigensection of $\Delta_Z$ with eigenvalue $\nu^2_l$.
Here
\begin{equation}
\label{equation_holomorphic_cohomology}
    u_l(x)=\exp \Big(\int \frac{\nu_l}{f_0(x)} dx \Big), \quad u_l(x)=F(\phi)_l^{-1}\exp \Big(-\int \frac{\nu_l}{f_0(x)} dx \Big)
\end{equation}
for the cases where $B=N,D$ respectively, where $u_l\phi_l$ is integrable on $X$ and is in the domain (meets the condition $W$).
We can arrange the eigenvalues so that $\nu^2_l \rightarrow \infty$ as $l$ goes to $\infty$.
In addition we have the following.
\begin{enumerate}
    \item If there is an Hermitian $S^1$ action on $X$ generated by a vector field $V$ on $Z$, that induces a geometric endomorphism on a twisted Dolbeault complex, then these eigensections $\phi_l$ will have eigenvalues $\upsilon_l$ for the operator $\sqrt{-1}L_V$ where $\upsilon_l \in I$,
    where $I$ is a finite union of countable subsets $\{n+j | n \in \mathbb{Z}\} \subset \mathbb{R}$, where $j$ is real.

    \item If for a Hermitian swp metric which is of Reeb type, where the Reeb rescaling function is bounding from below (from above), the eigensections can be enumerated by $l \in I$ where $I \subset \mathbb{Z}$ which is bounded from below (above).
\end{enumerate}

\end{proposition}

\begin{proof}[Proof of Proposition \ref{proposition_Dolbeault_cohomology}]
Similar to the proof of Proposition \ref{proposition_cohomology_of_de_Rham}, we see that the harmonic sections are solutions of
\begin{equation}
\label{equation_intermediate_Dolbeault}
    (P+P^*) (u\phi)=0, \quad \sigma(P^*)(dx) (u\phi)|_{x=1}=0, \quad \sigma(P^*)(dx)P (u\phi)|_{x=1}=0
\end{equation}
where $\phi$ is an eigensection (on the form bundle of $Z$) of the operator $D_Z$ with eigenvalue $\mu^2_j$ for each operator $\Delta_{Z_j}$, and by the Hodge-Kodaira decomposition that these harmonic forms are in the null space of $P$ and $P^*$.

A version of the arguments in the proof of Proposition \ref{proposition_cohomology_of_de_Rham} adapted to the Dolbeault case (through the correspondence of the techniques used in Subsections \ref{subsection_eigensection_ansatz_Hodge} and \ref{subsection_eigensection_ansatz_Dolbeault}) can be used to show that the elements in the null space are either tangential (in the null space of $\iota_{\beta^\#}$) or normal (in the null space of $\beta \wedge$), where $\beta=dx - \sqrt{-1}J(dx)$ as introduced in Subsection \ref{subsection_eigensection_ansatz_Dolbeault}, and we will use the operators $P_1, P_2$ defined there.

Now we treat the case of normal forms which can be written as $\beta \wedge u_n \phi$. Since these must be $P_1$ co-exact, we get the equation
\begin{equation}
\label{equation_refering_Yasuke_29}
    -\frac{1}{F(\phi)}(u_2 F(\phi))'(x)\phi-\frac{1}{f_0(x)}\sqrt{-1} \nabla_V \phi u_n(x)=0.
\end{equation}
where we use equation \eqref{equation_definition_delta_1}.
Since the operators $\nabla_V, P_2$ commute, and $d_Z=P_2+V^{\flat} \wedge \nabla_V$ also commutes with $\nabla_V^2$, we see that the eigensections $\phi$ of $\Delta_Z$ are also eigensections of $\sqrt{-1} \nabla_V$ with eigenvalue $\nu$. Then the solutions of the differential equation \eqref{equation_refering_Yasuke_29} are given by 
\begin{equation}
\label{equation_normal_cohomology_Dolbeault}
    F(\phi)u_n(x)=\exp \Big(\int -\frac{\nu}{f_0(x)} dx \Big)
\end{equation}
up to multiplicative constants. Since $f_0(x)$ is positive for $x>0$, the first boundary condition in \eqref{equation_intermediate_Dolbeault} implies that $u_n(x)$ must vanish.  

Now we treat the case of tangential forms which can be written as $u_t \phi$. Since these must be $P_1$ exact, we get the equation 
\begin{equation}
    u_t'(x)\phi-\frac{1}{f_0(x)}\sqrt{-1} \nabla_V \phi u_t(x)=0,
\end{equation}
the solutions of which are
\begin{equation}
    u_t(x)=\exp \Big(\int +\frac{\nu}{f_0(x)} dx \Big)
\end{equation}
which solve the boundary conditions as well. In addition, the section $\phi$ must be closed under both $P_2$ and $P_2^*$, equivalently in the null space of the transversal Laplacian $\Delta_T=P_2P^*_2+P_2^*P_2$ introduced in Subsection \ref{subsection_eigensection_ansatz_Dolbeault}. This is equivalent to the section being in the null space of $\widehat{\Delta_T}:=\Delta_Z -(\sqrt{-1} \nabla_V)^2$, and we get $\Delta_Z \phi =-\nabla_V^2 \phi=\nu^2 \phi$.
Since the eigenvalues of $\Delta_Z$ grow according to Weyl's law, we see that the eigenvalues $\nu^2_l$ specified in the theorem statement go to $\infty$ as $l$ goes to $\infty$.

We observe that the Hodge star operator intertwines the boundary conditions $N$ and $D$ for the Serre dual complexes, and the operators as well (see Remark \ref{Remark_SUSY_Hodge_star_rescalings}).

\noindent {\textbf{Proof of 1:}}
The first numbered item can be proven as follows. We see that for the case of the trivial bundle, the Dolbeault complex and in particular the cohomology can be expressed as a direct sum of irreducible representations using the Peter-Weyl theorem, where the eigenvalues of $\sqrt{-1}L_V$ are integers (the Pontryagin dual of $S^1$ is $\mathbb{Z}$).
Twisting by Hermitian line bundles, we may get additional contributions corresponding to the infinitesimal action on the section generating the bundle, which shifts the eigenvalues by some constant. 
The result follows for general Hermitian bundles $E$ when $V$ lifts to $E$ and generates a geometric endomorphism.

\noindent {\textbf{Proof of 2:}}
This follows in a straighforward manner from Definition \ref{Definition_Reeb_type_notions}, which was motivated by the structure of the harmonic sections that were presented in this section.

\end{proof}

Given explicit forms of the Reeb rescaling functions, we can better understand the Dolbeault cohomology. For the case where $f_0(x)=x$, this was achieved in \cite{jayasinghe2023l2} and we can match it with the smooth case.

\begin{example}[Dolbeault cohomology on a disc]
\label{example_Dolbeault_cohomology_disc}
The Dolbeault cohomology on a smooth disc with del-bar Neumann conditions corresponds to the holomorphic functions, as can be checked on the disc in $\mathbb{C}$ using the fact that $\phi_l=e^{i\theta \nu_l}$ where $\nu_l \in \mathbb{Z}$ and $f_0(x)=x$, and using the integrability of the resulting functions $u_l(x)=x^{\nu_l}$ for $\nu_l \geq 0$. This is the basis $\{z^l=(xe^{i\theta})^l\}_{l \geq 0}$ for the holomorphic functions.

For the adjoint boundary conditions ($B=D$), for $\phi_l=e^{i\theta \nu_l} dx-\sqrt{-1}xd\theta$, we see that we get $u_l(x)=x^{-\nu_l}$ for $\nu_l \leq 0$, Which corresponds to the forms $\{\overline{z}^{l}d\overline{z}\}_{l \geq 0}$, where $z=xe^{i\theta}$.
\end{example}

\begin{remark}
\label{Remark_Dolbeault_cohom_Serre_duality_contravallation}
In degree $p=0,q=0$ of the Dolbeault complex without twists, the harmonic representatives of the cohomology for $B=N$ are simply the meromorphic functions that are integrable and in the domain, which for the minimal domain corresponds to the holomorphic functions. We refer to Subsection \ref{subsection_Illustrative_examples} for examples where the eigenvalues of $\sqrt{-1}L_V$ are shifted integers.
\end{remark}

We note that general bundles admit local cohomology groups which are not locally free modules over the cohomology of the untwisted complex, and refer the reader to the case of the Dolbeault complex with Hodge degree $p=1$ for the example in \cite[\S 7.3.3.]{jayasinghe2023l2}.

\begin{proposition}
\label{proposition_class_of_metrics_bounding_below}
Given a swp metric of discrete type which is Hermitian, of Reeb type and has Reeb rescaling function $f_0(x)=x^{\alpha}$ for $\alpha \geq 1$, then it is bounding from below.
\end{proposition}

\begin{proof}
When $\alpha>1$, for $B=N$ we get
\begin{equation}
    u_l(x)=\exp \Big(  \frac{\nu_l}{1-\alpha} x^{1-\alpha} \Big)
\end{equation}
which satsifies a growth condition (determined by the choice of $W$) at $x=0$ and is $L^2$ bounded for large values of $\nu_l^2$ only when $\nu_l \geq 0$ (there maybe some negative eigenvalues with small magnitude for which the corresponding functions are $L^2$ bounded), and thus we have a lower bound on admissible values of $\nu_l$ for the cohomology. We know that $\nu_l^2$ grows by Weyl's law for eigenvalues of $\Delta_Z$ which is an elliptic operator on $Z$.
A similar argument shows that for $B=D$ (see \eqref{equation_holomorphic_cohomology}) there is an upper bound on admissible values of $\nu_l$ for the cohomology.
We observe that in the case where $f_0(x)=x$ ($\alpha=1$), $u_l(x)=x^{ \pm \nu_l}$ where only finitely many such functions with $\pm \nu_l <0$ can appear due to the growth of $\nu^2_l$ and the $L^2$ bounded condition for cohomology.
The case of the smooth disc that we studied in Example \ref{example_Dolbeault_cohomology_disc} illustrates this. The general case is similar, and in the conic case was studied in \cite{jayasinghe2023l2}.
\end{proof}

\subsection{Geometric endomorphisms, supertraces and cancellations}
\label{Subsection_supertraces_cancellations}

\begin{definition}
\label{geometric_endo}
An \textbf{\textit{endomorphism}} $T$  \textbf{\textit{of a Hilbert complex $\mathcal{P}=(H_k, \mathcal{D}(P_k), P_k)$}} is given by an n-tuple of maps $T=(T_0, T_1,...,T_n)$, where $T_k:H_k \rightarrow H_k$ are bounded maps of Hilbert spaces, that satisfy the following properties.
\begin{enumerate}
    \item $T_k(\mathcal{D}(P_k)) \subseteq \mathcal{D}(P_k)$
    \item $P_k \circ T_k= T_{k+1} \circ P_k$ on $\mathcal{D}(P_k)$
\end{enumerate}
Each endomorphism $T_k$ has an adjoint $T_k^*$. If each $T_k^*$ preserves the domain $\mathcal{D}(P_k^*)$, then we call $T^*=(T^*_0, T^*_1,...,T^*_n)$ the \textit{\textbf{adjoint endomorphism}} of the dual complex.
\end{definition}

If the endomorphism is associated to a self-map $f:X \rightarrow X$, we say that it is a \textbf{geometric endomorphism}. One can check that the pullback by a holomorphic isometry of $X$ gives geometric endomorphisms on all the complexes we study in this article. When the self maps are not isometries, more intricate constructions of geometric endomorphisms in necessary to get the correct formulas and we refer to \cite[\S 4]{jayasinghe2023l2} (c.f. \cite{AtiyahBott1,Bei_2012_L2atiyahbottlefschetz}) for more details. In particular we note that even algebraic versions given by \cite{baum1979lefschetz} which are proven for finite actions only extend to the non-isometric case when geometric endomorphisms are considered, as can be checked in Examples 7.34 and Example 7.38 of \cite{jayasinghe2023l2}.
    
We define \textbf{polynomial Lefschetz heat supertrace functions}
\begin{equation}
    \mathcal{L}(\mathcal{P},T,t,s)(b)
    :=\sum_{q=0}^n b^q Tr \Big(T \circ e^{-t\Delta}e^{-s(\Delta+\sqrt{\Delta_Z})} \Big)|_{\mathcal{P}}.
\end{equation}
and \textbf{polynomial Lefschetz supertrace functions}
\begin{equation}
    L(\mathcal{P},T,s)(b)
    :=\sum_{q=0}^n b^q Tr \Big(T\circ e^{-s\sqrt{\Delta_Z}} \Big)|_{\mathcal{H}^q(\mathcal{P})}.
\end{equation}
analogous to \cite{jayasinghe2023l2} in the wedge setting.
The following is a version of the renormalized McKean Singer theorem given in Theorem 5.41 of \cite{jayasinghe2023l2}.
A simpler version of this result is given in Theorem 3.21 \cite{jayasinghe2023l2}, suitable for Fredholm complexes (such as twisted de Rham ones we study here). 
Given a $S^1_{\theta}$ group action, we will denote the induced geometric endomorphism as $T_{\theta}$. We will continue to use this notation even when the action extends to a $\mathbb{C}^*$ action.

\begin{theorem}
\label{Theorem_Morse_supertrace}
Let $\mathcal{P}=(H,\mathcal{D}(P),P)$ be a Witten deformed twisted de Rham/Dolbeault complex (of degree $m$) for some $\varepsilon \geq 0$ described in Theorem \ref{Theorem_main_spectral}. 
Let $T=T_{\theta}$ be an endomorphism of the complex corresponding to a $\mathbb{C}^*$ action. For all $t,s \in \mathbb{R}^+$, we can expand the Lefschetz heat supertrace functions as
\begin{equation}
\label{equation_with_the_b}
    \mathcal{L}(\mathcal{P},T,t,s)(b)
    =L(\mathcal{P},T,t,s)(b)+ (1+b) \sum_{k=0}^{n-1} b^k S_k(t,s)
\end{equation}
where
\begin{equation} 
\label{equation_residuals_101}
    S_k(t)=\sum_{{n,k}} e^{-(t+s) \lambda^2_{n,k} - s \mu_{n}}  \langle T \psi_{n,k}, \psi_{n,k} \rangle
\end{equation}
where $\{\psi_{n,k}\}$ are an orthonormal basis of $P$ \textbf{co-exact} eigensections as in Theorem \ref{Theorem_main_spectral}. In particular 
\begin{equation}
\label{equation_Namal_Hora}
    \mathcal{L}(\mathcal{P},T,t,s)(-1)= L(\mathcal{P},T,s)(-1),
\end{equation}
and in the case of a twisted de Rham complex, this expression can be evaluated for any value of $s \in \mathbb{C}$.
For a twisted Dolbeault complex if the geometric endomorphism $T_{\theta}$ is norm preserving, then the expression in \eqref{equation_Namal_Hora} is holomorphic in the variable $s$ for $\text{Re}(s)>0$.
In particular, $\mathcal{L}(\mathcal{P},T,t,s)(-1)$ is independent of $t$.
\end{theorem}

\begin{proof}
We can expand the polynomial Lefschetz heat supertrace function using the Hodge-Kodaira decomposition and the orthonormal basis of eigenvectors given in Theorem \eqref{Theorem_main_spectral} to get 
\begin{multline}
\label{equation_Mahinda_Hora}
\sum_{q=0}^m b^q Tr(T e^{-t\Delta}e^{-s(\Delta+\sqrt{\Delta_Z})}) =  \sum_{q=0}^m b^k  tr(e^{-s \sqrt{\Delta_Z}}T|_{\mathcal{H}^k(\mathcal{P})})\\ 
+\sum_{q=0}^{m-1} b^q \sum_{n,k} e^{-(t+s) \lambda^2_{n,k} -s\mu_n}  \langle T \psi_{n,k}, \psi_{n,k} \rangle  
+\sum_{q=0}^{m-1}  b^{q+1} \sum_{n,k} e^{-(t+s) \lambda^2_{n,k} -s\mu_n}  \langle T P \frac{\psi_{n,k}}{\lambda_{n,k}}, P \frac{\psi_{n,k}}{\lambda_{n,k}} \rangle
\end{multline}
where we have used the fact that there are no exact elements in degree 0, and that there are no co-exact elements in degree $m$. To show that we have a $1+b$ factor and the expression for $S_k$, we use the supersymmetry given by Lemma \ref{Lemma_super_symmetry} as follows:
\begin{multline}
\label{argument_Carrickfergus}
\frac{1}{\lambda^2_{n,k}}\langle T P_q \psi_{n,k}, P_q \psi_{n,k} \rangle = \frac{1}{\lambda^2_{n,k}}\langle P_q T \psi_{n,k}, P_q \psi_{n,k} \rangle=\frac{1}{\lambda^2_{n,k}}\langle T \psi_{n,k}, P_{q}^* P_q \psi_{n,k} \rangle\\
=\langle T \psi_{n,k}, \frac{1}{\lambda^2_{n,k}} \Delta_q \psi_{n,k} \rangle=\langle T \psi_{n,k}, \psi_{n,k} \rangle.
\end{multline}
It is now easy to see that we have the expression for $S_k(t,s)$ in \eqref{equation_residuals_101}, and that equation \eqref{equation_with_the_b} holds for $t>0, s>0$. It is clear that for $b=-1$, this yields \eqref{equation_Namal_Hora} where the right hand side is independent of $t$.

Since the twisted de Rham cohomology is finite dimensional, we can evaluate \eqref{equation_Namal_Hora} for any $s$.
The results for the Dolbeault cohomology in Proposition \ref{proposition_Dolbeault_cohomology} shows that 
\begin{equation}
    L(\mathcal{P},T,s)(-1)=\sum_{q=0}^m (-1)^q  tr(e^{-s \sqrt{\Delta_Z}}T|_{\mathcal{H}^q(\mathcal{P})})
\end{equation}
where the expressions in the sum for each degree $q$ (of the complex) is of the form
\begin{equation}
    \sum_{l \in \mathbb{Z}} e^{-s|\nu_l|} \langle T u_l\phi_l, u_l\phi_l \rangle_{L^2}
\end{equation}
where only the forms of said degree are taken in the sum. Using the norm preserving condition and the growth of eigenvalues $\nu_l$ given in Proposition \ref{proposition_Dolbeault_cohomology}, it is easy to see that this sum converges for $\text{Re}(s)>0$.
\end{proof}

\begin{remark}
In the conic (wedge) case, it was shown in \cite{jayasinghe2023l2} that the Lefschetz supertrace functions are holomorphic on a half plane that includes $s=0$ for attracting self maps (with $B=N$), and for expanding self maps (with $B=D$), where the explicit form of the functions $u_l$ in the wedge setting were used in the proof.
\end{remark}

\begin{corollary}
\label{Corollary_Morse_supertraces}
 Given a swp metric of discrete type on $X$ that is adapted K\"ahler and is of Reeb type, 
where there is a K\"ahler Hamiltonian Morse function, the polynomial Morse supertrace function is a finite linear combination 
\begin{equation}
\label{equation_Laurent_series_combinations_Morse}
    M(\mathcal{P}_{W,B},T_\theta,s)(b)= \sum_{q=0}^n b^q \sum_{k \in I'} C_{k,\theta} u_{k, s,\theta}
\end{equation}
where $u_{k, s,\theta} \in \mathbb{Z}_{\geq 0}[\lambda,\lambda^{-1}]$, where $\lambda=e^{-s+i\theta}$, $I'$ is a finite indexing set and where each $C_{k,\theta}$ is a real power of $\lambda$.
Additionally we have the following.
\begin{enumerate}
\item If both of the analytic functions (in the variable $s$), $M(\mathcal{P}_{W,B},T_\theta,s)(b)$ and $L(\mathcal{P}_{W,B},T_\theta,s)$ are regular at $s=0$, then
\begin{equation}
    M(\mathcal{P}_{W,B},T_\theta,0)(-1)=L(\mathcal{P}_{W,B},T_\theta,0).
\end{equation}
\item If the Reeb rescaling function $f_0(x)$ is bounding from below (respectively above), 
then $M(\mathcal{P}_{W,N},T_\theta,s)(b)$ converges for $s>0$ (respectively $s<0$), and $M(\mathcal{P}_{W,D},T_\theta,s)(b)$ converges for $s<0$ (respectively $s>0$).
\end{enumerate}
\end{corollary}

\begin{proof}[Proof of Corollary \ref{Corollary_Morse_supertraces}]
As in the proof of Proposition \ref{proposition_Dolbeault_cohomology}, the circle action induces a representation on the cohomology which gives rise to a decomposition of the harmonic representatives in the cohomology into the irreducible representations of the circle action. For the case of the trivial bundle, the operator $\sqrt{-1}L_V$ has integer eigenvalues, and the supertrace is thus an element of $\mathbb{Z}_{\geq 0}[\lambda,\lambda^{-1}]$.
For line bundles, the twists may give rise to a shift of the eigenvalues of $\sqrt{-1}L_V$, which we incorporate into an overall factor $C_{k,\theta}$. For general adapted Hermitian bundles to which the action lifts, we get \eqref{equation_Laurent_series_combinations_Morse}. 

The first numbered item follows from the definitions since one knows that the functions can be evaluated at $s=0$.
The second numbered item follows from the definitions and the second numbered item in Proposition \ref{proposition_Dolbeault_cohomology}.
\end{proof}

We note that the proof of the second numbered item here is similar to the proof of the holomorphicity of the Supertrace functions for $\text{Re}(s)>0$ in the proof of Theorem \ref{Theorem_Morse_supertrace}. There the weight used to renormalize was the exponential of $\sqrt{\Delta_Z}$, which is non-negative whereas here the eigenvalues of the operator $\sqrt{-1}L_V$ can be negative, and this is why the region of holomorphicity changes depending on the Reeb rescaling function.

\subsection{Examples}
\label{subsection_Illustrative_examples}

We study trace formulas for several complexes on the unit disc in the complex plane, briefly discuss trace formulas on compactifications (spinning the two sphere) and a singular space (spinning the cusp curve) and direct the reader to \cite{jayasinghe2023l2,jayasinghe2024holomorphic} for many more singular examples and applications in the wedge setting, including extensions of formulas for various partition functions arising in the study of equivariant Donaldson-Witten theory, gravitational entropy and supergravity. 

It is easy to extend the computations given in this section to singular examples with metrics of the form studied in this article, for instance to topological spheres with the metric $d\phi^2+\sin^{2c}(\phi) d\theta^2$ in spherical coordinates $(\phi,\theta) \in [0,\pi]\times [0,2\pi\beta)$ (the round metric corresponds to $c=1,\beta=1$).

\subsubsection{Example: Rotation of the complex plane}
\label{Example_rotation_of_complex_plane}

Consider the space $\widehat{X}=\mathbb{D}^2\subset \mathbb{C}^2$ with the complex coordinate $z$, equipped with the K\"ahler Hamiltonian $S^1$ action generated by the vector field $V=\partial_{\theta}$. 
We consider the Dolbeault complex with del-bar Neumann conditions at the boundary,
($\mathcal{P}_{\min,N}=(L^2\Omega^{0,\cdot}(\widehat{X};E), \mathcal{D}(\overline{\partial}_E),\overline{\partial}_E)$ for $E=\mathbb{C}$ in the notation used in this article).
The cohomology group (which we identify with the space of harmonic sections of $\mathcal{H}^q(\mathcal{P}_{\min,N})$) of the Dolbeault complex is well known to be the space of holomorphic functions in degree $0$ (and is the trivial group in higher degrees), which has a Schauder basis $\{ z^k \}_{k \in I}$ where the indexing set $I$ is the set of non-negative integers. If we define a renormalized supertrace function
\begin{equation}
    Tr_{re}(f^* \circ e^{-s\sqrt{-1} L_V} )|_{\mathcal{H}^0(\mathcal{P}_{\min,N})}=\sum_{k=0}^{\infty} \lambda^k e^{-ks}=\frac{1}{1-\lambda e^{-s}}
\end{equation}
where the first equality is given by the fact that $z=re^{i\theta}$, then the evaluation of the expression on the right hand side at $s=0$ is the local Lefschetz number (the equivariant Todd genus) at the fixed point of the action at the origin, as well as the holomorphic Morse polynomial corresponding to the critical point of the K\"ahler Hamiltonian function $|z|^2$.

If we instead consider the Dolbeault complex with coefficients in the canonical bundle $E=\mathcal{K}$ (which has trivializing section $dz$), then the cohomology of the complex has a basis $\{ z^k dz \}_{k \in I}$, and we get
\begin{equation}
    Tr_{re}(f^* \circ e^{-s\sqrt{-1} L_V} )|_{\mathcal{H}^0(\mathcal{P}_{\min,N})}=\sum_{k=0}^{\infty} \lambda^{k+1} e^{-{k+1}s}=\frac{\lambda e^{-s}}{1-\lambda e^{-s}}
\end{equation}
and evaluating the meromorphic continuation at $s=0$ gives the local Local Lefschetz number and holomorphic Morse polynomial for this complex.
The local equivariant Hirzebruch $\chi_y$ invariant is the linear combination given by
\begin{equation}
\label{chi_y_example_1}
    \chi_y (\lambda):= \frac{1}{1-\lambda}+y\frac{\lambda}{1-\lambda}
\end{equation}
and setting $y=0, -1, +1$  yields the equivariant Todd genus, equivariant Euler characteristic (the equivariant index for the de Rham complex with absolute boundary conditions) and the equivariant Signature invariant respectively.
Similarly the Dolbeault complex with coefficients in the bundle $E=\sqrt{\mathcal{K}}$ (by which we mean a line bundle that squares to the canonical bundle) with cohomology having basis $\{ z^k \sqrt{dz} \}_{k \in I}$ corresponds to the spin Dirac complex, and the cohomology corresponds to the local harmonic spinors and the local Lefschetz number/local Morse polynomial is 
\begin{equation}
    \frac{{\lambda}^{1/2}}{1-\lambda}.
\end{equation}
In this case the spin-Dirac operator is precisely the twisted Dolbeault-Dirac operator (also known as the spin$^\mathbb{C}$ Dirac operator) $D=\overline{\partial}_E+\overline{\partial}^*_E$, and $E$ is isomorphic to the spin bundle. We refer the reader to \cite[\S 7]{jayasinghe2024holomorphic} for more interesting cases of complexes where twists of the canonical bundle appear.

The formulas in \cite{AtiyahBott1,AtiyahBott2} express the denominators of the formulas given above at the determinant of the matrix $Id-df'_p$ where $df'$ is the holomorphic derivative of the self-map corresponding to the rotation we study, at the fixed point $p$.

The work in \cite{jayasinghe2023l2,weiping1990note} shows the relation of the local Lefschetz numbers with equivariant eta invariants. In the case of this example we can understand it as follows in the notation of \cite{weiping1990note}. We denote by $h$ the trace of the geometric endomorphism over the nullspace of the Dirac-type operator on the link $Z=S^1$ (corresponding to the cone at the origin) on functions. Since the Dirac-type operator is simply $D=i^{-1}\partial_{\theta}$, the nullspace is given by the constants, a vector space over $\mathbb{C}$ generated by $1$, and the trace $h=1$. The equivariant eta invariant is the evaluation at $s=0$ of the function $\eta_T(s,D)$, which as in \cite{weiping1990note} can be written as
\begin{equation}
\eta_T(s,D)= \sum_{k \in \mathbb{Z} \setminus \{0\} } \frac{sign(k)}{|k|^s}e^{ik\alpha} 
\end{equation}
where we use that the non-zero eigenspaces are one dimensional with eigenvalues given by the integers, where each factor $e^{ik\alpha}=\lambda^k$ is the trace of the eigensections of $D$ (given by $e^{ik\theta}$). An alternate renormalization corresponding to this article can be used to write the equivariant index as half of
\begin{equation}
\begin{split}
\widetilde{\eta_{T_{\alpha}}(s,D)}+h & = \sum_{k=-\infty}^{-1} -e^{-s|k|} e^{ik\alpha} + \sum_{k=1}^{\infty} e^{-sk} e^{ik\alpha} +1\\
 & = str\Big(e^{-s\sqrt{\Delta_Z}} T_{\alpha}|_{ \mathcal{H}(\mathcal{P}_D(C(Z)))} \Big)+ str\Big(e^{-s\sqrt{\Delta_Z}} T_{\alpha}|_{ \mathcal{H}(\mathcal{P}_N(C(Z)))} \Big)\\
 & = -\frac{e^{-s}e^{-i\alpha}}{1-e^{-s}e^{-i\alpha}}+\frac{1}{1-e^{-s}e^{i\alpha}}
\end{split}
\end{equation}
where the computations hold for $s>0$, for $\alpha$ not an integer multiple of $2\pi$. Evaluating the last expression at $s=0$ we get twice the equivariant Lefschetz number for the local Dolbeault complex.

If instead one evaluates the expression at $\alpha=0$ for general $s$, we see that the (non-equivariant) eta invariant at $s=0$ is $0$ since
\begin{equation}
\widetilde{\eta_{T_{0}}(s,D)}+1 = -\frac{e^{-s}}{1-e^{-s}}+\frac{1}{1-e^{-s}}=1=\sum_{k \in \mathbb{Z} \setminus \{0\} } sign(k)e^{-s|k|} +h
\end{equation}
for $s>0$.
The results in \cite{jayasinghe2023l2} indicate how the equivariant eta invariants are renormalized supertraces of local cohomology groups when there are almost complex structures compatible with the Dirac operator.

\subsubsection{Example: Spinning the 2-sphere}
\label{Example_spinning_sphere_intro}
    Consider $\mathbb{CP}^1$ with projective coordinates $[Z_0:Z_1]$, equipped with the Hamiltonian $\mathbb{C}^*$ action $(\lambda)[Z_0:Z_1]=[\lambda Z_0: Z_1]$ where $\lambda=se^{i\theta} \in \mathbb{C}^*$. 
    This has two fixed points at $[0:1]$ and $[1:0]$. The bundles studied above all extend to this space, and the local Lefschetz numbers at the point $[0:1]$ correspond to those in Subsection \ref{Example_rotation_of_complex_plane} above, whereas those at $[1:0]$ are obtained by subtituting $\lambda^{-1}$ for $\lambda$.
    Thus we can express the sum of global Lefschetz numbers for each complex as the sum of the corresponding local Lefschetz numbers.

We consider the spin bundle $\mathcal{K}^{1/2}$ on this space. This bundle admits a Hermitian connection such that its curvature is equal to the K\"ahler form on the sphere when the form $\omega/(2\pi)$ defines an integral cohomology class (which can be arranged up to a rescaling of the space), and thus is called a \textbf{prequantum line bundle}. Then the global Lefschetz number is 
\begin{equation}
    \frac{{\lambda}^{1/2}}{1-\lambda}+\frac{{\lambda}^{-1/2}}{1-\lambda^{-1}}=0
\end{equation}
where $\lambda^{1/2}=e^{it}$. This character formula for the prequantum line bundle is sometimes called a \textbf{Quantum Duistermaat-Heckman} formula, which can be obtained by expanding the denominators at each fixed points in powers of $t$ and truncating it in the leading powers of $t^n$ (for $2n$ dimensional spaces). We refer the reader to \cite[\S 3]{guillemin1996symplectic} for more details (in particular see the discussion above Proposition 3.4.1, and Example 3.10 in \textit{loc. cit.}).
For the space above we get the Duistermaat-Heckman formula,
\begin{equation}
    \frac{e^{it}}{it}-\frac{e^{-it}}{-it}=\frac{1}{2\pi}\int_{\theta=0}^{2\pi} \int_{\phi=0}^{\pi} e^{it\cos(\phi)} \sin(\phi) d\phi d\theta.
\end{equation}

\subsubsection{Example: Spinning the cusp curve}
\label{Example_spinning_cusp_curve}

Consider the algebraic curve $ZY^2-X^3=0$ in $\mathbb{CP}^2$ equipped with the $\mathbb{C}^*$ action $(\lambda)\cdot [X:Y:Z]=[\lambda^2X:\lambda^3Y:Z]$. This is a K\"ahler Hamiltonian action with one smooth fixed point at $[0:1:0]$ with holomorphic Lefschetz number $1/(1-\lambda^{-1})$. The other fixed point is at the singularity $[0:0:1]$, where the tangent cone is a cone over the trefoil knot parametrized by $\theta \in [0,4\pi)$ where $\theta$ can be identified with the argument of $y/x$. We denote $t^2=\rho e^{i\theta}$ and we refer the reader to subsection 7.1.2 of \cite{jayasinghe2023l2} for a detailed study of the space and the cohomology with the domains studied here. There it is shown that the cohomology of the Dolbeault complex $\mathcal{P}_{\min,N}$ has a basis $\{t^k\}_k \geq 0$, while for $\mathcal{P}_{\max,N}$ there is a basis $\{t^k\}_k \geq -1$ (in both cases $k \in \mathbb{Z}$).

Since the action sends $t \rightarrow \lambda t$, we can compute that the equivariant $\chi_y$ invariant for the cohomology with $W=\min$ is the same as that in \ref{chi_y_example_1} (see Example 7.3.4 of \cite{jayasinghe2023l2}), which for $y=-1$ corresponds to the equivariant index of the de Rham complex $\chi_{-1}(\lambda)=1$, in particular independent of $\lambda$.
For $W=\max$ we get
\begin{equation}
\label{chi_y_example_2}
    \chi_y (\lambda):= \lambda^{-1}+\frac{1}{1-\lambda}+y\frac{1}{1-\lambda}
\end{equation}
which for $y=-1$ is not the equivariant index of the de Rham complex.
This is a particular case of the correspondence of the de Rham and Dolbeault complexes in Remark \ref{Remark_de_Rham_Dolbeault_correspondence}.

\bibliographystyle{alpha}
\bibliography{reference}

\newcommand{\etalchar}[1]{$^{#1}$}
\begin{thebibliography}{JMMV14}

\bibitem[AB67]{AtiyahBott1}
M.~F. Atiyah and R.~Bott.
\newblock A {L}efschetz fixed point formula for elliptic complexes. {I}.
\newblock {\em Ann. of Math. (2)}, 86:374--407, 1967.

\bibitem[AB68]{AtiyahBott2}
M.~F. Atiyah and R.~Bott.
\newblock A {L}efschetz fixed point formula for elliptic complexes. {II}. {A}pplications.
\newblock {\em Ann. of Math. (2)}, 88:451--491, 1968.

\bibitem[AGR23]{Albin_2017_index}
Pierre Albin and Jesse Gell-Redman.
\newblock The index formula for families of {D}irac type operators on pseudomanifolds.
\newblock {\em J. Differential Geom.}, 125(2):207--343, 2023.

\bibitem[Alb]{Pierre_Exposition}
Pierre Albin.
\newblock {Lecture notes for a minicourse at “Noncommutative geometry and index theory for group actions and singular spaces”, May 21–May 25, 2018, at Texas A \& M (unpublished). available at \url{https://faculty.math.illinois.edu/~palbin/IntroIndexSing.pdf}}.

\bibitem[ALC17]{Jesus2018Wittens}
Jes\'{u}s~A. \'{A}lvarez L\'{o}pez and Manuel Calaza.
\newblock Witten's perturbation on strata.
\newblock {\em Asian J. Math.}, 21(1):47--125, 2017.

\bibitem[{\'A}LCF15]{Jesus_Dunkl_2015}
Jes{\'u}s~A. {\'A}lvarez~L{\'o}pez, Manuel Calaza, and Carlos Franco.
\newblock A perturbation of the {Dunkl} harmonic oscillator on the line.
\newblock {\em SIGMA, Symmetry Integrability Geom. Methods Appl.}, 11:paper 059, 33, 2015.

\bibitem[ALCF18]{Jesus2018Wittensgeneral}
Jes\'{u}s~A. \'{A}lvarez L\'{o}pez, Manuel Calaza, and Carlos Franco.
\newblock Witten's perturbation on strata with general adapted metrics.
\newblock {\em Ann. Global Anal. Geom.}, 54(1):25--69, 2018.

\bibitem[ALMP12]{Albin_signature}
Pierre Albin, \'{E}ric Leichtnam, Rafe Mazzeo, and Paolo Piazza.
\newblock The signature package on {W}itt spaces.
\newblock {\em Ann. Sci. \'{E}c. Norm. Sup\'{e}r. (4)}, 45(2):241--310, 2012.

\bibitem[ALMP17]{albin2013novikov}
Pierre Albin, Eric Leichtnam, Rafe Mazzeo, and Paolo Piazza.
\newblock The {N}ovikov conjecture on {C}heeger spaces.
\newblock {\em J. Noncommut. Geom.}, 11(2):451--506, 2017.

\bibitem[ALMP18]{Albin_hodge_theory_cheeger_spaces}
Pierre Albin, Eric Leichtnam, Rafe Mazzeo, and Paolo Piazza.
\newblock Hodge theory on {C}heeger spaces.
\newblock {\em J. Reine Angew. Math.}, 744:29--102, 2018.

\bibitem[Bau82]{Baumformula81}
Paul Baum.
\newblock Fixed point formula for singular varieties.
\newblock In {\em Current trends in algebraic topology, {P}art 2 ({L}ondon, {O}nt., 1981)}, volume~2 of {\em CMS Conf. Proc.}, pages 3--22. Amer. Math. Soc., Providence, R.I., 1982.

\bibitem[BE24]{berwickevans2024ellipticcohomologyquantumfield}
Daniel Berwick-Evans.
\newblock Elliptic cohomology and quantum field theory, 2024.

\bibitem[Bei13]{Bei_2012_L2atiyahbottlefschetz}
Francesco Bei.
\newblock The {$L^2$}-{A}tiyah-{B}ott-{L}efschetz theorem on manifolds with conical singularities: a heat kernel approach.
\newblock {\em Ann. Global Anal. Geom.}, 44(4):565--605, 2013.

\bibitem[BFQ79]{baum1979lefschetz}
Paul Baum, William Fulton, and George Quart.
\newblock Lefschetz-{R}iemann-{R}och for singular varieties.
\newblock {\em Acta Math.}, 143(3-4):193--211, 1979.

\bibitem[BGRH23]{Dean_Jesse_Quantum_Ergodic_WP_2023}
Dean Baskin, Jesse Gell-Redman, and Xiaolong Han.
\newblock Riemann moduli spaces are quantum ergodic.
\newblock {\em J. Differ. Geom.}, 123(3):391--410, 2023.

\bibitem[BGV04]{berline2003heat}
Nicole Berline, Ezra Getzler, and Mich\`ele Vergne.
\newblock {\em Heat kernels and {D}irac operators}.
\newblock Grundlehren Text Editions. Springer-Verlag, Berlin, 2004.
\newblock Corrected reprint of the 1992 original.

\bibitem[BH24]{bandara2024geometric}
Lashi Bandara and Georges Habib.
\newblock Geometric singularities and hodge theory.
\newblock {\em arXiv preprint arXiv:2407.01170}, 2024.

\bibitem[BL91]{bismut1991complex}
Jean-Michel Bismut and Gilles Lebeau.
\newblock Complex immersions and {Q}uillen metrics.
\newblock {\em Inst. Hautes \'{E}tudes Sci. Publ. Math.}, (74):ii+298 pp. (1992), 1991.

\bibitem[BL92]{bru1992hilbert}
J.~Br\"{u}ning and M.~Lesch.
\newblock Hilbert complexes.
\newblock {\em J. Funct. Anal.}, 108(1):88--132, 1992.

\bibitem[Br{\"u}96]{bruning1996signature}
Jochen Br{\"u}ning.
\newblock The signature theorem for manifolds with metric horns.
\newblock {\em Journ{\'e}es {\'E}quations aux d{\'e}riv{\'e}es partielles}, pages 1--10, 1996.

\bibitem[Che80]{cheeger1980hodge}
Jeff Cheeger.
\newblock On the {H}odge theory of {R}iemannian pseudomanifolds.
\newblock In {\em Geometry of the {L}aplace operator ({P}roc. {S}ympos. {P}ure {M}ath., {U}niv. {H}awaii, {H}onolulu, {H}awaii, 1979)}, Proc. Sympos. Pure Math., XXXVI, pages 91--146. Amer. Math. Soc., Providence, R.I., 1980.

\bibitem[Che83a]{cheeger1983hodge}
Jeff Cheeger.
\newblock Hodge theory of complex cones.
\newblock {\em Ast{\'e}risque}, 102:118--134, 1983.

\bibitem[Che83b]{cheeger1983spectral}
Jeff Cheeger.
\newblock Spectral geometry of singular {R}iemannian spaces.
\newblock {\em J. Differential Geom.}, 18(4):575--657 (1984), 1983.

\bibitem[Cru20]{cruz2020examples}
Joshua Cruz.
\newblock {\em Examples of the {L}ocal {L}2-{C}ohomology of {A}lgebraic {V}arieties}.
\newblock ProQuest LLC, Ann Arbor, MI, 2020.
\newblock Thesis (Ph.D.)--Duke University.

\bibitem[DEF{\etalchar{+}}99]{deligne1999quantum}
Pierre Deligne, Pavel Etingof, Daniel~S Freed, Lisa~C Jeffrey, David Kazhdan, John~W Morgan, David~R Morrison, and Edward Witten.
\newblock {\em Quantum Fields and Strings: A Course for Mathematicians: Volume 2}, volume~2.
\newblock American Mathematical Society, 1999.

\bibitem[Dui11]{duistermaat2013heat}
J.~J. Duistermaat.
\newblock {\em The heat kernel {L}efschetz fixed point formula for the spin-{$c$} {D}irac operator}.
\newblock Modern Birkh\"{a}user Classics. Birkh\"{a}user/Springer, New York, 2011.
\newblock Reprint of the 1996 edition.

\bibitem[Eps06]{epstein2006subelliptic}
Charles~L. Epstein.
\newblock Subelliptic boundary conditions for {$\rm Spin_{\Bbb C}$}-{D}irac operators, gluing, relative indices, and tame {F}redholm pairs.
\newblock {\em Proc. Natl. Acad. Sci. USA}, 103(42):15364--15369, 2006.

\bibitem[Eps08]{EpsteinSubellipticSpinc3_2007}
Charles~L. Epstein.
\newblock Subelliptic {${\rm Spin}_{\Bbb C}$} {D}irac operators. {III}. {T}he {A}tiyah-{W}einstein conjecture.
\newblock {\em Ann. of Math. (2)}, 168(1):299--365, 2008.

\bibitem[GH79]{gibbons1979classification}
Gary~W Gibbons and Stephen~W Hawking.
\newblock Classification of gravitational instanton symmetries.
\newblock {\em Communications in Mathematical Physics}, 66:291--310, 1979.

\bibitem[GLS96]{guillemin1996symplectic}
Victor Guillemin, Eugene Lerman, and Shlomo Sternberg.
\newblock {\em Symplectic fibrations and multiplicity diagrams}.
\newblock Cambridge University Press, Cambridge, 1996.

\bibitem[HHM04]{HodgetheorygravitinstaRafe2004}
Tam{\'a}s Hausel, Eugenie Hunsicker, and Rafe Mazzeo.
\newblock Hodge cohomology of gravitational instantons.
\newblock {\em Duke Math. J.}, 122(3):485--548, 2004.

\bibitem[HR15]{hunsicker2015weighted}
Eug\'{e}nie Hunsicker and Fr\'{e}d\'{e}ric Rochon.
\newblock Weighted {H}odge cohomology of iterated fibred cusp metrics.
\newblock {\em Ann. Math. Qu\'{e}.}, 39(2):177--184, 2015.

\bibitem[HS87]{helffer1987wells}
Bernard Helffer and Johannes Sj{\"o}strand.
\newblock Multiple wells in semi-classical mechanics vi. (case of sub-varietal wells).
\newblock In {\em Annals of the IHP Theoretical Physics}, volume~46, pages 353--372, 1987.

\bibitem[Jay24a]{jayasinghe2023l2}
Gayana Jayasinghe.
\newblock An analytic approach to {L}efschetz and {M}orse theory on stratified pseudomanifolds.
\newblock {\em arXiv preprint arXiv:2404.13481}, 2024.

\bibitem[Jay24b]{jayasinghe2024holomorphic}
Gayana Jayasinghe.
\newblock {H}olomorphic {W}itten instanton complexes on stratified pseudomanifolds with {K}{\"a}hler wedge metrics.
\newblock {\em arXiv preprint arXiv:2404.13481}, 2024.

\bibitem[JMMV14]{MazzeoVasy2014SpectralWeilPeterson}
Lizhen Ji, Rafe Mazzeo, Werner M{\"u}ller, and Andras Vasy.
\newblock Spectral theory for the {Weil}-{Petersson} {Laplacian} on the {Riemann} moduli space.
\newblock {\em Comment. Math. Helv.}, 89(4):867--894, 2014.

\bibitem[JQY24]{jayasinghe2024witten}
Gayana Jayasinghe, Hadrian Quan, and Xinran Yu.
\newblock Witten instanton complex and {M}orse-{B}ott inequalities on stratified pseudomanifolds.
\newblock {\em arXiv preprint arXiv:2412.12003}, 2024.

\bibitem[LC14]{Jesus_Dunkl_2014}
Jes{\'u}s A.~{\'A}lvarez L{\'o}pez and Manuel Calaza.
\newblock Embedding theorems for the {Dunkl} harmonic oscillator on the line.
\newblock {\em SIGMA, Symmetry Integrability Geom. Methods Appl.}, 10:paper 004, 16, 2014.

\bibitem[Lud13]{ludwig2013analytic}
Ursula Ludwig.
\newblock An analytic approach to the stratified {M}orse inequalities for complex cones.
\newblock {\em Internat. J. Math.}, 24(12):1350100, 12, 2013.

\bibitem[Lud17a]{ludwig2017comparison}
Ursula Ludwig.
\newblock Comparison between two complexes on a singular space.
\newblock {\em J. Reine Angew. Math.}, 724:1--52, 2017.

\bibitem[Lud17b]{ludwig2017index}
Ursula Ludwig.
\newblock An index formula for the intersection {E}uler characteristic of an infinite cone.
\newblock {\em C. R. Math. Acad. Sci. Paris}, 355(1):94--98, 2017.

\bibitem[MW97]{mathai1997equivariant}
Varghese Mathai and Siye Wu.
\newblock Equivariant holomorphic {M}orse inequalities. {I}. {H}eat kernel proof.
\newblock {\em J. Differential Geom.}, 46(1):78--98, 1997.

\bibitem[NS04]{NekrasovABCDinsta}
Nikita Nekrasov and Sergey Shadchin.
\newblock A{BCD} of instantons.
\newblock {\em Comm. Math. Phys.}, 252(1-3):359--391, 2004.

\bibitem[Pes12]{pestun2012localization}
Vasily Pestun.
\newblock Localization of gauge theory on a four-sphere and supersymmetric {W}ilson loops.
\newblock {\em Comm. Math. Phys.}, 313(1):71--129, 2012.

\bibitem[Pes18]{peskin2018introduction}
Michael~E Peskin.
\newblock {\em An introduction to quantum field theory}.
\newblock CRC press, 2018.

\bibitem[PS24]{stanfill2024sturmliouville}
Mateusz Piorkowski and Jonathan Stanfill.
\newblock Finer limit circle/limit point classification for {Sturm}-{Liouville} operators.
\newblock Preprint, {arXiv}:2407.04847 [math.{SP}] (2024), 2024.

\bibitem[RS75]{reed1975ii}
Michael Reed and Barry Simon.
\newblock {\em II: Fourier analysis, self-adjointness}, volume~2.
\newblock Elsevier, 1975.

\bibitem[Sim08]{simon2008schrodinger}
Barry Simon.
\newblock Schrodinger operators with purely discrete spectrum.
\newblock {\em arXiv preprint arXiv:0810.3275}, 2008.

\bibitem[Tay11]{taylor1996partial}
Michael~E. Taylor.
\newblock {\em Partial differential equations {I}. {B}asic theory}, volume 115 of {\em Applied Mathematical Sciences}.
\newblock Springer, New York, second edition, 2011.

\bibitem[TZ98a]{TianZhangQUANTIZATION1998}
Youliang Tian and Weiping Zhang.
\newblock An analytic proof of the geometric quantization conjecture of {Guillemin}-{Sternberg}.
\newblock {\em Invent. Math.}, 132(2):229--259, 1998.

\bibitem[TZ98b]{tian1998holomorphic}
Youliang Tian and Weiping Zhang.
\newblock Holomorphic {M}orse inequalities in singular reduction.
\newblock {\em Math. Res. Lett.}, 5(3):345--352, 1998.

\bibitem[Wit81]{witten1981dynamical}
Edward Witten.
\newblock Dynamical breaking of supersymmetry.
\newblock {\em Nuclear Physics B}, 188(3):513--554, 1981.

\bibitem[Wit82a]{witten1982constraints}
Edward Witten.
\newblock Constraints on supersymmetry breaking.
\newblock {\em Nuclear Physics B}, 202(2):253--316, 1982.

\bibitem[Wit82b]{witten1982supersymmetry}
Edward Witten.
\newblock Supersymmetry and {M}orse theory.
\newblock {\em J. Differential Geometry}, 17(4):661--692 (1983), 1982.

\bibitem[Wit83]{witten1983fermion}
Edward Witten.
\newblock Fermion quantum numbers in kaluza-klein theory.
\newblock {\em et al: Modern Kaluza-Klein Theories}, pages 438--511, 1983.

\bibitem[Wit84]{witten1984holomorphic}
Edward Witten.
\newblock Holomorphic {M}orse inequalities.
\newblock In {\em Algebraic and differential topology---global differential geometry}, volume~70 of {\em Teubner-Texte Math.}, pages 318--333. Teubner, Leipzig, 1984.

\bibitem[Wri20]{AlexWright_Mirzakhani_2020}
Alex Wright.
\newblock A tour through {Mirzakhani}'s work on moduli spaces of {Riemann} surfaces.
\newblock {\em Bull. Am. Math. Soc., New Ser.}, 57(3):359--408, 2020.

\bibitem[Wu03]{wu2003instanton}
Siye Wu.
\newblock On the instanton complex of holomorphic {M}orse theory.
\newblock {\em Comm. Anal. Geom.}, 11(4):775--807, 2003.

\bibitem[WZ98]{wu1998equivariant}
Siye Wu and Weiping Zhang.
\newblock Equivariant holomorphic {M}orse inequalities. {III}. {N}on-isolated fixed points.
\newblock {\em Geom. Funct. Anal.}, 8(1):149--178, 1998.

\bibitem[Zet05]{zettl2005sturm}
Anton Zettl.
\newblock {\em Sturm-liouville theory}.
\newblock Number 121. American Mathematical Soc., 2005.

\bibitem[Zha90]{weiping1990note}
Weiping Zhang.
\newblock A note on equivariant eta invariants.
\newblock {\em Proc. Amer. Math. Soc.}, 108(4):1121--1129, 1990.

\bibitem[Zha01]{Zhanglectures}
Weiping Zhang.
\newblock {\em Lectures on {C}hern-{W}eil theory and {W}itten deformations}, volume~4 of {\em Nankai Tracts in Mathematics}.
\newblock World Scientific Publishing Co., Inc., River Edge, NJ, 2001.

\end{thebibliography}

\end{document}